\documentclass[11pt,a4paper]{amsart}

\usepackage[T1]{fontenc}
\usepackage{lmodern}
\usepackage{microtype}
\usepackage{mathtools,amssymb,mathrsfs}
\usepackage{enumitem}

\theoremstyle{plain}
\newtheorem{maintheorem}{Main Theorem}

\newtheorem{thm}{Theorem}[section]

\newtheorem{lem}[thm]{{Lemma}}
\newtheorem{pro}[thm]{Proposition}

\theoremstyle{definition}

\newtheorem{defi}[thm]{{Definition}}

\theoremstyle{remark}
\newtheorem{rmk}[thm]{Remark}

\numberwithin{equation}{section}

\usepackage[margin=27mm]{geometry}
\usepackage{xcolor}
\usepackage{tikz}
\usetikzlibrary{arrows.meta}
\usepackage{xurl}
\usepackage{hyperref}
\usepackage{bookmark}
\hypersetup{
colorlinks=true,
linkcolor=blue,
citecolor=blue,
urlcolor=blue,
bookmarksnumbered,
pdftitle={Central Limit Theorems for Common Zeros via Bergman Kernels},
pdfauthor={Bin Guo},
pdfkeywords={Bergman kernel, central limit theorem, finite-point conditioning, common zeros, Wiener chaos}
}

\allowdisplaybreaks[2]

\makeatletter
\renewcommand{\paragraph}{\@startsection{paragraph}{4}
  {\z@}{-.9\baselineskip\@plus-.2\baselineskip\@minus-.1\baselineskip}
  {.35\baselineskip\@plus.1\baselineskip}
  {\normalfont\normalsize\bfseries}}
\makeatother

\begin{document}

	\title[Central Limit Theorems for Common Zeros via Bergman Kernels]{
		Central Limit Theorems\\
        for Common Zeros via Bergman Kernels
}


\author{Bin Guo}

\address{Academy of Mathematics and Systems Sciences, Chinese Academy of Sciences \\
and University of Chinese Academy of Sciences, Beijing 100190, China}

\email{guobin181@mails.ucas.ac.cn}

\date{November 2026}

\begin{abstract}
We prove central limit theorems as \(N\to\infty\) for common-zero
statistics of \(k\) independent standard Gaussian holomorphic sections of
\(L^N\), where \(L\) is a positive Hermitian holomorphic line bundle over
a compact K\"ahler manifold of complex dimension \(m\).
For fixed real \(\mathscr C^3\) forms \(\varphi\) of bidegree
\((m-k,m-k)\) with \(\partial\bar\partial\varphi\not\equiv0\), we prove
the CLT in codimensions \(2\leqslant k\leqslant m\) and show that the
integrated coefficient in the Shiffman--Zelditch variance formula at
scale \(N^{2k-m-2}\) is strictly positive.
Shiffman and Zelditch proved the unconditioned smooth CLT for \(k=1\).
For \(1\leqslant k\leqslant m\) we also prove a CLT for weighted volumes
in fixed domains \(U\) with \(\mathscr C^2\) boundary and fixed real
\(\mathscr C^3\) weights \(f\) whose restriction to \(\partial U\) is
not identically zero.  The variance of a weighted volume is asymptotic to
\(N^{2k-m-1/2}\nu_{m,k}\int_{\partial U}f^2\,\mathrm d\sigma\), with
\(\nu_{m,k}>0\) the Shiffman--Zelditch boundary constant and \(\sigma\)
the induced surface measure.  For \(k=m\) and \(f\equiv1\), the statistic
counts common zeros in \(U\).
Conditioning each section to vanish on the same fixed finite set preserves
the leading variance and the Gaussian limit for both classes of tests,
including smooth tests for \(k=1\).  The set may meet \(\partial U\).
A Bergman kernel criterion yields variance bounds and the new
unconditioned CLTs from a uniform local \(L^2\) Bargmann--Fock limit and
integrable off-diagonal derivative bounds.
\end{abstract}

\maketitle

\section{Introduction}\label{sec:setting}

Shiffman and Zelditch~\cite[Section~5, Problems~(1)--(2)]{MR2742043}
asked whether central limit theorems hold for smooth statistics of
common zeros in higher codimension and for full-codimension zero counts
in a domain.  We answer both questions for independent standard Gaussian
holomorphic sections of high powers of a positive line bundle over a
compact K\"ahler manifold of complex dimension \(m\).  We prove the smooth
CLT for \(2\leqslant k\leqslant m\) and the CLT for weighted volumes in
fixed domains with \(\mathscr C^2\) boundary for \(1\leqslant k\leqslant m\).
With \(k=m\) and unit weight, the statistic counts common zeros in the
domain.  The integrated variance coefficient is strictly positive for
admissible smooth tests, and the weighted-domain coefficient is expressed
through the
Shiffman--Zelditch boundary constant
(Proposition~\ref{prop:weighted-domain-variance-asymptotics}).
Conditioning all sections to vanish on the same fixed finite set
preserves the leading variance and the Gaussian limit for both classes
of tests.

Let \((L,h)\) be a holomorphic line bundle with a smooth positive
Hermitian metric over a compact complex manifold \(M\) of complex
dimension \(m\), and set \(\omega:=\pi c_1(L,h)\); positivity of \(h\)
makes \(\omega\) a K\"ahler form.  For each \(N\in\mathbb N\), the tensor power
\(L^N=L^{\otimes N}\) carries the induced metric \(h_N:=h^{\otimes N}\).
We equip the finite-dimensional space \(H^0(M,L^N)\) of global
holomorphic sections with the \(L^2\) inner product
\[
 \langle s_1,s_2\rangle
 :=\int_M h_N(s_1,s_2)\,\frac{\omega^m}{m!},
 \qquad s_1,s_2\in H^0(M,L^N).
\]
Let \(d_N:=\dim H^0(M,L^N)\) and choose an orthonormal basis
\(\{S_1^N,\ldots,S_{d_N}^N\}\) of \(H^0(M,L^N)\).  A standard Gaussian
section has the form
\[
 s^N=\sum_{j=1}^{d_N}\zeta_j S_j^N,
\]
where the \(\zeta_j\) are independent complex Gaussian variables with the
standard law \(\mathcal N_{\mathbb C}(0,1)\), normalized by
\(\mathbb E|\zeta_j|^2=1\).  Because the joint law of the coefficients is
unitarily invariant, the law of \(s^N\) does not depend on the chosen
orthonormal basis.
Write \(\mathbb E\) and \(\operatorname{Var}\) for the corresponding
expectation and variance.

For \(1\leqslant k\leqslant m\), let \(s_1^N,\ldots,s_k^N\) be
independent sections with this law.  Their common-zero set is
\[
 Z_{s_1^N,\ldots,s_k^N}
 :=\{z\in M:s_1^N(z)=\cdots=s_k^N(z)=0\}.
\]
For all sufficiently large \(N\), this set is almost surely a proper
intersection of the zero divisors and has complex dimension \(m-k\)
(Lemma~\ref{lem:unconditioned-intersection-facts}).  Write
\([Z_{s_1^N,\ldots,s_k^N}]\) for its integration current with analytic
intersection multiplicities.  For \(k=1\) this is a divisor current,
and for \(k=m\) it is a counting measure.

The diagonal Bergman density
\[
 \mathsf b_N(z):=\sum_{j=1}^{d_N}|S_j^N(z)|_{h_N}^2
\]
determines the expected divisor current.  For sufficiently large \(N\),
\(L^N\) is base-point free and \(\mathsf b_N>0\) on \(M\).  Shiffman and
Zelditch~\cite[author version, Lemma~3.1]{MR1675133} wrote the expected
divisor current as the pullback of the Fubini--Study form.  In terms of
\(\mathsf b_N\), their formula reads
\begin{equation}\label{eq:C0}
 \mathcal{C}^N_0(z):=\mathbb{E}\bigl[[Z_{s^N}]\bigr]
 =\frac{\sqrt{-1}}{2\pi}\,\partial\bar{\partial}\log \mathsf b_N(z)
  +c_1(L^N,h_N).
\end{equation}
The joint law of the independent standard Gaussian sections is
unitarily invariant, so their span has the invariant Grassmannian
distribution.  Shiffman and Zelditch
\cite[author version, Lemma~4.3]{MR1675133} therefore identify the
expected common-zero current with \((\mathcal C_0^N)^{\wedge k}\).
In full codimension, Edelman and Kostlan~\cite[Theorem~8.1]{MR1290398}
had already found the corresponding local formula for the expected
Gaussian zero count.  Zelditch's diagonal Bergman kernel expansion
\cite[author version, Theorem~1]{MR1616718} then gives
\[
 \frac{1}{N^k}\mathbb E\bigl[[Z_{s_1^N,\ldots,s_k^N}]\bigr]
 \longrightarrow c_1(L,h)^{\wedge k}=\frac{\omega^k}{\pi^k}
 \qquad\text{in }\mathscr C^\infty(M).
\]

Pairing the common-zero current with a fixed test form gives a scalar
statistic whose variance measures its fluctuations.  We use two classes
of tests.

\begin{defi}
\label{def:admissible-zero-statistics}
An admissible test form \(\varphi\) has bidegree \((m-k,m-k)\) on \(M\)
and is of one of two types:
\begin{enumerate}
    \item[\textbf{(S)}:] \(\varphi\) is a real form with
          \(\mathscr C^3\) coefficients satisfying
          \(\partial\bar\partial\varphi\not\equiv0\).
    \item[\textbf{(D)}:]
          \(\displaystyle
          \varphi=\varphi_U:=f\chi_U\frac{\omega^{m-k}}{(m-k)!},
          \)
          where \(f\in\mathscr C^3(M,\mathbb R)\) is fixed and
          \(U\subsetneq M\) is a fixed open domain whose boundary
          is locally an embedded \(\mathscr C^2\) hypersurface with \(U\)
          on one side.  Here \(\chi_U\) is the characteristic
          function of \(U\), with the convention \(\chi_U=0\) on \(\partial U\), and
          \(\sigma_{\partial U}\) is the induced surface measure.
          The weight satisfies
          \(\int_{\partial U}f^2\,\mathrm d\sigma_{\partial U}>0\).
\end{enumerate}
\end{defi}

For any fixed real \(\mathscr C^3\) form \(\varphi\) of bidegree
\((m-k,m-k)\), or a domain test \(\varphi=\varphi_U\) as above, set
\[
 X_N^\varphi:=\int_{Z_{s_1^N,\ldots,s_k^N}}\varphi.
\]
If \(\varphi\) is a real \(\mathscr C^3\) form with
\(\partial\bar\partial\varphi=0\), successive applications of
Poincar\'e--Lelong and integration by parts against the remaining closed
intersection currents give
\(X_N^\varphi=\int_M c_1(L^N,h_N)^{\wedge k}\wedge\varphi\) almost surely.

In type~\textup{(D)} with \(f\equiv1\), \(X_N^{\varphi_U}\) is the
\((2m-2k)\)-dimensional volume of the common-zero set in \(U\); when
\(k=m\), it counts the common zeros in \(U\).  The smooth convergence of the
normalized expected current gives
\[
 \frac{1}{N^k}\mathbb E X_N^\varphi
 \longrightarrow\frac{1}{\pi^k}\int_M\omega^k\wedge\varphi.
\]

Shiffman and Zelditch established the unconditioned variance asymptotics
for smooth tests~\cite[Theorem~1.1]{MR2742043} and for unweighted domain
tests~\cite[Theorem~1.4]{MR2465693}.

\begin{thm}
\label{thm:SZ-variance-asymptotics}
For \(1\leqslant k\leqslant m\) and every fixed \(\varepsilon>0\), the
following hold as \(N\to\infty\).
\begin{enumerate}[label=\textup{(\roman*)}]
\item For every fixed real \(\mathscr C^3\) form \(\varphi\) of bidegree
\((m-k,m-k)\),
\[
 \operatorname{Var}X_N^\varphi
 =N^{2k-m-2}\left[
 \int_M B_{m,k}(\partial\bar\partial\varphi,
                 \partial\bar\partial\varphi)\,\frac{\omega^m}{m!}
 +O(N^{-\frac12+\varepsilon})\right].
\]
Here \(B_{m,k}\) is a Hermitian form on the bundle of
\((m-k+1,m-k+1)\)-forms whose coefficients in unitary coframes depend
only on \(m,k\).
\item For a domain test \(\varphi_U\) of type~\textup{(D)} with
\(f\equiv1\),
\[
 \operatorname{Var}X_N^{\varphi_U}
 =N^{2k-m-\frac12}\left[
 \nu_{m,k}\int_{\partial U}\mathrm d\sigma_{\partial U}
 +O(N^{-\frac12+\varepsilon})\right],
\]
where \(\nu_{m,k}>0\) depends only on \(m,k\).
\end{enumerate}
The implied constants may depend on the fixed test, \(\varepsilon\), and
the background geometry, including \(U\) in \textup{(ii)}, but not on \(N\).
\end{thm}

For a single divisor and \(\mathscr C^\infty\) test forms, Shiffman
\cite{MR4293941} derived a full asymptotic expansion of the variance.
For \(k>1\), Shiffman and Zelditch~\cite[p.~1148]{MR2742043} conjectured
that \(B_{m,k}\) is pointwise positive definite.  The variance asymptotic
gives nonnegativity of the integrated coefficient.

\begin{thm}[{\cite[Theorem~1.2]{MR2742043}}]
\label{thm:SZ-smooth-divisor-clt}
In the unconditioned ensemble with \(k=1\), for every fixed real
\(\mathscr C^3\) form \(\varphi\) of bidegree \((m-1,m-1)\) satisfying
\(\partial\bar\partial\varphi\not\equiv0\),
\[
 \frac{X_N^\varphi-\mathbb EX_N^\varphi}
 {\sqrt{\operatorname{Var}X_N^\varphi}}
 \xrightarrow{\ d\ }\mathcal N_{\mathbb R}(0,1).
\]
\end{thm}

\paragraph{Finite-point conditioning}
Conditioning on finitely many prescribed zeros subtracts a finite-rank term
from the covariance kernel.  Following Shiffman, Zelditch and
Zhong~\cite{MR2806467}, we draw sections from the standard Gaussian law on
\(H^0(M,L^N)\) conditioned on vanishing at each point of a fixed finite set
\(\mathcal P\subset M\), the same for all \(k\) sections, possibly empty and
independent of \(N\).
For all sufficiently large \(N\), the evaluation map at \(\mathcal P\)
is surjective, with kernel
\begin{equation}\label{eq:finite-conditioning-subspace}
 \mathcal H_{N,\mathcal P}
 :=\{s\in H^0(M,L^N):s(p)=0\text{ for every }p\in\mathcal P\}.
\end{equation}
Gaussian conditioning identifies the resulting law with the standard
Gaussian law on \(\mathcal H_{N,\mathcal P}\) for the restricted
\(L^2\) inner product~\cite[Proposition~3.6 and Definition~3.7]{MR2806467}.
When \(\mathcal P=\varnothing\), the subspace \(\mathcal H_{N,\varnothing}\)
equals \(H^0(M,L^N)\) and the conditional law is the unconditioned law.
Write \(s_N^{\mathcal P}\) for a section drawn from the conditional law
and \(\mathbb E_{\mathcal P}\) and \(\operatorname{Var}_{\mathcal P}\)
for the corresponding expectation and variance.

Let \(s_{1,N}^{\mathcal P},\ldots,s_{k,N}^{\mathcal P}\) be
independent sections drawn from this law, with \(s_j^N=s_{j,N}^{\varnothing}\)
in the unconditioned case.  For all sufficiently large \(N\), their
common-zero set \(Z_{s_{1,N}^{\mathcal P},\ldots,s_{k,N}^{\mathcal P}}\)
is almost surely a proper intersection of the zero divisors, of complex
dimension \(m-k\) (Lemma~\ref{lem:conditional-intersection-facts}).  As in the
unconditioned case, brackets denote the integration current with
analytic intersection multiplicities.

For the same test forms, set
\[
 X_{N,\mathcal P}^{\varphi}
 :=\int_{Z_{s_{1,N}^{\mathcal P},\ldots,s_{k,N}^{\mathcal P}}}\varphi.
\]
For \(\mathcal P=\varnothing\), \(X_{N,\varnothing}^{\varphi}=X_N^\varphi\).

The main theorems include the unconditioned case \(\mathcal P=\varnothing\)
and extend the smooth codimension-one CLT of
Theorem~\ref{thm:SZ-smooth-divisor-clt} to finite-point conditioning.

\begin{maintheorem}
\label{thm:finite-point-conditioned-variance}
For every fixed finite set \(\mathcal P\subset M\) of distinct points,
possibly empty and independent of \(N\), every \(1\leqslant k\leqslant m\),
and every fixed admissible test form \(\varphi\), the following hold as
\(N\to\infty\).
\begin{enumerate}[label=\textup{(\roman*)}]
\item For a smooth test \(\varphi\) of type~\textup{(S)},
\[
 \operatorname{Var}_{\mathcal P}X_{N,\mathcal P}^{\varphi}
 =N^{2k-m-2}\left[
 \int_M B_{m,k}(\partial\bar\partial\varphi,
                 \partial\bar\partial\varphi)\,\frac{\omega^m}{m!}
 +o(1)\right],
\]
and the integral in the leading coefficient is strictly positive.
\item For a domain test \(\varphi=\varphi_U\) of type~\textup{(D)} with weight \(f\),
\[
 \operatorname{Var}_{\mathcal P}X_{N,\mathcal P}^{\varphi_U}
 =N^{2k-m-\frac12}\left[
 \nu_{m,k}\int_{\partial U}f^2\,\mathrm d\sigma_{\partial U}
 +o(1)\right].
\]
\end{enumerate}
Here \(B_{m,k}\) and \(\nu_{m,k}>0\) are those of
Theorem~\ref{thm:SZ-variance-asymptotics}.
\end{maintheorem}

\begin{maintheorem}
\label{thm:finite-point-conditioned-clt}
Under the hypotheses of
Main Theorem~\ref{thm:finite-point-conditioned-variance},
\begin{equation}\label{eq:conditional-statistic-clt}
 \frac{X_{N,\mathcal P}^{\varphi}
       -\mathbb E_{\mathcal P}X_{N,\mathcal P}^{\varphi}}
 {\sqrt{\operatorname{Var}_{\mathcal P}
                 X_{N,\mathcal P}^{\varphi}}}
 \xrightarrow{\ d\ }\mathcal N_{\mathbb R}(0,1).
\end{equation}
\end{maintheorem}

For a smooth test \(\varphi\) of type~\textup{(S)} with
\(2\leqslant k\leqslant m\), let
\(\eta=\sqrt{-1}\,\partial\bar\partial\varphi\) and let \(l_*\) be the least
\(l\in\{1,\dots,k\}\) for which \(\omega^l\wedge\eta\equiv0\).
Set \(t=k-l_*\) and \(\mathrm dV=\omega^m/m!\).
Orthogonality and the limits in
Theorem~\ref{thm:SZ-variance-asymptotics}\textup{(i)}
and~\eqref{eq:intrinsic-smooth-component-limit} give
\begin{equation}\label{eq:smooth-variance-coefficient-lower-bound}
\begin{aligned}
 &\int_M B_{m,k}(\partial\bar\partial\varphi,
                 \partial\bar\partial\varphi)\,\mathrm dV\\
 &\qquad\geqslant
 \frac{\pi^{m-2k}(t!)^2}{4(t+1)^{m+2}}
 \|\omega^{l_*-1}\wedge\eta\|_{L^2(\mathrm dV)}^2>0.
\end{aligned}
\end{equation}
The left-hand side is the leading coefficient in
Main Theorem~\ref{thm:finite-point-conditioned-variance}\textup{(i)}.
The norm on complex exterior forms is the Hermitian norm induced by
the real Riemannian metric.

The weighted-domain variance and CLT, including under finite-point
conditioning, extend to piecewise \(\mathscr C^2\) boundaries satisfying
the geometric conditions of Remark~\ref{rmk:piecewise-boundary-extension}.
These include the polyhedral local models used for the unweighted variance
asymptotic of Shiffman and Zelditch~\cite[Theorem~1.4]{MR2465693}.

\paragraph{Conditional means}
Shiffman, Zelditch and Zhong~\cite[Theorems~1.2 and~5.1]{MR2806467}
determined the local scaling limits of the expected zero currents near
a single prescribed point.  They also computed the leading correction
to the smooth mean for one divisor conditioned at a finite set of
points~\cite[Theorem~1.1]{MR2806467}.  The normalized expected common-zero
current has the following global limit in every codimension.

\begin{thm}
\label{thm:conditional-mean-equidistribution}
For every fixed finite set \(\mathcal P\subset M\) of distinct points,
possibly empty and independent of \(N\), and every \(1\leqslant k\leqslant m\),
\[
 \frac1{N^k}\mathbb E_{\mathcal P}
 \bigl[[Z_{s_{1,N}^{\mathcal P},\ldots,s_{k,N}^{\mathcal P}}]\bigr]
 \longrightarrow c_1(L,h)^{\wedge k}=\frac{\omega^k}{\pi^k}.
\]
The convergence is weak in the sense of currents on \(M\) and
\(\mathscr C^\infty\) on compact subsets of \(M\setminus\mathcal P\).
For \(\mathcal P=\varnothing\), the convergence is smooth on all of \(M\).
\end{thm}

The proof in Section~\ref{sec:finite-conditioning} combines smooth
convergence away from the prescribed points with equality of the trace
masses of the normalized expected currents and the limiting form.

When \(k=m\), each prescribed point is a simple zero almost surely
(Lemma~\ref{lem:conditional-intersection-facts}), so the normalized
conditional expected measure has an atom of mass \(N^{-m}\) there.

Theorem~\ref{thm:conditional-mean-equidistribution} gives the same limit
for \(N^{-k}\mathbb E_{\mathcal P}X_{N,\mathcal P}^\varphi\) as for the
unconditioned statistic.  For a domain test, wedging the normalized
expected currents with \(\omega^{m-k}/(m-k)!\) gives positive trace measures
that converge weakly to \(\omega^m/[\pi^k(m-k)!]\).  This limiting measure
assigns zero mass to \(\partial U\), so the integrals of the bounded
function \(f\chi_U\) against the trace measures converge.

For either class of admissible tests, the conditional and unconditioned
means differ by \(o(\sqrt{\operatorname{Var}X_N^\varphi})\).  This estimate
holds for domain tests even when \(\mathcal P\) meets \(\partial U\).
The variance ratio tends to one, and the conditional CLT holds with
the unconditioned mean, standard deviation, or both
(Proposition~\ref{prop:quantitative-conditional-comparison}).
Proposition~\ref{prop:conditional-mean-asymptotics} computes the coefficient
of \(N^{k-m-1}\) in the mean shift for smooth tests and for domain tests
with \(\mathcal P\cap\partial U=\varnothing\).  In the domain case, only
the prescribed points in \(U\) enter this coefficient.

\paragraph{Proof strategy}
For one unconditioned divisor, Poincar\'e--Lelong and integration by parts
turn a centred smooth statistic into an integral of the centred
logarithmic norm against the deterministic form
\((\sqrt{-1}/\pi)\partial\bar\partial\varphi\).  Shiffman and Zelditch
\cite[pp.~1163--1164]{MR2742043} applied the correlation-kernel criterion
of Sodin and Tsirelson~\cite[Theorem~2.2]{MR2121537} to this integral.
That criterion concerns integrals of radial functions of normalized
complex Gaussian processes.  Here the radial Wiener--chaos expansion and
Wick's formula~\cite[author version, Sections~2.2--2.3]{MR2121537}
reduce the moment estimates to integrals of products of the normalized
correlation kernel; off-diagonal Szeg\H{o} decay then controls the
contribution away from the diagonal, while the near-diagonal Gaussian
scaling supplies the positive lower bound relative to the kernel row
integral.

In higher codimension, Poincar\'e--Lelong applied to one section pairs its
logarithmic norm with a test built from the random intersection current
of the others.  This pairing lies outside the direct scope of the
Sodin--Tsirelson criterion.  Independence and deterministic intersection
masses nevertheless let us truncate the radial Wiener--chaos expansion
in each section separately, which yields orthogonal projections of the
common-zero statistic in \(L^2\).

For the unconditioned ensemble,
Theorem~\ref{thm:two-estimate-BF-criterion} gives two-sided variance
bounds and CLTs for smooth statistics in higher codimension and
weighted-domain statistics in every codimension.  The criterion uses a
uniform local \(L^2\) Bargmann--Fock limit and integrable off-diagonal
derivative bounds for the normalized Bergman kernel.
Section~\ref{sec:gaussian-sections-verification} verifies both hypotheses
for the positive line-bundle ensemble.

At a fixed cutoff, integration by parts and Wick's formula express the
moments as sums over diagrams, with at most two derivatives at each
kernel endpoint.  The integrable off-diagonal bounds on these
derivatives control the integrations along a spanning tree of each
connected component, following Sodin and Tsirelson
\cite[author version, Section~2.4]{MR2121537}.  With a matching lower
bound for the cutoff variance, the standardized contributions of components with at least
three vertices tend to zero, and two-vertex components give the
Gaussian pairings.

The same off-diagonal bounds localize the variance of a selected joint
chaos component near the diagonal.  Holomorphicity and the derivative
bounds turn the local \(L^2\) Bargmann--Fock limit into the derivative
convergence needed to compute its leading coefficient.  For smooth tests
this coefficient is a positive multiple of the squared \(L^2\)-norm
of the wedge of \(\sqrt{-1}\,\partial\bar\partial\varphi\) with the
highest power of \(\omega\) for which the product is nonzero.
In the three-curve product of Section~\ref{sec:completes-the-comparison},
\(\omega\wedge\eta=0\).  The components with block orders \((2,0)\)
and \((0,2)\) have variances \(o(N^{-1})\) by~\eqref{eq:W-l-over-N}
with \(l=2\) and symmetry of the sections.  The \((2,2)\) component
has variance asymptotic to a positive multiple of \(N^{-1}\)
by~\eqref{eq:intrinsic-smooth-component-limit}.
For domain tests the selected coefficient is a positive multiple of
\(\int_{\partial U}f^2\,\mathrm d\sigma_{\partial U}\).
Every cutoff that retains the second-order chaos of each section
retains the selected component.  By orthogonality, its variance bounds
both the truncated and full variances from below.  Abel summation and
bounds on the sets where the correlation modulus is close to one control
the truncation error uniformly in \(N\) at that variance scale.
This transfers the fixed-cutoff CLT to the full statistic.

For weighted domain tests, localizing the covariance at the boundary
and freezing the weight at one endpoint give a variance coefficient
at fixed cutoff equal to a universal constant times
\(\int_{\partial U}f^2\,\mathrm d\sigma_{\partial U}\).  The uniform truncation
estimate gives the full variance asymptotic once the limit of these
constants is known.  Comparing at \(f\equiv1\) with the Shiffman--Zelditch
formula identifies that limit as \(\nu_{m,k}\).

Finally, orthogonal decomposition couples the conditional and
unconditioned sections.  For smooth tests, we estimate the logarithmic
norm difference in the global Poincar\'e--Lelong telescoping identity.
For domain tests, we combine decay of the finite-rank covariance
correction away from the prescribed points with second-moment bounds
for the intersection masses near them.  The difference of the coupled
statistics is \(o(\sqrt{\operatorname{Var}X_N^\varphi})\) in \(L^2\).
The reverse triangle inequality for the centred statistics shows that
the variance ratio tends to one, and Slutsky's theorem transfers the CLT.

\paragraph{Related results}
Bloch and P\'olya~\cite{MR1576817} and Kac~\cite{MR7812} studied real
roots, and Hammersley~\cite{MR84888} studied complex zeros.  In quantum
chaos, Bogomolny, Bohigas and Leboeuf~\cite{BogomolnyBohigasLeboeuf1992}
studied random polynomial models, including one whose expected zero
density is uniform on the sphere.  Their later work~\cite{MR1418808}
discusses zero correlations.  Nonnenmacher and
Voros~\cite{MR1649013} analyzed zeros of Bargmann representations of
chaotic eigenstates on the two-torus.  For positive line bundles over
compact K\"ahler manifolds, Shiffman and Zelditch
\cite[author version, Theorem~1.1]{MR1675133} proved almost-sure
convergence of normalized random divisor currents to the first Chern
form.

% Share the local page stretch between related-work paragraphs.
\vspace{0pt plus 1pt}
Dinh and Sibony~\cite[Corollary~7.4]{MR2208805} obtained almost-sure
equidistribution of common zeros in arbitrary codimension on projective
manifolds, with quantitative bounds on deviation probabilities.
Their proof uses meromorphic transforms and estimates for
quasi-plurisubharmonic functions.  For zero divisors of square-integrable
sections on complete noncompact Hermitian manifolds, Dinh, Marinescu and
Schmidt~\cite[Theorems~1.2 and~1.5]{MR2950760} proved equidistribution,
with convergence estimates on compact subsets under curvature and
dimension-growth conditions.

\vspace{0pt plus 1pt}
Singular Hermitian metrics were treated by Coman and
Marinescu~\cite[Theorems~1.1--1.2]{MR3377051}.  Under semipositive
curvature, their geometric hypotheses, and a locally uniform logarithmic
asymptotic condition on the Bergman density off a compact analytic set
\(\Sigma\), with the metric continuous on its complement, they proved
weak convergence of normalized Fubini--Study currents.  When \(\operatorname{codim}\Sigma\geqslant k\),
this extends to \(k\)-th powers and, on compact manifolds, to normalized
expected common-zero currents of \(k\) sections.  Dinh, Ma and
Marinescu~\cite[Theorem~1.2]{MR3554701} later obtained convergence rates
for zeros of sections of semipositive singular Hermitian line bundles
under additional positivity hypotheses.  Bayraktar, Coman and
Marinescu~\cite[Theorem~1.1]{MR4105509} proved equidistribution of zeros
of random sections of sequences of singular Hermitian line bundles
on compact normal K\"ahler spaces under curvature and logarithmic
moment hypotheses.  The survey of Shiffman and
Zelditch~\cite{MR4748494} places these developments in stochastic
K\"ahler geometry.  These results describe the limiting spatial
distribution of the zeros.

The Sodin--Tsirelson criterion has since been used for smooth
single-divisor statistics in several settings, including weighted random
polynomials~\cite{MR3637941}, sequences of positive line
bundles~\cite[Theorem~1.2]{MR4861154} (as corrected
in~\cite{BojnikGunyuzCorrection2025}), noncompact Hermitian
manifolds~\cite{MR4748124}, Toeplitz-generated
sections~\cite[Theorem~1.17]{DrewitzLiuMarinescu2024}, and semipositive
bundles on punctured Riemann
surfaces~\cite[Theorem~1.5.2]{LiuZielinski2025}.  For holomorphic
systems, Bojnik and G\"uny\"uz~\cite[Theorems~1.1--1.2]{BojnikGunyuzSystems2026}
give variance upper bounds for smooth linear statistics and prove
equidistribution of normalized common-zero currents of random sections
of sequences of positive line bundles under geometric and coefficient-law
hypotheses.

For the planar Gaussian entire function, Nazarov and Sodin
\cite[Theorem~1.2 and Corollary~1.3]{MR2863379} use almost-independence
to prove CLTs for nonzero real compactly supported \(C^\alpha\) tests,
\(\alpha>0\), under a standard-deviation lower bound that is automatic
for \(\alpha>1\).  Their clustering and cumulant method
\cite[Theorem~1.5]{MR2885614} also covers bounded real compactly
supported measurable tests, provided the variance grows at least as a
positive power of the dilation parameter.  For indicators of bounded
measurable sets of positive area, the variance lower
bound~\cite[Lemma~1.6]{MR2885614} yields the CLT for zero counts in their
dilates, the planar analogue of the counting statistic studied here.

Real Gaussian models have also been analyzed.  Armentano, Aza\"is,
Dalmao and Le\'on~\cite{ArmentanoAzaisDalmaoLeon2022} prove a CLT for
common-zero volumes of rectangular Kostlan--Shub--Smale systems, while
Ancona and Letendre~\cite{AnconaLetendre2021} prove one for the root
counting measure of Kostlan polynomials.  In the general zero-set case,
the CLT of Ancona, Gass, Letendre and
Stecconi~\cite[Theorem~1.14]{AnconaGassLetendreStecconi2025} requires
\(\mathrm{ND}(q)\) for every integer \(q\geqslant0\).  Already
\(\mathrm{ND}(1)\) entails joint nondegeneracy of the limiting field's
value and first real derivatives at a point.  In a local holomorphic
frame, the Cauchy--Riemann equations impose linear relations on this
joint vector, so \(\mathrm{ND}(1)\) fails for the holomorphic model.

\paragraph{Organization}
Section~\ref{sec:kernel-criterion} states the Bergman kernel criterion.
The chaos projections are built in
Section~\ref{sec:projection-foundation}, and
Section~\ref{sec:feynman-formalism} proves the fixed-cutoff Gaussian
limit.  Section~\ref{sec:criterion-proof} removes the cutoff and proves
the correlation-volume criterion.  Section~\ref{sec:two-point-nondegeneracy}
derives the pointwise, integral and sublevel-set bounds from \textup{(BF)}
and \textup{(OD)} and proves the required variance lower bounds.  Applying
Proposition~\ref{thm:correlation-volume-criterion} then proves the Bergman
kernel criterion of Theorem~\ref{thm:two-estimate-BF-criterion}.
Section~\ref{sec:gaussian-sections-verification} verifies \textup{(BF)}
and \textup{(OD)} for the positive line-bundle ensemble and proves the
weighted-domain variance asymptotics.
Section~\ref{sec:finite-conditioning} proves the conditional mean results
and transfers the variance asymptotics and CLT to the conditional
ensemble.  Appendices~\ref{app:BF-complementary-indices}
and~\ref{app:piecewise-boundary-extension} contain the complementary-index
identities and the piecewise \(\mathscr C^2\) extension.

\section{A Bergman kernel criterion}\label{sec:kernel-criterion}

We consider the unconditioned ensemble of independent standard Gaussian
sections in \(H^0(M,L^N)\).  A uniform local \(L^2\) Bargmann--Fock limit
and integrable off-diagonal derivative bounds for the normalized Bergman
kernel yield the two-sided variance bounds and the CLTs in
Theorem~\ref{thm:two-estimate-BF-criterion}.
Section~\ref{sec:finite-conditioning} treats finite-point conditioning.

In local holomorphic frames, let \(B_N(z,w)\) be the Bergman kernel
coefficient and, for all sufficiently large \(N\), set
\[
 R_N(z,w):=\frac{B_N(z,w)}{\sqrt{B_N(z,z)B_N(w,w)}}.
\]
The diagonal coefficients are computed in the respective endpoint
frames.  Write \(\operatorname{dist}\) for the geodesic distance induced
by \(\omega\).

For the local limit, choose holomorphic coordinates \(z\) centred at
\(x\), unitary for \(\omega_x\), and a preferred holomorphic frame
\(e_x\) of \(L\) such that
\[
 -\log|e_x(z)|_h^2=|z|^2+O(|z|^3).
\]
These choices depend smoothly on \(x\) on each member of a finite
cover, with uniform bounds for every fixed derivative order.  Let
\(R_{N,x}\) denote the normalized kernel in this same frame at both
endpoints.  Its rescaled form is
\[
 \mathcal R_{N,x}(u,v)
 :=R_{N,x}\!\left(\frac{u}{\sqrt N},\frac{v}{\sqrt N}\right).
\]
The scaled balls and the Bargmann--Fock kernel are
\[
 D_T:=\{u\in\mathbb C^m:|u|<T\},\qquad
 \mathcal R_{\mathrm{BF}}(u,v)
 :=\exp\!\left(u\cdot\bar v-\frac{|u|^2+|v|^2}{2}\right).
\]
The common-centre coordinates and frame remain fixed on
\(D_T\times D_T\); the \(L^2\)-norm below uses Lebesgue measure.

For the off-diagonal bounds, use a fixed finite family of holomorphic
charts \(U_a\) with subcharts \(V_a\Subset U_a\) covering \(M\), as in
Section~\ref{sec:Bergman-setting}.  For \(p\in V_a\) and \(q\in V_b\),
choose holomorphic frames \(e_{a,p}\) on \(U_a\) and \(e_{b,q}\) on
\(U_b\), centred separately at the endpoints.  The normalization at
\(p\) is
\[
 |e_{a,p}(p)|_h^2=1,\qquad
 \mathrm d\bigl(\log|e_{a,p}|_h^2\bigr)_p=0,
\]
with the analogous conditions at \(q\).  These frames are independent of
\(N\), depend smoothly on their centres, and have metric coefficients with
uniform bounds for every fixed derivative order.  For \(\alpha=(\alpha',\alpha'')\in
\mathbb Z_{\geqslant0}^{m}\times\mathbb Z_{\geqslant0}^{m}\), set
\[
D_z^\alpha:=\prod_{a=1}^m
 \left(\frac{\partial}{\partial z_a}\right)^{\alpha'_a}
 \left(\frac{\partial}{\partial\bar z_a}\right)^{\alpha''_a},
\qquad |\alpha|:=|\alpha'|+|\alpha''|,
\]
and use the analogous notation \(D_w^\beta\) in the second variable.
Keep the centres, frames and charts fixed while differentiating, and
then evaluate at \((p,q)\).  Lemma~\ref{lem:Bergman-horizontal-comparison}
constructs these frames and compares the scaled derivative bounds for
different families satisfying these conditions.

\begin{thm}
\label{thm:two-estimate-BF-criterion}
For this unconditioned Gaussian ensemble, assume:
\begin{enumerate}
\item[\textup{(BF)}] For every fixed \(T>0\),
\[
 \sup_{x\in M}
 \|\mathcal R_{N,x}-\mathcal R_{\mathrm{BF}}\|_{L^2(D_T\times D_T)}
 \longrightarrow0.
\]
\item[\textup{(OD)}] There exist \(A>2m\) and \(C>0\) such that, for all
sufficiently large \(N\), all chart pairs \(a,b\), all \(p\in V_a\) and
\(q\in V_b\), and all \(|\alpha|,|\beta|\leqslant2\),
\[
 N^{-(|\alpha|+|\beta|)/2}
 \bigl|(D_z^\alpha D_w^\beta R_N)(p,q)\bigr|
 \leqslant C(1+\sqrt N\,\operatorname{dist}(p,q))^{-A}.
\]
\end{enumerate}
Then for every fixed admissible test form \(\varphi\) of
Definition~\ref{def:admissible-zero-statistics}, with
\(2\leqslant k\leqslant m\) in type~\textup{(S)} and
\(1\leqslant k\leqslant m\) in type~\textup{(D)}, the statistic
\(X_N^\varphi\) satisfies
\[
 \operatorname{Var}X_N^\varphi\asymp
 \begin{cases}
  N^{2k-m-2},&\textup{smooth statistics},\\[2pt]
  N^{2k-m-\frac12},&\textup{domain statistics},
 \end{cases}
 \qquad
 \frac{X_N^\varphi-\mathbb EX_N^\varphi}
 {\sqrt{\operatorname{Var}X_N^\varphi}}
 \xrightarrow{\ d\ }\mathcal N_{\mathbb R}(0,1).
\]
Here \(\asymp\) means two-sided bounds with constants that may depend on the
fixed test, the domain and weight, and the background geometry, but not on
\(N\).
\end{thm}

The exponent \(A>2m\) in~\textup{(OD)} makes the decay profile integrable
after rescaling distances by \(\sqrt N\).  Integrating the bound in its
second variable gives \(O(N^{-m})\) against K\"ahler volume on \(M\)
and \(O(N^{-m+1/2})\) against surface measure on \(\partial U\), uniformly
in the first variable.  Lemma~\ref{lem:BF-OD-consequences} gives the
integral and rescaled tail estimates.  After Stokes' theorem, the
boundary main term retains one derivative of the scalar covariance
factor at each endpoint.
Together the two derivatives contribute a factor \(N\) relative to the
smooth pairing.

Section~\ref{sec:projection-foundation} constructs the orthogonal
projections.  Section~\ref{sec:feynman-formalism} uses connected diagrams
to prove the Gaussian limit at a fixed cutoff given a matching variance
lower bound.  In Section~\ref{sec:criterion-proof}, the uniform tail
bound and a lower bound for the full variance give the cutoff variance
lower bound for every fixed common cutoff \(q\geqslant q_0\), with
\(q_0\) independent of \(N\).  This allows passage to the full statistic
by taking \(N\to\infty\) and then \(q\to\infty\).

Section~\ref{sec:two-point-nondegeneracy} derives the correlation-volume
estimates from \textup{(BF)} and \textup{(OD)} and obtains a variance
lower bound for both the full and truncated statistics from a joint
chaos component retained by every choice of cutoffs \(q_j\geqslant1\).  Section~\ref{sec:gaussian-sections-verification}
verifies \textup{(BF)} and \textup{(OD)} for the positive line-bundle
ensemble and identifies the leading coefficient of the weighted-domain
variance.  For admissible smooth tests in higher codimension, the lower
bound and the asymptotic formula of
Theorem~\ref{thm:SZ-variance-asymptotics}\textup{(i)} give the strict
positivity asserted in
Main Theorem~\ref{thm:finite-point-conditioned-variance}\textup{(i)}.
The unconditioned smooth CLT for \(k=1\) is
Theorem~\ref{thm:SZ-smooth-divisor-clt}.

\section{Orthogonal Wiener-chaos projections}\label{sec:projection-foundation}

We expand each divisor current in the Wiener chaoses of its Gaussian
coefficient vector.  For common-zero statistics, we identify the
orthogonal projections obtained by truncating the divisor factors
separately.  These projections underlie the moment and tail estimates
in Sections~\ref{sec:feynman-formalism}--\ref{sec:criterion-proof}.

\subsection{Gaussian sections and expected currents}

In a nonvanishing local frame \(e_L\), our curvature convention is
\begin{equation*}
c_1(L,h)=-\frac{\sqrt{-1}}{\pi}\partial\bar\partial\log|e_L|_h.
\end{equation*}
Unless specified otherwise, tensor norms, covariant derivatives
\(\nabla\), and Sobolev norms use the fixed K\"ahler metric and its
product metrics.
We use the orthonormal basis of \(H^0(M,L^N)\) and the positive diagonal
density \(\mathsf b_N\) from the Introduction.  The normalized section
\[
\widetilde s^N(z):=\frac{s^N(z)}{\sqrt{\mathsf b_N(z)}}
\]
has standard complex Gaussian marginals.  Hermitian products are linear
in the first variable.
For each \(N\), all Gaussian coefficient vectors and their couplings
are realized on a sufficiently large product Gaussian probability space
\((\Omega_N,\mathbb P_N)\).  We suppress the index \(N\).

The mean divisor current \(\mathcal C_0^N\) is given by~\eqref{eq:C0}.

Fix \(1\leqslant k\leqslant m\).
For the unconditioned sections \(s_1^N,\dots,s_k^N\), the statistics
of the Introduction have the current representation
\[
X_N^\varphi
=\int_M[Z_{s_1^N}]\wedge\cdots\wedge[Z_{s_k^N}]\wedge\varphi.
\]
Wedges of divisor currents refer to their almost-sure proper
intersections, with the usual analytic intersection multiplicities.

We first recall properness and mean-current identities without
asymptotic kernel estimates.

\begin{lem}
\label{lem:unconditioned-intersection-facts}
For all sufficiently large \(N\) and every \(0\leqslant h\leqslant k\),
almost surely every
\(h\)-element subcollection of the unconditioned sections has a proper
common-zero current.  The expected current and total mass are
\begin{equation}\label{eq:foundational-intersection-facts}
\mathbb E\bigl[[Z_{s_1^N,\ldots,s_h^N}]\bigr]
 =(\mathcal C_0^N)^{\wedge h},
\qquad
\int_M[Z_{s_1^N,\ldots,s_h^N}]
 \wedge\frac{\omega^{m-h}}{(m-h)!}
 =\frac{N^h}{\pi^h(m-h)!}\int_M\omega^m .
\end{equation}
Here the empty-intersection conventions are
\([Z_\varnothing]=[M]\) and \((\mathcal C_0^N)^{\wedge0}=1\).
If a fixed deterministic Borel set \(E\subset M\) has zero ambient
volume, then the positive trace measure
\([Z_{s_1^N,\ldots,s_h^N}]\wedge\omega^{m-h}/(m-h)!\)
assigns zero mass to \(E\) almost surely.
\end{lem}

\begin{proof}
For large \(N\), \(L^N\) is base-point free, so Bertini's theorem
gives smooth transverse intersections almost surely
\cite[Section~2.1]{MR2465693}.  The mean-current identity follows
from \cite[Proposition~2.2]{MR2465693}, applied to \(H^0(M,L^N)\).
The mass identity is \cite[(18)]{MR2465693}, with
\(c_1(L^N,h_N)=N\omega/\pi\); the case \(h=0\) is immediate.
The expected trace measure
\((\mathcal C_0^N)^{\wedge h}\wedge\omega^{m-h}/(m-h)!\)
vanishes on \(E\); positivity and Tonelli imply the almost-sure claim.
\end{proof}

\subsection{Wiener-chaos expansion of a zero current}

Recall the normalized section, now written in the orthonormal basis:
\begin{equation*}
\widetilde s^N=\frac{s^N}{\sqrt{\mathsf b_N}}
=\sum_{j=1}^{d_N}\zeta_j\frac{S_j^N}{\sqrt{\mathsf b_N}},
\qquad \zeta_j\sim\mathcal N_{\mathbb C}(0,1)\ \text{i.i.d.}
\end{equation*}
On a trivializing set, write \(S_j^N=f_j^Ne_{L^N}\) and
\(F_j^N=f_j^N/(\sum_\ell|f_\ell^N|^2)^{1/2}\), so that
\begin{equation}\label{eq:local-Gaussian}
\widetilde s^N(z)=\xi(z)\frac{e_{L^N}}{|e_{L^N}|_{h_N}},
\qquad
\xi(z)=\sum_j\zeta_jF_j^N(z),\qquad \sum_j|F_j^N(z)|^2=1.
\end{equation}
Every finite-dimensional marginal of \(\xi(z)\) is thus centred circularly
symmetric complex Gaussian, possibly degenerate, with unit marginal variances
and \(\mathbb E[\xi(z^a)\xi(z^b)]=0\).
Poincar\'e--Lelong gives
\begin{equation}\label{eq:fluctuation-deterministic-split}
[Z_{s^N}]
=\frac{\sqrt{-1}}{\pi}\partial\bar\partial
 \log|\widetilde s^N|_{h_N}+\mathbb E[[Z_{s^N}]].
\end{equation}

For a standard complex Gaussian \(\xi\), the radial Wick polynomials are
\[
:|\xi|^{2\alpha}:
=\alpha!\sum_{j=0}^{\alpha}\binom{\alpha}{j}
 \frac{(-1)^{j+\alpha}}{j!}|\xi|^{2j}.
\]
See \cite[Chapter~3]{MR1474726} for background on Wick products.
The expansion in \cite[Section~2.2, (2.5)]{MR2121537} is
\begin{equation}\label{eq:decomposition-in-probability}
\log|\xi|=\sum_{\alpha=0}^{\infty}
\frac{c_{2\alpha}}{\alpha!}:|\xi|^{2\alpha}:
\quad\text{in }L^2,
\end{equation}
where \(c_0=-\gamma_{\mathrm E}/2\)
(\(\gamma_{\mathrm E}\) is Euler's constant) and
\(c_{2\alpha}=(-1)^{\alpha+1}/(2\alpha)\) for \(\alpha\geqslant1\)
\cite[Section~2.1, Lemma~2.1]{MR2863379}.  This rotation-invariant scalar chaos expansion
\cite[Chapters~1--3]{MR1474726} applies at each \(\xi(z)\).
We use the homogeneous Wiener chaoses with respect to the coefficient
vector \(\boldsymbol\zeta=(\zeta_1,\dots,\zeta_{d_N})\).

\begin{defi}
\label{def:homogeneous-Wiener-chaos}
For a standard complex Gaussian vector
\(\boldsymbol\eta=(\eta_1,\dots,\eta_d)\) with independent components
and \(r\in\mathbb Z_{\geqslant0}\), let
\[
\mathscr P_{\leqslant r}(\boldsymbol\eta)
:=\operatorname{span}\{\boldsymbol\eta^{\mathbf p}
\overline{\boldsymbol\eta}^{\mathbf q}:
\mathbf p,\mathbf q\in\mathbb Z_{\geqslant0}^d,\ 
|\mathbf p|+|\mathbf q|\leqslant r\},
\qquad \mathscr P_{\leqslant-1}:=\{0\}.
\]
Here \(\boldsymbol\eta^{\mathbf p}=\prod_a\eta_a^{p_a}\) and
\(|\mathbf p|=\sum_a p_a\).  The homogeneous Wiener chaoses
\[
\mathscr H_r(\boldsymbol\eta)
:=\mathscr P_{\leqslant r}(\boldsymbol\eta)
  \cap\mathscr P_{\leqslant r-1}(\boldsymbol\eta)^\perp
\]
form the orthogonal Gaussian \(L^2\)-decomposition.
\end{defi}

The scalar membership
\(:|\eta|^{2\alpha}:\in\mathscr H_{2\alpha}(\eta)\) and
\eqref{eq:local-Gaussian} give
\(:|\widetilde s^N(z)|_{h_N}^{2\alpha}:
\in\mathscr H_{2\alpha}(\boldsymbol\zeta)\) for each \(z\).
\begin{defi}\label{def:chaos-current}
Let \(s^N\in H^0(M,L^N)\) be a Gaussian random section.
\begin{itemize}
    \item For \(\alpha\geqslant 1\), the component of \([Z_{s^N}]\) in Wiener
          chaos of order \(2\alpha\) is the random smooth \((1,1)\)-form
          \begin{equation}\label{eq:defi-Ca}
          \mathcal{C}^N_\alpha(z)
          :=\frac{\sqrt{-1}}{\pi}\,\frac{c_{2\alpha}}{\alpha!}\,
             \partial\bar{\partial}\Bigl(:\!|\widetilde{s}^N(z)|_{h_N}^{2\alpha}\!:\Bigr),
          \end{equation}

    \item The order-zero component is the deterministic current
          \(\mathcal C_0^N=\mathbb E[[Z_{s^N}]]\) defined in~\eqref{eq:C0}.
\end{itemize}
\end{defi}

Spatial differentiation preserves the Wiener-chaos order with respect
to \(\boldsymbol\zeta\),
so \(\langle\mathcal C_\alpha^N,\phi\rangle\in
\mathscr H_{2\alpha}(\boldsymbol\zeta)\) for every deterministic
\(\mathscr C^2\) test form \(\phi\).  Combining
\eqref{eq:decomposition-in-probability} with
\eqref{eq:fluctuation-deterministic-split} gives
\begin{equation*}
[Z_{s^N}]=\sum_{\alpha=0}^{\infty}\mathcal{C}^N_\alpha,
\end{equation*}
meaning convergence of the mutually orthogonal pairings in \(L^2(\Omega)\)
for each deterministic \(\mathscr C^2\) test \((m-1,m-1)\)-form.
For intersections, the remaining divisor factors are random currents.
A domain test is a bounded form whose coefficients may jump across
\(\partial U\).  Proposition~\ref{prop:blockwise-chaos-projection}
establishes the corresponding \(L^2\) pairings and identifies their
projections.

\subsection{Projection identities for intersections}
\label{sec:projection-construction}

Write \(s_j^N=\sum_{a=1}^{d_N}\zeta_{j,a}S_a^N\), with independent
coefficient vectors
\(\boldsymbol\zeta_j=(\zeta_{j,1},\ldots,\zeta_{j,d_N})\),
\(1\leqslant j\leqslant k\).  Write \(\Omega_j\) for the probability
factor carrying \(\boldsymbol\zeta_j\).

\begin{defi}
\label{def:blockwise-chaos-projections}
For \(1\leqslant j\leqslant k\) and
\(\alpha\in\mathbb Z_{\geqslant0}\), let
\(\mathsf P_{j,=\alpha}\) be the orthogonal projection in \(L^2(\Omega)\)
onto \(\mathscr H_{2\alpha}(\boldsymbol\zeta_j)\) with respect to
\(\boldsymbol\zeta_j\), acting as the identity on all other probability
factors.  The corresponding projection with a finite cutoff is
\[
\mathsf P_{j,\leqslant q}
:=\sum_{\alpha=0}^{q}\mathsf P_{j,=\alpha},
\qquad q\in\mathbb Z_{\geqslant0}.
\]
Here \(\mathsf P_{j,=0}=\mathbb E_j\), where \(\mathbb E_j\) integrates
only over the factor \(\Omega_j\), leaving all other probability factors
fixed.
Projections acting on distinct coefficient vectors commute.
For \(\boldsymbol\alpha\in\mathbb Z_{\geqslant0}^k\) and
\(\mathbf q=(q_1,\dots,q_k)\in
(\mathbb Z_{\geqslant0}\cup\{\infty\})^k\), set
\[
\mathsf P_{=\boldsymbol\alpha}
:=\prod_{j=1}^k\mathsf P_{j,=\alpha_j},
\qquad
\mathsf P_{\leqslant\mathbf q}
:=\prod_{j:q_j<\infty}\mathsf P_{j,\leqslant q_j}.
\]
An empty product is the identity.
\end{defi}

For \(q\in\mathbb Z_{\geqslant0}\), write
\begin{equation*}
[Z_{s_j^N}]^{\leqslant q}:=\sum_{\alpha=0}^{q}\mathcal C_\alpha^{N,j},
\qquad
[Z_{s_j^N}]^{\leqslant \infty}:=[Z_{s_j^N}].
\end{equation*}
Here \(\mathcal C_\alpha^{N,j}\) is the component of
\([Z_{s_j^N}]\) in Wiener chaos of order \(2\alpha\).
For finite \(q\), the truncation is a random smooth \((1,1)\)-form;
\(q=\infty\) gives the divisor current.

\begin{pro}
\label{prop:blockwise-chaos-projection}
Let \(N\) be sufficiently large and let \(\varphi\) be a real
\(\mathscr C^3\) form of bidegree \((m-k,m-k)\) or a test form of
type~\textup{(D)}.  Then \(X_N^\varphi\in L^2(\Omega)\), and for every
\(\mathbf q\in(\mathbb Z_{\geqslant0}\cup\{\infty\})^k\),
\begin{equation}\label{eq:mixed-statistic-projection}
\mathsf P_{\leqslant\mathbf q}X_N^\varphi
=\int_M\bigwedge_{j=1}^k[Z_{s_j^N}]^{\leqslant q_j}\wedge\varphi
\quad\text{in }L^2(\Omega).
\end{equation}
\end{pro}

\begin{proof}[Proof of Proposition~\ref{prop:blockwise-chaos-projection}]
First consider the smooth test class in the statement.

\smallskip\noindent\emph{Integrability of mixed products.}
For \(\mathbf q\in(\mathbb Z_{\geqslant0}\cup\{\infty\})^k\), set
\[
Y_{\mathbf q}:=
\int_M\bigwedge_{i=1}^k[Z_{s_i^N}]^{\leqslant q_i}\wedge\varphi.
\]
Let \(I_\infty=\{i:q_i=\infty\}\), \(p=|I_\infty|\), and let
\(V\) be the intersection of these untruncated divisors, with \(V=M\)
if \(p=0\).  Its trace measure
\(\mu_V=[V]\wedge\omega^{m-p}/(m-p)!\) has deterministic mass
by~\eqref{eq:foundational-intersection-facts}.  Restriction to
\(V_{\mathrm{reg}}\), with intersection multiplicities, and expansion
into minors give
\begin{equation}\label{eq:comass-domination-on-Y}
|Y_{\mathbf q}|
\leqslant C_\varphi
\prod_{i\notin I_\infty}
\|[Z_{s_i^N}]^{\leqslant q_i}\|_{\mathscr C^0,*}\,\mu_V(M).
\end{equation}
Here \(\|\cdot\|_{\mathscr C^0,*}\) is the supremum of the
coefficient-matrix operator norm in \(\omega\)-unitary frames.
The total variation of the measure defining \(Y_{\mathbf q}\) satisfies
the same bound on every Borel set, with \(\mu_V(M)\) replaced by the
\(\mu_V\)-measure of that set.
For fixed \(N\) and finite cutoffs, these norms are bounded by
polynomials in the Gaussian coefficients, so every \(Y_{\mathbf q}\)
belongs to \(L^2(\Omega)\), including
\(Y_{(\infty,\ldots,\infty)}=X_N^\varphi\).

\smallskip\noindent\emph{Projection in one coefficient vector.}
Relabel the sections so that the factor to be projected is \(s_1^N\),
the other untruncated factors are \(s_2^N,\ldots,s_\ell^N\), and
\(q_i<\infty\) for \(\ell<i\leqslant k\), with
\(1\leqslant\ell\leqslant k\).
With the empty-intersection and empty-product conventions, set
\[
\Theta:=[Z_{s_2^N,\ldots,s_\ell^N}]
 \wedge\bigwedge_{i=\ell+1}^k[Z_{s_i^N}]^{\leqslant q_i}.
\]
For this configuration, write
\[
Y:=\int_M[Z_{s_1^N}]\wedge\Theta\wedge\varphi,
\qquad
Y_q:=\int_M[Z_{s_1^N}]^{\leqslant q}\wedge\Theta\wedge\varphi,
\quad q\in\mathbb Z_{\geqslant0}.
\]
Condition on the coefficient vectors \(\boldsymbol\zeta_2,\ldots,
\boldsymbol\zeta_k\), and let \(\|\cdot\|_{2,1}\) denote the
\(L^2\)-norm over \(\Omega_1\).
The current \(\Theta\) is closed and independent of \(s_1^N\).
Its mass \(\|\Theta\|_{\mathrm{mass}}\), taken with respect to \(\omega\),
satisfies \(\mathbb E\|\Theta\|_{\mathrm{mass}}^2<\infty\) by the
mass and polynomial bounds above.
The signed measure
\(\mu_{\Theta,\varphi}:=(\sqrt{-1}/\pi)
\Theta\wedge\partial\bar\partial\varphi\) satisfies
\(\|\mu_{\Theta,\varphi}\|_{\mathrm{TV}}
\leqslant C_\varphi\|\Theta\|_{\mathrm{mass}}\).
Base-point-freeness and independence imply that \(s_1^N\) is almost
surely not identically zero on any component of
\(Z_{s_2^N,\ldots,s_\ell^N}\).
The standard complex Gaussian marginal law and Tonelli's theorem give
\[
\mathbb E_1\int_M\bigl|\log|\widetilde s_1^N|_{h_N}\bigr|
 \,d|\mu_{\Theta,\varphi}|
=\mathbb E|\log|\xi||\,
 \|\mu_{\Theta,\varphi}\|_{\mathrm{TV}}<\infty,
\qquad \xi\sim\mathcal N_{\mathbb C}(0,1),
\]
and also give logarithmic integrability against
\([Z_{s_2^N,\ldots,s_\ell^N}]\wedge\omega^{m-\ell+1}/(m-\ell+1)!\),
the positive trace measure of the untruncated intersection.
Apply Stokes to
\(u_\varepsilon=\tfrac12\log(|\widetilde s_1^N|_{h_N}^2+\varepsilon^2)\).
For \(x>0\) and \(0<\varepsilon\leqslant1\), the bound
\(\bigl|\tfrac12\log(x^2+\varepsilon^2)\bigr|
\leqslant|\log x|+\tfrac12\log2\) gives dominated convergence
against both measures.  Closedness and Poincar\'e--Lelong on the
proper-intersection components therefore yield
\[
\langle\Theta\wedge
 ([Z_{s_1^N}]-\mathcal C_0^{N,1}),\varphi\rangle
=\int_M\log|\widetilde s_1^N|_{h_N}\,d\mu_{\Theta,\varphi}.
\]
Closedness also gives \(\mu_{\Theta,\varphi}(M)=0\).

\smallskip\noindent\emph{Finite chaos expansion.}
With \(e_\alpha(\xi):=\alpha!^{-1}:|\xi|^{2\alpha}:\), set
\(g^{[q]}=-\gamma_{\mathrm E}/2+
\sum_{\alpha=1}^q c_{2\alpha}e_\alpha\).
Radiality makes its evaluation on \(\widetilde s_1^N\) independent
of the unit-norm frame.  Subtracting this finite expansion from the logarithmic pairing gives
\[
Y-Y_q
=\int_M\bigl(\log|\widetilde s_1^N|_{h_N}
-g^{[q]}(\widetilde s_1^N)\bigr)\,d\mu_{\Theta,\varphi}.
\]
Minkowski's inequality gives
\begin{equation}\label{eq:mixed-chaos-tail-estimate}
\|Y-Y_q\|_{2,1}
\leqslant\|\mu_{\Theta,\varphi}\|_{\mathrm{TV}}
 \|\log|\xi|-g^{[q]}(\xi)\|_2
=\frac{\|\mu_{\Theta,\varphi}\|_{\mathrm{TV}}}{2}
 \left(\sum_{\alpha>q}\alpha^{-2}\right)^{1/2}.
\end{equation}
Averaging the squared conditional bound gives \(Y_Q\to Y\) in \(L^2\) as
\(Q\to\infty\).  For \(Q\geqslant q\), the term containing
\(\mathcal C_\alpha^{N,1}\) in \(Y_Q\) has chaos order \(2\alpha\) in
\(\boldsymbol\zeta_1\), since \(\Theta\) is independent of this vector.
Thus \(\mathsf P_{1,\leqslant q}Y_Q=Y_q\).  Passing to the limit by
contractivity of the projection gives
\begin{equation}\label{eq:partial-chaos-projection}
\mathsf P_{1,\leqslant q}Y=Y_q.
\end{equation}
Undoing the relabelling gives the same identity in any untruncated
coefficient vector.  Entries already finite are handled by their
finite chaos expansions.
Thus every smooth mixed statistic is the \(L^2\) limit of its finite
chaos projections in each coefficient vector.

\smallskip\noindent\emph{Domain tests.}
For \(\varphi=\varphi_U\), choose
\(h_\varepsilon\in\mathscr C^\infty(M)\) with
\(0\leqslant h_\varepsilon\leqslant1\) and
\(h_\varepsilon\to\chi_U\) off \(\partial U\).
Let \(Y_{\mathbf q}^{(\varepsilon)}\) be the mixed integral with
\(\varphi_{U,\varepsilon}=f h_\varepsilon\omega^{m-k}/(m-k)!\).  For each fixed \(\mathbf q\), the total-variation
domination by \(\mu_V\) and
\(\mu_V(\partial U)=0\) almost surely, by
Lemma~\ref{lem:unconditioned-intersection-facts}, give
\(Y_{\mathbf q}^{(\varepsilon)}\to Y_{\mathbf q}\) almost surely.
Both integrals are bounded by
\[
C\|f\|_\infty\prod_{i\notin I_\infty}
 \|[Z_{s_i^N}]^{\leqslant q_i}\|_{\mathscr C^0,*}\,\mu_V(M),
\]
which belongs to \(L^2\) by the mass and polynomial bounds above.
Dominated convergence thus gives convergence in \(L^2\) for each
fixed cutoff, and projection contractivity passes
\eqref{eq:partial-chaos-projection} to \(\varphi_U\) as
\(\varepsilon\downarrow0\).  To obtain convergence as the cutoff
tends to infinity in one coefficient vector, first fix a smooth
approximation.  Contractivity bounds its approximation error uniformly
in that cutoff; let the cutoff tend to infinity in the smooth
statistic, and then let \(\varepsilon\downarrow0\).

\smallskip\noindent\emph{Successive projections.}
Finally, apply~\eqref{eq:partial-chaos-projection} successively in each
coefficient vector, relabelling at each step, to obtain
\eqref{eq:mixed-statistic-projection}.
Taking all \(q_j=0\) in~\eqref{eq:mixed-statistic-projection} gives
the common mean with each divisor factor replaced by \(\mathcal C_0^N\).
Every mixed cutoff preserves this mean.
For \(\boldsymbol\alpha\in\mathbb Z_{\geqslant0}^k\), successive finite
differences at positive indices and constant projections at zero
indices give the joint \(\boldsymbol\alpha\)-projection by replacing
the \(j\)-th divisor factor with \(\mathcal C_{\alpha_j}^{N,j}\).
On a mixed cutoff, this component is retained when
\(\alpha_j\leqslant q_j\) for every \(j\) with \(q_j<\infty\), and is
zero otherwise.
\end{proof}

\subsection{Cutoff errors and chaos orders}
\label{sec:cutoff-errors}

For centred statistics, we bound the squared \(L^2\) error of a common
cutoff by the sum of the squared errors from truncating one coefficient
vector at a time.  Section~\ref{sec:criterion-proof} estimates the
single-vector errors uniformly in \(N\) at the variance scale.
We use a hat to denote the centred statistic and its projections:
\begin{equation*}
\widehat X_N^\varphi:=X_N^\varphi-\mathbb E X_N^\varphi,
\end{equation*}
\begin{equation}\label{eq:general-mixed-statistic}
\widehat X_{N,\mathbf q}^\varphi:=\mathsf P_{\leqslant\mathbf q}\widehat X_N^\varphi,
\qquad
\widehat X_{N,q}^\varphi:=\widehat X_{N,(q,\dots,q)}^\varphi.
\end{equation}
For the centred statistics, the projection onto constants is zero.
The other joint chaos projections follow from
\eqref{eq:mixed-statistic-projection} by taking finite differences,
subject to the same cutoff restrictions.
The cutoff \(q=1\) retains orders zero and two in each coefficient vector.
For \(k>1\), the centred projection can therefore contain the total
Wiener-chaos orders \(2,4,\dots,2k\).

For fixed \(N\), every mixed cutoff satisfies, for
\(1\leqslant j\leqslant k\),
\begin{equation}\label{eq:blockwise-projection-convergence}
\widehat X_{N,\mathbf q}^\varphi
=\sum_{\alpha=0}^{\infty}
\mathsf P_{j,=\alpha}\widehat X_{N,\mathbf q}^\varphi
\quad\text{in }L^2(\Omega).
\end{equation}
The same expansion holds before centring.
Iterating the commuting projections gives
\(\widehat X_{N,q}^\varphi\to\widehat X_N^\varphi\) in \(L^2(\Omega)\)
as \(q\to\infty\).

For \(1\leqslant j\leqslant k\) and \(q\in\mathbb Z_{\geqslant0}\), set
\[
\mathcal T_{N,q}^{\varphi,(j)}
:=(I-\mathsf P_{j,\leqslant q})\widehat X_N^\varphi.
\]

\smallskip\noindent\emph{Common-cutoff error.}
The common-cutoff error satisfies
\begin{equation}\label{eq:common-cutoff-error}
\|\widehat X_N^\varphi-\widehat X_{N,q}^\varphi\|_2^2
\leqslant\sum_{j=1}^k\|\mathcal T_{N,q}^{\varphi,(j)}\|_2^2
=k\|\mathcal T_{N,q}^{\varphi,(1)}\|_2^2.
\end{equation}
For the error comparison, use the telescoping identity
\[
\widehat X_N^\varphi-\widehat X_{N,q}^\varphi
=\sum_{j=1}^k
 \left(\prod_{i<j}\mathsf P_{i,\leqslant q}\right)
 (I-\mathsf P_{j,\leqslant q})\widehat X_N^\varphi.
\]
The summands are pairwise orthogonal: for \(i<j\), the \(i\)-th lies
in the kernel of \(\mathsf P_{i,\leqslant q}\), while the \(j\)-th
lies in its range.  Contractivity and permutation symmetry of the
independent sections give~\eqref{eq:common-cutoff-error}.
Since each cutoff projection preserves the mean, the comparison is
unchanged by centring.

\section{Gaussian limits at a fixed chaos cutoff}\label{sec:feynman-formalism}

At a fixed cutoff, Wick contractions express the moments as sums over
diagrams.  After integration by parts, horizontal kernel estimates on the
unit circle bundle control the connected terms by integration along
spanning trees.  With a matching variance lower bound, only pair
partitions contribute to the limiting normalized moments.

\subsection{Bergman kernel derivatives and lifted correlations}\label{sec:Bergman-setting}

We use the fixed finite holomorphic atlas and endpoint-centred frames
of~\textup{(OD)} in Theorem~\ref{thm:two-estimate-BF-criterion}.
Write \(U_a\) for the charts and choose open subcharts
\(V_a\Subset U_a\) covering \(M\).  Denote by \(e_{a,p}\),
\(p\in V_a\), the corresponding frame on \(U_a\).
Lemma~\ref{lem:Bergman-horizontal-comparison} constructs these frames
with the required normalization and uniform metric derivative bounds.

For \(p\in V_a\) and \(q\in V_b\), write
\(S_j^N=f_j^N e_{a,p}^{\otimes N}\) and
\(S_j^N=g_j^N e_{b,q}^{\otimes N}\).  In these frames, the Bergman
and normalized correlation kernels have the coefficient representations

\begin{equation}\label{eq:normalized-Bergman-kernel}
\begin{split}
B_N(z,w)&=\sum_{j=1}^{d_N}f_j^N(z)\overline{g_j^N(w)},\\
R_N(z,w)&=
\frac{B_N(z,w)}
 {\bigl(\sum_j|f_j^N(z)|^2\bigr)^{1/2}
  \bigl(\sum_j|g_j^N(w)|^2\bigr)^{1/2}}.
\end{split}
\end{equation}
For large \(N\) the denominator is positive, and \(R_N\) is the
correlation of the scalar representatives of \(\widetilde s^N\) in
the associated unit-norm frames.  Write its intrinsic squared modulus as
\begin{equation*}
\tau_N(p,q):=|R_N(p,q)|^2.
\end{equation*}
Then \(0\leqslant\tau_N\leqslant1\), and \(\tau_N\) is globally smooth
on \(M^2\), including at zeros of the Bergman kernel.  We differentiate
powers of the kernel modulus only in the smooth form
\(\tau_N^\alpha=|R_N|^{2\alpha}\).

For the fixed-centre derivatives in~\textup{(OD)}, the normalized
maximum is
\begin{equation}\label{eq:kernel-jet-envelope}
J_N(p,q):=\max_{\substack{|\alpha|\leqslant2\\|\beta|\leqslant2}}
N^{-(|\alpha|+|\beta|)/2}
\bigl|(D_z^\alpha D_w^\beta R_N)(p,q)\bigr|,
\end{equation}
where the maximum also runs over all chart pairs with
\(p\in V_a,q\in V_b\).
Thus \(J_N\) includes mixed derivatives of total order four.  It is
symmetric and Borel measurable, with \(J_N\geqslant|R_N|\).
For any two finite families of frames satisfying the normalization and
uniform derivative bounds above, the corresponding maxima are comparable
by Lemma~\ref{lem:Bergman-horizontal-comparison}.

Thus condition~\textup{(OD)} of Theorem~\ref{thm:two-estimate-BF-criterion}
is equivalently the bound
\(J_N(z,w)\leqslant C(1+\sqrt N\,\operatorname{dist}(z,w))^{-A}\).
We use \(B(x,r)\) for geodesic balls on \(M\) and for Euclidean balls
in the scaling coordinates.

Let \(\pi:X\subset L^*\to M\) be the unit circle bundle for the dual
metric \(h^*\) \cite{MR1794066}.  Evaluation of \(s^N\) gives the
equivariant lift:
\begin{equation}\label{CR-function}
 \hat s^N(\lambda)=\langle s^N(z),\lambda^{\otimes N}\rangle,
 \qquad \lambda\in L_z^*,\qquad
 \hat s^N(e^{i\theta}\lambda)=e^{iN\theta}\hat s^N(\lambda).
\end{equation}
For a local frame \(e_L\) with \(h(z)=|e_L(z)|_h^2\), use the coordinates
\begin{equation*}
 (z,e^{i\theta})\longmapsto
 e^{i\theta}\frac{e_L^*(z)}{|e_L^*(z)|_{h^*}}
 =e^{i\theta}h(z)^{1/2}e_L^*(z)\in X.
\end{equation*}
For \(s^N=f^Ne_L^{\otimes N}\), these give
\begin{align}
 \hat s^N(z,\theta)&=e^{iN\theta}h(z)^{N/2}f^N(z),
 \label{eq:local-expression}\\
 |\hat s^N(z,\theta)|&=|s^N(z)|_{h_N}.
 \label{eq:abs-of-CR-fun}
\end{align}
The lifted basis gives the Szeg\H{o} kernel
\begin{equation*}
 \Pi_N(x,y)=\sum_{j=1}^{d_N}\widehat S_j^N(x)
                  \overline{\widehat S_j^N(y)}.
\end{equation*}
The process \(\widehat{\widetilde s^N}(x)=\hat s^N(x)/\sqrt{\Pi_N(x,x)}\)
has circularly symmetric complex Gaussian finite vectors with unit
marginal variances and correlation
\begin{equation*}
 \rho_N(x,y):=\mathbb E\!\left[
 \widehat{\widetilde s^N}(x)\overline{\widehat{\widetilde s^N}(y)}\right]
 =\frac{\Pi_N(x,y)}{\sqrt{\Pi_N(x,x)\Pi_N(y,y)}}.
\end{equation*}
Thus \(|\rho_N|\leqslant1\) and \(\rho_N(x,x)=1\).

The horizontal lifts of the coordinate fields \(\partial/\partial z_a\),
with respect to the Chern connection, are
\begin{equation}\label{eq:horizontal-derivative}
 Z_a=\frac{\partial}{\partial z_a}-A_a\frac{\partial}{\partial\theta},
 \qquad A_a=\frac{\sqrt{-1}}2\frac{\partial\log h}{\partial z_a}.
\end{equation}
Their conjugates lift the antiholomorphic fields.  The lifts realize
covariant differentiation:
\(Z_a\hat s^N=\widehat{\nabla^N_{\partial/\partial z_a}s^N}\), where
\(\nabla^N\) is the induced Chern connection \cite{MR1794066}.
Let \(\mathcal H\) denote an ordered composition of the horizontal
fields \(Z_a\) and \(\bar Z_a\), acting on either factor of \(X^2\).
Write \(|\mathcal H|_j\) for the number of fields acting in the
\(j\)-th variable, \(j=1,2\), and set
\(|\mathcal H|=|\mathcal H|_1+|\mathcal H|_2\).
The empty composition denotes the identity operator.

In any fixed local holomorphic frames, normalization
as in~\eqref{eq:normalized-Bergman-kernel} cancels the metric factors
in~\eqref{eq:local-expression}:
\begin{equation}\label{eq:Bergman-Szego-bridge}
\rho_N((z,\theta),(w,\psi))
 =e^{iN(\theta-\psi)}R_N(z,w).
\end{equation}
In particular, \(|\rho_N(x,y)|=|R_N(\pi x,\pi y)|\) and
\(\rho_N\overline{\rho_N}=\tau_N\).
The analogous maximum for horizontal derivatives is
\begin{equation*}
J_N^{\mathrm H}(p,q):=
\max_{\substack{|\mathcal H|_1\leqslant2\\|\mathcal H|_2\leqslant2}}
N^{-|\mathcal H|/2}|\mathcal H\rho_N(x,y)|,
\quad p=\pi x,\ q=\pi y,
\end{equation*}
over the same eligible chart pairs as in~\eqref{eq:kernel-jet-envelope},
including the identity operator.  This maximum is symmetric and independent
of the fibre lifts.

\begin{lem}
\label{lem:Bergman-horizontal-comparison}
There is a constant \(C\), independent of \(N,p,q\), such that
\begin{equation}\label{eq:Bergman-horizontal-comparison}
C^{-1}J_N(p,q)\leqslant J_N^{\mathrm H}(p,q)
 \leqslant C J_N(p,q).
\end{equation}
The same comparison holds for any finite family of holomorphic frames
satisfying the same normalization and uniform derivative bounds.
\end{lem}

\begin{proof}
To construct the endpoint frames, fix a nonvanishing holomorphic frame
\(e_a\) on each \(U_a\).  With \(\phi_a=-\log|e_a|_h^2\), set
\begin{equation}\label{eq:centred-holomorphic-frame}
e_{a,p}(z)=\exp\!\left(\frac{\phi_a(p)}2+
 \sum_j\frac{\partial\phi_a}{\partial z_j}(p)\,[z_j-z_j(p)]\right)e_a(z).
\end{equation}
These frames depend smoothly on \(p\) and satisfy the required
normalization and uniform derivative bounds.  Let \(D_\mu\)
range over \(\partial/\partial z_a\) and \(\partial/\partial\bar z_a\),
with the circle coordinate held fixed.  Their horizontal lifts are
\(D_\mu^{\mathrm H}=D_\mu+a_\mu T\) and range over \(Z_a,\bar Z_a\);
here \(T=\partial_\theta\), \(a_\mu(p)=0\), and \(a_\mu\) is
angle-independent.
Only two derivatives at each endpoint are needed.  For any smooth
function \(f\) on \(X\) and \(x\in X\) with \(\pi(x)=p\),
\begin{equation*}
\begin{aligned}
 (D_\mu^{\mathrm H} f)(x)&=(D_\mu f)(x),\\
 (D_\mu^{\mathrm H} D_\nu^{\mathrm H} f)(x)
   &=(D_\mu D_\nu f)(x)+(D_\mu a_\nu)(p)(Tf)(x).
\end{aligned}
\end{equation*}
The coefficients \(D_\mu a_\nu\) are uniformly bounded.  Thus replacing
two horizontal derivatives by coordinate derivatives adds only a bounded
multiple of \(T\); solving the same identity gives the converse.
Apply these identities independently at the two endpoints of
\eqref{eq:Bergman-Szego-bridge}.  Vertical derivatives act by \(iN\)
at the first endpoint and \(-iN\) at the second, so each replacement
has the same scale as two coordinate derivatives.  For a composition with \(|\mathcal H|_1=a\leqslant2\) and
\(|\mathcal H|_2=b\leqslant2\), division by \(N^{(a+b)/2}\) gives both bounds
in~\eqref{eq:Bergman-horizontal-comparison}.  For changes between frames
satisfying the same normalization and uniform derivative bounds,
the first phase derivatives vanish
at the centres and the second phase derivatives are of order
\(O(N)\), absorbed by the factor \(N^{-1}\) for two derivatives at an
endpoint.  Coordinate changes have uniformly bounded coefficients.
Thus the same comparison holds for any such family of frames.
\end{proof}

On \(M^p\), write \(\partial_j,\bar\partial_j\) for the Dolbeault
operators in the \(j\)-th factor:
\begin{equation}\label{def:partial_j}
 \partial_j=\sum_{a=1}^m\mathrm dz_a^j\,\frac{\partial}{\partial z_a^j},
 \qquad
 \bar\partial_j=\sum_{a=1}^m\mathrm d\bar z_a^j\,
                 \frac{\partial}{\partial\bar z_a^j}.
\end{equation}

\begin{lem}
\label{lem:weighted-mean-current-bound}
On the diagonal \(\Delta=\{(z,z):z\in M\}\), the mean current satisfies
\begin{equation}\label{eq:weighted-mean-from-kernel}
\mathcal C_0^N
=-\frac{\sqrt{-1}}{2\pi}
 \left.\partial_1\bar\partial_1\tau_N\right|_\Delta.
\end{equation}
In particular, if \(\sup_{z\in S}J_N(z,z)\leqslant C\) on a set
\(S\subset M\), then \(\sup_{z\in S}\|\mathcal C_0^N(z)\|\leqslant C'N\).
\end{lem}

\begin{proof}
In a common holomorphic frame, write
\(B_N(z,w)=\sum_a f_a^N(z)\overline{f_a^N(w)}\).  The quotient
\[
\tau_N(z,w)=\frac{|B_N(z,w)|^2}
 {B_N(z,z)B_N(w,w)}
\]
satisfies \(\tau_N|_\Delta=1\) and has vanishing first derivatives
on \(\Delta\).  Differentiating there gives
\[
\left.\partial_1\bar\partial_1\tau_N\right|_\Delta
=-\partial\bar\partial\log B_N(z,z).
\]
The metric terms cancel in~\eqref{eq:C0}.  The product rule for
\(\tau_N=|R_N|^2\) and the bound on \(J_N(z,z)\) give the \(O(N)\) estimate.
\end{proof}

\subsection{Wick contractions and products of correlation kernels}\label{sec:feynman-diagrams}

For one section, Wick's formula evaluates the scalar expectation left
after extracting the Dolbeault operators and deterministic zeroth-chaos
factors from the expected exterior product
\[
\mathbb{E}\!\left[ \pi_1^* \mathcal{C}^N_{\alpha_1} \wedge \pi_2^* \mathcal{C}^N_{\alpha_2}
      \wedge \cdots \wedge \pi_p^* \mathcal{C}^N_{\alpha_p} \right],
\]
where \(\pi_j:M^{p}\to M\) denotes the projection onto the \(j\)-th factor.

Split the indices into
\begin{align}
I_+&:=\{i\in\{1,\dots,p\}:\alpha_i>0\},
\label{eq:I-plus}\\
I_0&:=\{i\in\{1,\dots,p\}:\alpha_i=0\}.
\label{eq:I-zero}
\end{align}
The deterministic factors \(\pi_b^*\mathcal C_0^N\), \(b\in I_0\),
can be taken outside the expectation.  First suppose that
\(I_0=\varnothing\) and \(I_+=\{1,\dots,p\}\).

For fixed \(N\) and indices, the Wick functions and their spatial
derivatives are finite polynomials in the Gaussian coefficients, with
smooth deterministic coefficients.  Differentiating under expectation
in~\eqref{eq:defi-Ca} gives
\begin{equation}\label{eq:correlation-expectation}
\mathbb{E}\!\left[\bigwedge_{j=1}^{p}\pi_j^*\mathcal{C}^N_{\alpha_j}\right]
= \frac{c_{2\alpha_{1}}\cdots c_{2\alpha_p}}{\alpha_{1}!\cdots\alpha_{p}!}
\Bigl(\frac{\sqrt{-1}}{\pi}\partial_{1}\bar\partial_{1}\Bigr)\cdots
\Bigl(\frac{\sqrt{-1}}{\pi}\partial_{p}\bar\partial_{p}\Bigr)
\mathbb{E}\!\left[\prod_{j=1}^{p}
:\!|\widetilde{s}^N(z^j)|_{h_N}^{2\alpha_j}\!:\right].
\end{equation}

Using the local representatives \(\xi(z^j)\) from~\eqref{eq:local-Gaussian},
we encode the expectation in~\eqref{eq:correlation-expectation} by perfect
matchings between unbarred and barred copies from distinct factors,
following~\cite[Section~2.3]{MR2121537} and allowing \(\alpha_i=0\).

\begin{defi}\label{def:Gamma}
For \(\boldsymbol\alpha=(\alpha_1,\dots,\alpha_p)\in\mathbb Z_{\geqslant0}^p\), let
\(\Gamma(\alpha_1,\dots,\alpha_p)\) consist of the perfect matchings
\(\gamma=(\mathsf V_{\boldsymbol\alpha},\mathsf E_\gamma)\) on the
distinct labelled copies
\[
 \mathsf V_{\boldsymbol\alpha}:=
 \bigsqcup_{i=1}^p\bigl(
 \{i^{(a)}:1\leqslant a\leqslant\alpha_i\}
 \sqcup\{\bar i^{(a)}:1\leqslant a\leqslant\alpha_i\}\bigr),
\]
whose edges have the form \(\{i^{(a)},\bar j^{(b)}\}\), with \(i\ne j\).
Each copy occurs in exactly one edge, and the superscripts distinguish
labelled copies.  For \(\boldsymbol\alpha=\boldsymbol0\), the set consists of the
single empty diagram \(\gamma_\varnothing^p=(\varnothing,\varnothing)\).
For a centred complex Gaussian vector \((\xi_1,\dots,\xi_p)\), set
\[
 v(\gamma):=\prod_{\{i^{(a)},\bar j^{(b)}\}\in\mathsf E_\gamma}
                  \mathbb E[\xi_i\overline{\xi_j}].
\]
Throughout the diagram formulas, empty sums are zero, empty scalar and
wedge products are one, and empty operator products are the identity.
\end{defi}

The complex diagram formula~\cite[author version, Section~2.3,
(2.9)]{MR2121537} applies: circular symmetry excludes contractions of
like types, and Wick ordering excludes contractions within one factor.

\begin{pro}\label{prop:wick-formula}
For any integers $\alpha_1,\dots ,\alpha_p\geqslant 0$ and any possibly degenerate
centred circularly symmetric complex Gaussian vector $(\xi_1,\dots,\xi_p)$ with
unit marginal variances, one has
\[
\mathbb{E}\!\left[\,\prod_{i=1}^{p} :\!|\xi_i|^{2\alpha_i}\!:\right]
   \;=\; \sum_{\gamma\in\Gamma(\alpha_1,\dots ,\alpha_p)} v(\gamma).
\]
\end{pro}

\begin{rmk}\label{rmk:on-Gamma(alpha)}
For $\alpha\geqslant1$, one has $\Gamma(\alpha)=\varnothing$.  For \(p=2\),
$\Gamma(\alpha_1,\alpha_2)=\varnothing$ whenever
$\alpha_1\neq\alpha_2$.  If $\alpha_1=\alpha_2=\alpha\geqslant1$, the
unbarred copies of each label are matched bijectively with the barred copies
of the other label.  The two bijections can be chosen independently, and hence
\[
|\Gamma(\alpha,\alpha)|=(\alpha!)^2.
\]
\end{rmk}

\begin{defi}
\label{def:value-function}
For \(\gamma\in\Gamma(\alpha_1,\dots,\alpha_p)\), the directed
multigraph \(\gamma^*\) has vertex set \(\{1,\dots,p\}\) and one edge
\(e\) for each edge joining \(i\) to \(\bar j\) in \(\gamma\), with
source \(\mathfrak s(e)=i\) and target \(\mathfrak t(e)=j\).
Parallel edges remain distinct.  The outdegree and indegree are
\[
\deg^+(v)=|\mathfrak s^{-1}(v)|,\qquad
\deg^-(v)=|\mathfrak t^{-1}(v)|.
\]
The empty diagram gives no edges.  Retaining the edge labels, set on \(X^p\)
\[
V_N^{\gamma}(\vec{x}) := \prod_{e \in \mathsf{E}_{\gamma^*}} \rho_N\bigl(x^{\mathfrak{s}(e)},\, x^{\mathfrak{t}(e)}\bigr),
\qquad \vec x=(x^1,\dots,x^p).
\]
The function \(V_N^\gamma\) is independent of \(x^b\) whenever \(\alpha_b=0\).
\end{defi}

For an ordered composition of horizontal coordinate fields on \(X^p\),
we label each term in the Leibniz expansion of \(V_N^\gamma\) by
\(\mathfrak a\).  The index records, for each occurrence of a field,
the incident labelled edge factor on which it acts.  Write
\(\mathcal H_{e,\mathfrak a}\) for the subcomposition assigned to
edge \(e\), in the inherited order, taking the identity when no field
is assigned.  The orders \(|\mathcal H_{e,\mathfrak a}|_1\) and
\(|\mathcal H_{e,\mathfrak a}|_2\) count the fields at the source and
target endpoints, respectively.

Figure~\ref{fig:wick-diagram-multigraph} illustrates the correspondence
between a diagram and its associated multigraph.
\begin{figure}[htbp]
\centering
\begin{minipage}[c]{0.46\textwidth}
\centering
\begin{tikzpicture}[font=\small, line width=0.6pt,
  copy/.style={circle, fill=black, inner sep=1.25pt}]
\node at (0,3.75) {Unbarred};
\node at (3.6,3.75) {Barred};
\node[copy, label=left:{$1^{(1)}$}] (u11) at (0,3.2) {};
\node[copy, label=right:{$\bar2^{(1)}$}] (b21) at (3.6,3.2) {};
\node[copy, label=left:{$1^{(2)}$}] (u12) at (0,2.4) {};
\node[copy, label=right:{$\bar2^{(2)}$}] (b22) at (3.6,2.4) {};
\node[copy, label=left:{$2^{(1)}$}] (u21) at (0,1.6) {};
\node[copy, label=right:{$\bar1^{(1)}$}] (b11) at (3.6,1.6) {};
\node[copy, label=left:{$2^{(2)}$}] (u22) at (0,0.8) {};
\node[copy, label=right:{$\bar3^{(1)}$}] (b31) at (3.6,0.8) {};
\node[copy, label=left:{$3^{(1)}$}] (u31) at (0,0) {};
\node[copy, label=right:{$\bar1^{(2)}$}] (b12) at (3.6,0) {};
\draw[line width=1.2pt] (u11) -- node[above=2pt] {$e_1$} (b21);
\draw (u12) -- node[above=2pt] {$e_2$} (b22);
\draw (u21) -- node[above=2pt] {$e_3$} (b11);
\draw[line width=1.2pt] (u22) -- node[above=2pt] {$e_4$} (b31);
\draw (u31) -- node[above=2pt] {$e_5$} (b12);
\end{tikzpicture}
\par\smallskip
{\small (a) Diagram \(\gamma\)}
\end{minipage}\hfill
\begin{minipage}[c]{0.46\textwidth}
\centering
\begin{tikzpicture}[font=\small, line width=0.6pt,
  vertex/.style={circle, draw, minimum size=6mm, inner sep=1pt},
  edge/.style={-{Stealth[length=2mm,width=1.3mm]}}]
\node[vertex] (v1) at (-1.7,1.1) {$1$};
\node[vertex] (v2) at (1.7,1.1) {$2$};
\node[vertex] (v3) at (0,-1.5) {$3$};
\draw[edge, line width=1.2pt] (v1) to[bend left=35] node[above=2pt] {$e_1$} (v2);
\draw[edge] (v1) -- node[above=2pt] {$e_2$} (v2);
\draw[edge] (v2) to[bend left=35] node[below=2pt] {$e_3$} (v1);
\draw[edge, line width=1.2pt] (v2) -- node[right=3pt] {$e_4$} (v3);
\draw[edge] (v3) -- node[left=3pt] {$e_5$} (v1);
\end{tikzpicture}
\par\smallskip
{\small (b) Directed multigraph \(\gamma^*\)}
\end{minipage}
\caption{A diagram \(\gamma\in\Gamma(2,2,1)\) for one section evaluated
at three points, together with its associated directed multigraph
\(\gamma^*\).  Each edge
\(\{i^{(a)},\bar j^{(b)}\}\) on the left becomes \(i\to j\) on the
right.  All edge labels are retained; \(e_1\) and \(e_2\) are distinct
parallel edges.  The indegree and outdegree sequences are both
\((2,2,1)\).  The thick edges \(e_1,e_4\) form a spanning tree of the
underlying undirected graph.}
\label{fig:wick-diagram-multigraph}
\end{figure}

By~\eqref{eq:abs-of-CR-fun} and Proposition~\ref{prop:wick-formula},
\begin{equation}\label{eq:Wick-z-x}
\mathbb{E}\!\left[\prod_{j=1}^{p}
:\!|\widetilde{s}^N(z^j)|_{h_N}^{2\alpha_j}\!:\right]
=\sum_{\gamma\in\Gamma(\alpha_1,\dots,\alpha_p)}V_N^\gamma(\vec x)
\end{equation}
for any lifts \(x^i\in X\) of \(z^i\).

\subsection{Moment expansion and connectivity}
\label{sec:combined-directed-multigraph}

We first associate an intrinsic correlation current to each diagram, then
expand the moments at a fixed cutoff.  Each summand
in~\eqref{eq:Wick-z-x}, not merely their sum, is intrinsic.

Indeed, each vertex has \(\deg^+(i)=\deg^-(i)=\alpha_i\), so the
fibre phases from~\eqref{CR-function} cancel.  Hence \(V_N^\gamma\)
descends to a function on \(M^p\), denoted by \(V_N^\gamma(\vec z)\).

\begin{defi}\label{def:Feynman-correlation-current}
For \(\gamma\in\Gamma(\alpha_1,\dots,\alpha_p)\), the sets
\(I_+(\gamma)\) and \(I_0(\gamma)\) are those in~\eqref{eq:I-plus}--\eqref{eq:I-zero}.
Products indexed by subsets of \(\{1,\dots,p\}\) are taken in increasing
order.  The associated \((p,p)\)-form on \(M^p\) is
\begin{equation}\label{eq:FC-current}
\mathrm{FC}_{N}^{\gamma}(\vec z)
:=\left(\prod_{a\in I_+(\gamma)}\frac{c_{2\alpha_a}}{\alpha_a!}\right)
\left(\prod_{a\in I_+(\gamma)}\frac{\sqrt{-1}}{\pi}
 \partial_a\bar\partial_a\right)
\left[V_N^\gamma(\vec z)\wedge
 \bigwedge_{b\in I_0(\gamma)}\pi_b^*\mathcal C_0^N\right].
\end{equation}
In particular,
\(\mathrm{FC}_N^{\gamma_\varnothing^p}
=\bigwedge_{s=1}^p\pi_s^*\mathcal C_0^N\).
The coefficients, operators and mean current are those
of~\eqref{eq:decomposition-in-probability}, \eqref{def:partial_j}
and~\eqref{eq:C0}.
\end{defi}

\begin{pro}
\phantomsection
\label{prop:p-point-correlation}
\begin{enumerate}
\item[\textup{(i)}] For one Gaussian section and integers
\(\alpha_1,\dots,\alpha_p\geqslant0\), the expected exterior product is
\[
\mathbb{E}\Bigl[\,\pi_{1}^{*}\mathcal{C}^{N}_{\alpha_{1}}\wedge\cdots
\wedge\pi_{p}^{*}\mathcal{C}^{N}_{\alpha_{p}}\Bigr]
= \sum_{\gamma \in \Gamma(\alpha_1, \dots, \alpha_p)}
\mathrm{FC}_{N}^{\gamma}(\vec{z})
\]
as an identity of smooth \((p,p)\)-forms on \(M^p\).

\item[\textup{(ii)}] Let \(\varphi\) be a real \(\mathscr C^3\) test form
of bidegree \((m-k,m-k)\), or a test of type~\textup{(D)}, and let
\(\widehat X_{N,\mathbf q}^\varphi\) be the centred
statistic in~\eqref{eq:general-mixed-statistic}, where
\(1\leqslant q_1,\dots,q_k<+\infty\).  Write \(\vec0=(0,\dots,0)\)
and use the finite index set
\begin{equation*}
A_{\mathbf q}:=\prod_{s=1}^k\{0,\dots,q_s\}.
\end{equation*}
Then, for every integer \(p\geqslant1\),
\begin{equation}\label{eq:EXP}
\begin{aligned}
\mathbb{E}\!\left[ \left( \widehat X_{N,\mathbf q}^\varphi \right)^p \right]
={}&\sum_{\vec\alpha^1,\dots,\vec\alpha^p\in
 A_{\mathbf q}\setminus\{\vec0\}}
 \sum_{\substack{\gamma_s\in\Gamma(\alpha_s^1,\dots,\alpha_s^p)\\
                  1\leqslant s\leqslant k}}
 \int_{M^p}\bigwedge_{s=1}^k\mathrm{FC}_N^{\gamma_s}
 \wedge\bigwedge_{i=1}^p\pi_i^*\varphi.
\end{aligned}
\end{equation}
Here \(\vec\alpha^i=(\alpha_1^i,\dots,\alpha_k^i)\), and the inner sum
runs over one diagram for each independent section.
\end{enumerate}
\end{pro}

\begin{proof}
For~\textup{(i)}, take the deterministic factors
\(\pi_b^*\mathcal C_0^N\), \(b\in I_0\), outside the expectation and
insert~\eqref{eq:Wick-z-x} into~\eqref{eq:correlation-expectation}.
Equation~\eqref{eq:FC-current} identifies each summand, including the
unique empty diagram when \(I_+=\varnothing\).
Remark~\ref{rmk:on-Gamma(alpha)} also gives
\(\mathbb E[\mathcal C_\alpha^N]=0\) for \(\alpha\geqslant1\).

For~\textup{(ii)}, the geometric representation
\eqref{eq:mixed-statistic-projection} and its zero-cutoff case give
the finite expansion
\begin{equation*}
\widehat X_{N,\mathbf q}^\varphi
=\sum_{\vec\alpha\in A_{\mathbf q}\setminus\{\vec0\}}
 \int_M\bigwedge_{j=1}^k\mathcal C_{\alpha_j}^{N,j}\wedge\varphi.
\end{equation*}
Expand the \(p\)-th power and group the even-degree factors by section.
Finite sums of Gaussian polynomials on compact \(M\), paired with
bounded test coefficients in either case, justify Fubini.  Independence
and~\textup{(i)} give~\eqref{eq:EXP}.
\end{proof}

For \(\gamma_s\in\Gamma(\alpha_s^1,\dots,\alpha_s^p)\),
\(1\leqslant s\leqslant k\), the combined directed multigraph
\(G=G(\gamma_1,\dots,\gamma_k)\) has vertex set \(\{1,\dots,p\}\)
and edge set \(\bigsqcup_s\mathsf E_{\gamma_s^*}\).
Each edge retains its source and target and receives colour \(s\);
parallel edges remain distinct.  Connectedness and connected components
refer to the underlying undirected graph.  For any such diagram tuple and either
test class in Proposition~\ref{prop:p-point-correlation}\textup{(ii)}, set
\begin{equation*}
 I_N(G):=\int_{M^p}\bigwedge_{s=1}^k\mathrm{FC}_N^{\gamma_s}
             \wedge\bigwedge_{j=1}^p\pi_j^*\varphi.
\end{equation*}

\subsection{Integration by parts}
\label{sec:asymptotic-analysis}

We extract one Dolbeault pair at each moment variable and establish the
Stokes identities for domain tests.

\begin{pro}
\phantomsection
\label{prop:extract-partial}
\begin{enumerate}
\item[\textup{(i)}] Let
\(\gamma\in\Gamma(\alpha_1,\dots,\alpha_p)\), with the notation of
Definition~\ref{def:Feynman-correlation-current}.  For
\(\mathsf C\subset I_+(\gamma)\), set
\begin{equation}\label{eq:FC-c}
\mathrm{FC}_{N}^{\gamma,\mathsf C}(\vec z)
:=\left(\prod_{a\in I_+(\gamma)}\frac{c_{2\alpha_a}}{\alpha_a!}\right)
 \left(\prod_{d\in I_+(\gamma)\setminus\mathsf C}
  \frac{\sqrt{-1}}{\pi}\partial_d\bar\partial_d\right)
 \left[V_N^\gamma(\vec z)\wedge
  \bigwedge_{b\in I_0(\gamma)}\pi_b^*\mathcal C_0^N\right].
\end{equation}
This form has bidegree \((0,0)\) in each variable in
\(\mathsf C\), and \((1,1)\) in each remaining variable.
It gives the factorization
\begin{equation}\label{eq:partial-extraction-factorization}
\mathrm{FC}_N^\gamma
=\left(\prod_{c\in\mathsf C}
 \frac{\sqrt{-1}}{\pi}\partial_c\bar\partial_c\right)
 \mathrm{FC}_{N}^{\gamma,\mathsf C},
\end{equation}
with \(\mathrm{FC}_N^{\gamma,\varnothing}=\mathrm{FC}_N^\gamma\), and,
for every \(i\notin\mathsf C\),
\begin{equation}\label{eq:partial-extraction-annihilation}
 \partial_i\mathrm{FC}_{N}^{\gamma,\mathsf C}
 =\bar\partial_i\mathrm{FC}_{N}^{\gamma,\mathsf C}=0.
\end{equation}

\item[\textup{(ii)}] Fix \(p\geqslant2\) and diagrams
\(\gamma_s\in\Gamma(\alpha_s^1,\dots,\alpha_s^p)\),
\(1\leqslant s\leqslant k\), whose combined graph \(G\) is connected.
There exists a partition
\begin{equation*}
\{1,\dots,p\}=\bigsqcup_{s=1}^k\mathsf C_s,
\qquad \mathsf C_s\subset I_+(\gamma_s),
\end{equation*}
allowing empty parts, such that
\[
\bigwedge_{s=1}^k\mathrm{FC}_N^{\gamma_s}
=\left(\prod_{j=1}^p
 \frac{\sqrt{-1}}{\pi}\partial_j\bar\partial_j\right)
 \left[\bigwedge_{s=1}^k
 \mathrm{FC}_N^{\gamma_s,\mathsf C_s}\right].
\]
\end{enumerate}
\end{pro}

\begin{proof}
For~\textup{(i)}, restore the operators
\((\sqrt{-1}/\pi)\partial_c\bar\partial_c\), \(c\in\mathsf C\),
in~\eqref{eq:FC-c}.  This gives~\eqref{eq:FC-current} without sign or
multiplicity.  If \(i\notin\mathsf C\), then either \(i\in I_0(\gamma)\),
in which case \(V_N^\gamma\) is independent of that variable and
\(\mathcal C_0^N\) is closed, or \(i\in I_+(\gamma)\setminus\mathsf C\),
in which case the displayed operator already contains
\(\partial_i\bar\partial_i\).  Since
\(\partial_i^2=\bar\partial_i^2=0\), both assertions follow.

For~\textup{(ii)}, connectedness and \(p\geqslant2\) ensure that
the \(I_+(\gamma_s)\) cover the vertices.  Assign each vertex to one such
set, independently of \(N\), to obtain the partition \((\mathsf C_s)\).
Apply~\eqref{eq:partial-extraction-factorization} in each colour.
By~\eqref{eq:partial-extraction-annihilation}, each
\(\partial_j,\bar\partial_j\) annihilates every factor except its assigned
one.  Only one Leibniz term survives, and all factors have even degree,
so there is no sign or multiplicity.  The bracket has bidegree
\((k-1,k-1)\) in each variable, and \(\sum_s|\mathsf C_s|=p\) gives the coefficient.
\end{proof}

\begin{lem}
\label{lem:boundary-Stokes}
Let \(U\subsetneq M\) be a domain with \(\mathscr C^2\) boundary.
Orient \(\partial U\) by the outward-normal-first convention and its
products by the ordered product orientation.
\textup{(i)} Let \(q\geqslant1\), let \(H\) be a smooth form on a neighbourhood of
\(\overline U^{\,q}\), of even total degree in each variable, and let
\(\Theta_1,\dots,\Theta_q\) be smooth, \(\mathrm d\)-closed forms of even
degree on \(M\).  Assume that
\(H\wedge\bigwedge_{j=1}^q\pi_j^*\Theta_j\) has degree \(2m-2\) in each
variable.  With the operator products in increasing order of the variable
index, successive applications of Stokes' theorem give
\begin{equation}\label{eq:repeated-boundary-Stokes}
\int_{U^q}
 \left(\prod_{j=1}^q\partial_j\bar\partial_j\right)H
 \wedge\bigwedge_{j=1}^q\pi_j^*\Theta_j
=
\int_{(\partial U)^q}
 \bar\partial_1\cdots\bar\partial_qH
 \wedge\bigwedge_{j=1}^q\pi_j^*\Theta_j.
\end{equation}

\textup{(ii)} If, in addition, \(H\) has pure bidegree in each variable
and each \(\Theta_j\) has pure bidegree,
\(H\wedge\bigwedge_j\pi_j^*\Theta_j\) has bidegree \((m-1,m-1)\) in
each variable, and \(f\) is a fixed real \(\mathscr C^1\) function on a
neighbourhood of \(\overline U\), then set, for
\(A\subset\{1,\dots,q\}\),
\[
D_A:=\prod_{j=1}^q S_j^A,
\qquad
S_j^A:=\begin{cases}\partial U,&j\in A,\\
U,&j\notin A,\end{cases}
\qquad
f_A:=\prod_{j\in A}f(z^j).
\]
With wedges in increasing index order,
\begin{equation*}
\begin{split}
&\int_{U^q}\left(\prod_{j=1}^q f(z^j)\right)
 \left(\prod_{j=1}^q\partial_j\bar\partial_j\right)H
 \wedge\bigwedge_{j=1}^q\pi_j^*\Theta_j\\
&\quad=\sum_{A\subset\{1,\dots,q\}}
 (-1)^{\sum_{j\notin A}(q-j)}
 \int_{D_A}f_A\,\bar\partial_1\cdots\bar\partial_qH
 \wedge\bigwedge_{j\notin A}\partial_j f
 \wedge\bigwedge_{j=1}^q\pi_j^*\Theta_j.
\end{split}
\end{equation*}
Here \(\partial_jf=\pi_j^*(\partial f)\).  When \(f\equiv1\), only
\(A=\{1,\dots,q\}\) survives, and the weighted formula reduces
to~\eqref{eq:repeated-boundary-Stokes}.
\end{lem}

\begin{proof}[Proof of Lemma~\ref{lem:boundary-Stokes}]
\smallskip\noindent\emph{Unweighted Stokes identity.}
The usual Stokes formula on a domain with \(\mathscr C^2\) boundary gives,
for every smooth compactly supported \((2m-1)\)-form \(\eta\),
\begin{equation}\label{eq:current-Stokes-domain}
 \int_U\mathrm d\eta=\int_{\partial U}\eta.
\end{equation}
Work with ordered product orientations and one homogeneous
multidegree component, and set
\begin{equation*}
 \Omega_r
 :=\bar\partial_1\cdots\bar\partial_r
   \left(\prod_{j=r+1}^q\partial_j\bar\partial_j\right)H
   \wedge\bigwedge_{j=1}^q\pi_j^*\Theta_j,
\end{equation*}
and \(D_r:=(\partial U)^r\times U^{q-r}\),
\(0\leqslant r\leqslant q\).  The degree of \(\Omega_r\) is \(2m-1\) in
the first \(r\) variables and \(2m\) in the others, matching \(D_r\);
\(\Omega_0,\Omega_q\) are the integrands
in~\eqref{eq:repeated-boundary-Stokes}.

As operators on differential forms on \(M^q\), the differentials
\(\mathrm d_r=\partial_r+\bar\partial_r\) anticommute for distinct \(r\).
Since \(\mathrm d\Theta_j=0\) and \(\bar\partial_r^2=0\), anticommuting
\(\mathrm d_r\) past \(\bar\partial_1,\ldots,\bar\partial_{r-1}\) yields
\begin{equation}\label{eq:Stokes-operator-sign}
 \mathrm d_r\Omega_r=(-1)^{r-1}\Omega_{r-1}.
\end{equation}
Set \(A_{r-1}:=(\partial U)^{r-1}\).
The \(r\)-th outward normal crosses \(\dim A_{r-1}\) tangent directions,
so~\eqref{eq:current-Stokes-domain} in variable \(r\) and Fubini give
\begin{equation*}
 \int_{D_{r-1}}\mathrm d_r\Omega_r
 =(-1)^{\dim A_{r-1}}\int_{D_r}\Omega_r
 =(-1)^{(r-1)(2m-1)}\int_{D_r}\Omega_r.
\end{equation*}
Since \(2m-1\) is odd, this sign cancels~\eqref{eq:Stokes-operator-sign}:
\[
 \int_{D_{r-1}}\Omega_{r-1}
 =(-1)^{r-1}(-1)^{r-1}\int_{D_r}\Omega_r
 =\int_{D_r}\Omega_r.
\]
Iteration for \(r=1,\dots,q\) proves
\eqref{eq:repeated-boundary-Stokes}.

Compactness and finite surface measure justify Fubini.

\smallskip\noindent\emph{Weighted identity.}
First take \(q=1\) and write
\(\Theta=\Theta_1\).  The bidegree assumption and closedness of
\(\Theta\) give \(\partial\Theta=\bar\partial\Theta=0\) and
\(\mathrm d(f\bar\partial H\wedge\Theta)
=\partial f\wedge\bar\partial H\wedge\Theta
+f\partial\bar\partial H\wedge\Theta\).
By~\eqref{eq:current-Stokes-domain},
\begin{equation*}
\int_U f\partial\bar\partial H\wedge\Theta
=\int_{\partial U}f\bar\partial H\wedge\Theta
 +\int_U\bar\partial H\wedge\partial f\wedge\Theta.
\end{equation*}
The plus sign uses the odd degree of \(\bar\partial H\).
Iterate in the ordered variables: moving an interior factor
\(\partial_jf\) past the \(q-j\) later odd differentiated factors gives
\((-1)^{q-j}\); the even \(\Theta_j\) contribute no sign.
This applies coefficientwise in multidegree.  Bounded \(f,\mathrm df\)
and finite surface measure justify Fubini.
\end{proof}

\subsection{Kernel estimates and derivative orders}
\label{sec:kernel-connected}

For smooth tests, integration by parts transfers both derivatives of
each extracted pair to the test form.  For domain tests, one derivative
from each pair remains on the diagram.  We use the combined volume and
surface measure
\begin{equation*}
\mathrm d\mu_U:=\chi_U\frac{\omega^m}{m!}
                  +\mathrm d\sigma_{\partial U}
\quad\text{on }\overline U.
\end{equation*}

For the two test classes, set
\begin{equation}\label{eq:kernel-integration-scales}
(S_\varphi,\mu_\varphi,a_N):=
\begin{cases}
(M,\omega^m/m!,N^{k-1}),&\textup{(S)},\\
(\overline U,\mu_U,N^{k-\frac12}),&\textup{(D)}.
\end{cases}
\end{equation}
Here \(a_N\) is the power of \(N\) contributed by derivatives and
mean-current factors at each moment variable after integration by parts.

\begin{lem}
\label{lem:kernel-connected-bound}
Fix \(p\geqslant2\) and diagrams
\(\gamma_s\in\Gamma(\alpha_s^1,\dots,\alpha_s^p)\), \(1\leqslant s\leqslant k\),
whose combined graph \(G\) is connected.  Let \(\varphi\) be of
type~\textup{(S)} or~\textup{(D)}, with
\((S_\varphi,\mu_\varphi,a_N)\) as in~\eqref{eq:kernel-integration-scales}.
Suppose that \(J_N\) satisfies
\[
\sup_{z,w\in S_\varphi}J_N(z,w)\leqslant C,\qquad
\sup_{z\in S_\varphi}\int_{S_\varphi} J_N(z,w)\,\mathrm d\mu_\varphi(w)
 \leqslant C\varrho_N,\qquad 0<\varrho_N\longrightarrow0.
\]
The connected-diagram integral satisfies
\begin{equation}\label{eq:kernel-connected-integral}
 |I_N(G)|\leqslant C_G a_N^p\varrho_N^{p-1}.
\end{equation}
The constant may depend on the fixed diagrams, the geometric data,
the test form (including \(U,f\) in type~\textup{(D)}), and the constants
in the assumed bounds, but not on \(N\).
\end{lem}

\begin{proof}
Lemmas~\ref{lem:weighted-mean-current-bound}
and~\ref{lem:Bergman-horizontal-comparison} give
\(\|\mathcal C_0^N\|_{L^\infty(S_\varphi)}\leqslant CN\) and
\begin{equation}\label{eq:kernel-envelope}
 |\mathcal H\rho_N(x,y)|\leqslant C N^{|\mathcal H|/2}J_N(\pi x,\pi y),
 \quad |\mathcal H|_1,|\mathcal H|_2\leqslant2,\quad \pi x,\pi y\in S_\varphi.
\end{equation}
We use the partition \((\mathsf C_s)\) from
Proposition~\ref{prop:extract-partial}\textup{(ii)} and write
\[
\mathcal K_N:=\bigwedge_{s=1}^k\mathrm{FC}_N^{\gamma_s,\mathsf C_s},
\qquad
\mathcal B_N:=\bar\partial_1\cdots\bar\partial_p\mathcal K_N.
\]
For each vertex \(j\), let \(s\) be the unique index with
\(j\in\mathsf C_s\).  Proposition~\ref{prop:extract-partial}\textup{(i)}
shows that \(\partial_j\) and \(\bar\partial_j\) annihilate every
factor \(\mathrm{FC}_N^{\gamma_t,\mathsf C_t}\) with \(t\ne s\).
Thus the outer \(\bar\partial_j\) in \(\mathcal B_N\) acts only on
\(\mathrm{FC}_N^{\gamma_s,\mathsf C_s}\), from which the operator
\(\partial_j\bar\partial_j\) was omitted.

\smallskip\noindent\emph{Transfer the extracted derivatives.}
In type~\textup{(S)}, Fubini and Stokes on \(M^p\) give the exact identity
\[
I_N(G)=\int_{M^p}\mathcal K_N\wedge
 \bigwedge_{j=1}^p
 \frac{\sqrt{-1}}{\pi}\partial_j\bar\partial_j\varphi(z^j).
\]
The transfer precedes every absolute-value estimate; it removes two
kernel derivatives at each vertex.  In type~\textup{(D)},
\(\varphi=\varphi_U=f\chi_U\omega^{m-k}/(m-k)!\).
The smooth form \(\mathcal K_N\) has bidegree \((k-1,k-1)\) in each
  variable.  Lemma~\ref{lem:boundary-Stokes}\textup{(ii)}, with \(q=p\) and
\(\Theta_j=\omega^{m-k}\), gives
\[
I_N(G)=
 \left(\frac{\sqrt{-1}}{\pi(m-k)!}\right)^p
 \sum_{A\subset\{1,\dots,p\}}(-1)^{\sum_{j\notin A}(p-j)}
 \int_{D_A}f_A\,\mathcal B_N
       \wedge\bigwedge_{j\notin A}\partial_j f
       \wedge\bigwedge_{j=1}^p\pi_j^*\omega^{m-k}.
\]
\smallskip\noindent\emph{Count the derivatives remaining on the edges.}
The smooth and domain integrals require, respectively, a bound for the
form and a bound after one further derivative in each variable.  For a colour
\(\gamma\), set \(l=|I_+(\gamma)|\) and fix
\(\mathsf C\subset I_+(\gamma)\).
For an ordered composition \(D\) of the coordinate fields
\(\partial/\partial z_a^j,\partial/\partial\bar z_a^j\), write
\(D^{\mathrm H}\) for the composition obtained by replacing them
with \(Z_a^{(j)},\bar Z_a^{(j)}\) in the same order.
These horizontal fields commute with every circle generator
\(T_j=\partial/\partial\theta_j\), so all intermediate horizontal
derivatives of \(V_N^\gamma\) remain fibre-invariant.  Their vertical
parts in~\eqref{eq:horizontal-derivative} therefore vanish at each
step.  Induction, followed by restriction to a local product section
of \(X^p\to M^p\), gives
\begin{equation}\label{eq:balanced-base-horizontal-identity}
 D V_N^\gamma=D^{\mathrm H} V_N^\gamma.
\end{equation}
On the fixed atlas, write the mean-current factors as
\[
 \pi_b^*\mathcal C_0^N
 =\sum_{u,v=1}^m(\mathcal C_0^N)_{u\bar v}(z^b)
   \,\mathrm dz_u^b\wedge\mathrm d\bar z_v^b,
 \qquad b\in I_0(\gamma).
\]
Equation~\eqref{eq:balanced-base-horizontal-identity} lets us apply the
horizontal derivative bound to the coordinate coefficients of
\(\mathrm{FC}_N^{\gamma,\mathsf C}\).  Expanding by the Leibniz rule
on \(X^p\), every coefficient on the fixed atlas becomes a finite sum of
terms
\begin{equation*}
 c_{\mathfrak a}
 \left(\prod_{e\in\mathsf E_{\gamma^*}}
     \mathcal H_{e,\mathfrak a}\rho_{N,e}\right)
 \left(\prod_{b\in I_0(\gamma)}
     (\mathcal C_0^N)_{u_b\bar v_b}(z^b)\right),
\end{equation*}
where \(\rho_{N,e}:=\rho_N(x^{\mathfrak s(e)},x^{\mathfrak t(e)})\).
Here \(\mathfrak a\) records the edge assignments and the coordinate
indices, including the selected pair \((u_b,v_b)\) at each
\(b\in I_0(\gamma)\).  The numerical factor \(c_{\mathfrak a}\)
combines the fixed prefactors
\(\prod_{a\in I_+(\gamma)}c_{2\alpha_a}/\alpha_a!\) and
\((\sqrt{-1}/\pi)^{l-|\mathsf C|}\) in~\eqref{eq:FC-c}
with the Leibniz multiplicities and the signs from reordering exterior
factors; it is independent of \(N\) and \(\vec z\).

Lemma~\ref{lem:weighted-mean-current-bound} gives
\(\max_{u,v}|(\mathcal C_0^N)_{u\bar v}(z^b)|\leqslant CN\).
Every labelled edge occurs once,
and its derivative order at an endpoint is at most the number of
fields of \(D\) acting at that vertex.  For the coordinate-derivative compositions in
\(\mathrm{FC}_N^{\gamma,\mathsf C}\), we have
\(|D|=2(l-|\mathsf C|)\); hence
\[
\sum_e|\mathcal H_{e,\mathfrak a}|=2(l-|\mathsf C|),\qquad
|\mathcal H_{e,\mathfrak a}|_1\leqslant2,\quad
|\mathcal H_{e,\mathfrak a}|_2\leqslant2.
\]
The expression~\eqref{eq:FC-c} contains one
\(\partial_j\bar\partial_j\)-operator for each
\(j\in I_+(\gamma)\setminus\mathsf C\), and none at the other vertices.
These derivatives act at \(I_+(\gamma)\) and leave the deterministic
factors at \(I_0(\gamma)\) untouched.

By~\eqref{eq:kernel-envelope}, the derivative and deterministic powers in
one such term are
\[
N^{l-|\mathsf C|}N^{p-l}=N^{p-|\mathsf C|}.
\]
Taking the product over colours and using \(\sum_s|\mathsf C_s|=p\) gives
\begin{equation*}
\|\mathcal K_N(\vec z)\|
 \leqslant C_G N^{p(k-1)}
 \prod_{e\in\mathsf E_G}
 J_N(z^{\mathfrak s(e)},z^{\mathfrak t(e)}).
\end{equation*}
All norms here and below are induced by the fixed product metric.

For \(\mathcal B_N\), closedness and anticommutation leave exactly one
Leibniz term, up to sign, assigning each outer \(\bar\partial_j\) to its
colour.  The differentiated factor is
\begin{align*}
\left(\prod_{c\in\mathsf C_s}\bar\partial_c\right)
 \mathrm{FC}_N^{\gamma_s,\mathsf C_s}
 ={}&\left(\prod_{a\in I_+(\gamma_s)}
       \frac{c_{2\alpha_s^a}}{\alpha_s^a!}\right)
 \left(\bigwedge_{b\in I_0(\gamma_s)}\pi_b^*\mathcal C_0^N\right)\\
 &\quad\wedge
 \left(\prod_{d\in I_+(\gamma_s)\setminus\mathsf C_s}
       \frac{\sqrt{-1}}{\pi}\partial_d\bar\partial_d\right)
 \left(\prod_{c\in\mathsf C_s}\bar\partial_c\right)
 V_N^{\gamma_s}.
\end{align*}
The deterministic factors remain untouched.  Vertices in \(\mathsf C_s\)
receive one derivative in colour \(s\), and the other positive vertices
at most two.  Thus the endpoint orders remain at most two, and the total
additional order is \(p\).  Equations
\eqref{eq:balanced-base-horizontal-identity} and~\eqref{eq:kernel-envelope} give
\[
\|\mathcal B_N(\vec z)\|
 \leqslant C_G N^{p(k-\frac12)}
 \prod_{e\in\mathsf E_G}
 J_N(z^{\mathfrak s(e)},z^{\mathfrak t(e)}).
\]
The ambient derivative and norm estimates also bound the corresponding
forms pulled back to the real boundary.

\smallskip\noindent\emph{Integrate along a spanning tree.}
The transferred test forms, the K\"ahler factors, and \(f,\partial f\)
have bounded norms.  In type~\textup{(D)}, summing the mixed products
of volume and surface measure gives exactly \(\mu_U^{\otimes p}\).
Hence both cases give
\[
|I_N(G)|\leqslant C_G a_N^p
 \int_{S_\varphi^p}\prod_{e\in\mathsf E_G}
 J_N(z^{\mathfrak s(e)},z^{\mathfrak t(e)})
 \prod_{j=1}^p\mathrm d\mu_\varphi(z^j).
\]
Choose a spanning tree \(T\) of the underlying undirected graph of
\(G\), rooted at \(1\).  In Figure~\ref{fig:wick-diagram-multigraph},
\(e_1,e_4\) give such a tree, with leaf removal integrating first
in \(z^3\) and then in \(z^2\).  Symmetry permits either edge
orientation, and the non-tree factors are bounded by
\((\sup J_N)^{|\mathsf E_G|-(p-1)}\leqslant C_G\).
For a leaf \(v\ne1\), write \(\operatorname{par}_T(v)\) for its
parent in \(T\).  Integration in the leaf variable gives
\[
\int_{S_\varphi} J_N(z^{\operatorname{par}_T(v)},z^v)\,
          \mathrm d\mu_\varphi(z^v)\leqslant C\varrho_N.
\]
Remove leaves successively.  The \(p-1\) integrations contribute
\((C\varrho_N)^{p-1}\); the final root integral is the fixed finite mass
\(\mu_\varphi(S_\varphi)\).  This proves~\eqref{eq:kernel-connected-integral}.
\end{proof}

\subsection{Gaussian moments at a fixed cutoff}

The connected-diagram estimates of Section~\ref{sec:kernel-connected}
give an upper bound for the variance at a fixed cutoff.  With a matching
lower bound, components with at least three vertices have vanishing
normalized contributions, leaving the Gaussian pairings.

\begin{pro}\label{prop:finite-cutoff-clt}
Let \(\varphi\) be admissible, and assume the uniform and integral
bounds for \(J_N\) in Lemma~\ref{lem:kernel-connected-bound}, with
\(0<\varrho_N\to0\).
For every fixed integer \(q\geqslant1\),
\begin{equation}\label{eq:finite-cutoff-row-variance}
\|\widehat X_{N,q}^\varphi\|_2^2\leqslant C_q a_N^2\varrho_N.
\end{equation}
If also \(\|\widehat X_{N,q}^\varphi\|_2^2\geqslant c_q a_N^2\varrho_N\) for all
large \(N\), with \(c_q>0\), then
\[
\widehat X_{N,q}^\varphi/\|\widehat X_{N,q}^\varphi\|_2
\xrightarrow{\ d\ }\mathcal N_{\mathbb R}(0,1).
\]
\end{pro}

Lemma~\ref{lem:common-cutoff-removal} obtains the required variance
lower bound for every sufficiently large fixed \(q\) from the full
variance lower bound and the truncation estimate
\eqref{eq:abstract-single-block-tail}.

\begin{proof}
\smallskip\noindent\emph{Connected components.}
Fix \(q\geqslant1\) and write \(\sigma_{N,q}=\|\widehat X_{N,q}^\varphi\|_2\).
For \(p\geqslant2\), the expansion~\eqref{eq:EXP} has a finite index set
independent of \(N\).  Centring restricts the joint index at each vertex
to \(A=A_{(q,\dots,q)}\setminus\{\vec0\}\), so it is positive in some
colour.  The Wick matching in that colour joins distinct vertices,
and hence every vertex has an incident edge.  The connected components
\(G_1,\dots,G_{L(G)}\) of a graph \(G\) therefore have sizes \(p_a\geqslant2\),
with \(L(G)\leqslant p/2\).

Restrict each coloured diagram to the vertices of \(G_a\) and relabel them
in increasing order as \(1,\dots,p_a\), obtaining \(\gamma_s^a\).
Lemma~\ref{lem:kernel-connected-bound} gives
\begin{equation}\label{eq:criterion-connected}
|I_N(G_a)|\leqslant C_{q,G_a}a_N^{p_a}\varrho_N^{p_a-1}.
\end{equation}
Since no edge crosses components, the kernels, coefficients and
operators factor:
\begin{equation}\label{integration-factorizes}
I_N(G)=\prod_{a=1}^{L(G)}I_N(G_a).
\end{equation}
All forms have even degree, so grouping the variables introduces
neither a sign nor a multiplicity.  Since \(\sum_a p_a=p\),
\eqref{eq:criterion-connected} bounds this product by
\(C_{p,q,G}a_N^p\varrho_N^{p-L(G)}\).  At \(p=2\), each term is connected,
so summation proves~\eqref{eq:finite-cutoff-row-variance}.

Under the lower bound \(\sigma_{N,q}^2\geqslant c_q a_N^2\varrho_N\),
\[
\frac{|I_N(G)|}{\sigma_{N,q}^{p}}
\leqslant C_{p,q,G}\varrho_N^{p/2-L(G)}.
\]
The sum over graphs with fewer than \(p/2\) components thus satisfies
\begin{equation}\label{eq:L<p/2}
\sum_{L(G)<p/2}I_N(G)
=o(\sigma_{N,q}^{p}).
\end{equation}

\smallskip\noindent\emph{Pair partitions and the Gaussian limit.}
For \(p=2L\), each graph with \(L\) components gives a pair partition
\begin{equation*}
\{1,\dots,2L\}=\bigsqcup_{a=1}^L\{i_{2a-1},i_{2a}\}.
\end{equation*}
Order each pair increasingly and pairs by their smaller labels.
For each colour \(s\), equality of the incoming and outgoing edge
counts at each vertex forces
\(\alpha_s^{i_{2a-1}}=\alpha_s^{i_{2a}}=:\beta_s^a\), with
\(\gamma_s^a\in\Gamma(\beta_s^a,\beta_s^a)\).
Restriction and relabelling are reversed by disjoint union with the
original labels.  Since \(\vec\beta^a\in A\), each pair is connected.
Thus there is no extra multiplicity, and there are \((2L-1)!!\) pair
partitions.
The complete two-point sum is
\[
\sum_{\vec\beta\in A}
\sum_{\gamma_s\in\Gamma(\beta_s,\beta_s),\,1\leqslant s\leqslant k}
I_N(G)
=\mathbb E[(\widehat X_{N,q}^\varphi)^2]=\sigma_{N,q}^2,
\]
by~\eqref{eq:EXP} at \(p=2\) and Remark~\ref{rmk:on-Gamma(alpha)}.
Factorization~\eqref{integration-factorizes} makes each pair partition
contribute \(\sigma_{N,q}^{2L}\).  For the terms
in~\eqref{eq:L<p/2}, the exponent \(p/2-L(G)\) is at least \(1/2\)
for odd \(p\) and at least \(1\) for even \(p\).  Thus, for each fixed
integer \(p\geqslant3\),
\[
\mathbb E\!\left[\left(\frac{\widehat X_{N,q}^\varphi}{\sigma_{N,q}}\right)^p\right]
=\begin{cases}
 O_{p,q}(\varrho_N^{1/2}),&p\text{ odd},\\
 (p-1)!!+O_{p,q}(\varrho_N),&p\text{ even}.
\end{cases}
\]
The first two standardized moments are exactly zero and one.
Since the standard Gaussian distribution is determined by its moments,
the method of moments gives the asserted convergence.
\end{proof}

\section{Removing the chaos cutoff}
\label{sec:criterion-proof}

Abel regularization expresses the squared \(L^2\) error of truncating
one coefficient vector as a limit of two-point integrals.  Kernel
derivative bounds and sublevel-set estimates for the correlation gap
control these integrals.  The resulting uniform tail bound, together
with a variance lower bound, allows us to pass from the fixed-cutoff
Gaussian limit to the full statistic.

\subsection{A correlation-volume criterion}

The criterion combines pointwise and integral bounds for the kernel
envelope \(J_N\), a sublevel-set bound for the correlation gap
\(g_N:=1-\tau_N\), and a variance lower bound.

Fix an admissible test form \(\varphi\), and use
\((S_\varphi,\mu_\varphi,a_N)\) from~\eqref{eq:kernel-integration-scales}.

\begin{pro}
\label{thm:correlation-volume-criterion}
Let \(X_N^\varphi\) be a fixed admissible statistic.  Suppose that
there are \(\varrho_N>0\), \(\varrho_N\to0\), and constants \(C,c>0\) such that,
for all sufficiently large \(N\),
\begin{equation}\label{eq:cv-basic-bounds}
\sup_{S_\varphi^2}J_N\leqslant C,\qquad
\sup_{z\in S_\varphi}\int_{S_\varphi}
 J_N(z,w)\,\mathrm d\mu_\varphi(w)\leqslant C\varrho_N,\qquad
\operatorname{Var}X_N^\varphi\geqslant ca_N^2\varrho_N.
\end{equation}
Assume also that, for some \(\nu>k-1\) and \(t_0\in(0,1)\),
\begin{equation}\label{eq:cv-small-correlation-balls}
\sup_{z\in S_\varphi}
\mu_\varphi\{w\in S_\varphi:g_N(z,w)\leqslant t\}
\leqslant C\varrho_Nt^\nu,\qquad 0<t\leqslant t_0.
\end{equation}
The statistic then satisfies
\[
\operatorname{Var}X_N^\varphi\asymp a_N^2\varrho_N,
\qquad
\frac{X_N^\varphi-\mathbb EX_N^\varphi}
{\sqrt{\operatorname{Var}X_N^\varphi}}
\xrightarrow{\ d\ }\mathcal N_{\mathbb R}(0,1).
\]
\end{pro}

The tail estimates use \(g_N^{-(k-1)}\) in type~\textup{(S)}
and \(g_N^{-(k-1+\beta)}\), \(0<\beta<\nu-k+1\), in type~\textup{(D)}.
The condition \(\nu>k-1\) places these exponents below \(\nu\).
In type~\textup{(D)}, the Bargmann--Fock exponent
\(\nu=m-\tfrac12\) permits
\(0<\beta<m-k+\tfrac12\), including \(0<\beta<\tfrac12\) when \(k=m\).
As in the Sodin--Tsirelson criterion
\cite[Theorem~2.2 and Section~2.2]{MR2121537}, the kernel assumptions
control approximation by finite sums of Wiener-chaos components.
After integration by parts, the sublevel-set estimate controls the
singular factors in the tail integrals, and the variance lower bound
permits normalization.
The proof concludes after Lemma~\ref{lem:common-cutoff-removal}.

\subsection{Two-point chaos identities}

For a single section, the truncated pluri-bipotential gives the
covariance of the truncated divisor current.  Independence combines the
single-section identities in the second-moment formula for the Abel
tail in Lemma~\ref{lem:Abel-two-jet-tail}.

For the two-point formulas, let
\(\pi_j:M^2\to M\) be the factor projections and
\(\partial_j,\bar\partial_j\) the corresponding Dolbeault operators.
Recall that \(e_\alpha(\xi)=\alpha!^{-1}:|\xi|^{2\alpha}:\).
We use the operator and mean-current product
\[
\mathscr D_{12}:=\partial_1\partial_2\bar\partial_1\bar\partial_2,
\qquad
D_N:=\pi_1^*\mathcal C_0^N\wedge\pi_2^*\mathcal C_0^N.
\]
For \(q\in\mathbb Z_{\geqslant0}\), set
\begin{equation}\label{eq:finite-bi-potential}
Q_N^{\leqslant q}:=\frac{1}{4\pi^2}
\sum_{\alpha=1}^{q}\frac{\tau_N^\alpha}{\alpha^2}.
\end{equation}
The series in~\eqref{eq:finite-bi-potential} is the truncation of the
pluri-bipotential of Shiffman and Zelditch
\cite[Theorem~3.1 and (36)]{MR2465693}.
For currents \(T,S\) on \(M\), write
\(T\boxtimes S=\pi_1^*T\wedge\pi_2^*S\).

\begin{pro}
\label{prop:finite-cutoff-two-point-functional}
\leavevmode
\begin{enumerate}
\item[\textup{(i)}] Let \((\xi_1,\xi_2)\) be a possibly degenerate centred
circularly symmetric complex Gaussian pair with unit marginal variances,
and set \(\sigma=\mathbb E[\xi_1\overline{\xi_2}]\).  Then
\begin{equation}\label{eq:two-point-radial-chaos}
\mathbb E[e_\alpha(\xi_1)e_\beta(\xi_2)]
=\delta_{\alpha\beta}|\sigma|^{2\alpha},
\qquad \alpha,\beta\geqslant0.
\end{equation}
\item[\textup{(ii)}] For \(\alpha_1,\alpha_2\geqslant0\),
\begin{equation}\label{eq:two-point-current-chaos}
\mathbb E[\mathcal C_{\alpha_1}^N\boxtimes\mathcal C_{\alpha_2}^N]
=\delta_{\alpha_1\alpha_2}
\begin{cases}
D_N,&\alpha_1=0,\\[4pt]
\displaystyle
\frac{1}{4\pi^2\alpha_1^2}\mathscr D_{12}[\tau_N^{\alpha_1}],
&\alpha_1\geqslant1,
\end{cases}
\end{equation}
as an identity of smooth \((2,2)\)-forms on \(M^2\).  Moreover, for
\(n\in\mathbb Z_{\geqslant0}\),
\begin{equation}\label{2-correlation-of-Z}
\mathbb E\bigl[[Z_{s^N}]^{\leqslant n}
 \boxtimes[Z_{s^N}]^{\leqslant n}\bigr]
=D_N+\mathscr D_{12}Q_N^{\leqslant n}.
\end{equation}
\end{enumerate}
\end{pro}

\begin{proof}
For~\textup{(i)}, the Wick formula and the count in
Remark~\ref{rmk:on-Gamma(alpha)} give zero unless \(\alpha=\beta\).
For \(\alpha=\beta>0\), division by \((\alpha!)^2\) leaves
\(|\sigma|^{2\alpha}\); for \(\alpha=\beta=0\), the product is empty.

For positive indices in~\textup{(ii)}, apply~\textup{(i)} to the local
field \(\xi\) in~\eqref{eq:local-Gaussian} and differentiate in both
variables using
\(\mathcal C_\alpha^N=(\sqrt{-1}/\pi)c_{2\alpha}
\partial\bar\partial e_\alpha(\xi)\) and
\(c_{2\alpha}^2=(4\alpha^2)^{-1}\).
The coefficients are finite Gaussian polynomials, so differentiation
commutes with expectation.  The identity
\(-\partial_1\bar\partial_1\partial_2\bar\partial_2=\mathscr D_{12}\)
gives~\eqref{eq:two-point-current-chaos}; the cases with a zero index
follow from the deterministic \(\mathcal C_0^N\) and
\(\mathbb E\mathcal C_\alpha^N=0\) for \(\alpha\geqslant1\).
Summing proves~\eqref{2-correlation-of-Z}, including the empty sum at
\(n=0\).
\end{proof}

\subsection{Abel summation of the Wiener-chaos expansion}

For fixed \(N\), Abel regularization with \(0<r<1\) gives smooth forms
for products and integration by parts.  As \(r\uparrow1\), the
integrated statistics converge in \(L^2(\Omega)\).  Their limiting
second moments thus include the diagonal contribution on \(M^2\)
when \(k=m\).

\begin{defi}
\label{def:Abel-chaos-contractions}
For \(1\leqslant j\leqslant k\) and \(0<r<1\), the Abel summation
operator is
\begin{equation*}
\mathsf U_{j,r}
:=\sum_{\alpha=0}^{\infty}r^{\alpha/2}\mathsf P_{j,=\alpha}.
\end{equation*}
The Mehler regularization of \(\log|\xi|\) is
\[
g_r(\xi):=\mathbb E_\eta\!\left[
\log\bigl|r^{1/4}\xi+\sqrt{1-\sqrt r}\,\eta\bigr|\right],
\qquad
g_r^{[n]}:=c_0+\sum_{\alpha=1}^n r^{\alpha/2}c_{2\alpha}e_\alpha,
\]
where \(\eta\) is a standard complex Gaussian and
\(n\in\mathbb Z_{\geqslant0}\).  The Mehler parameter \(r^{1/4}\)
multiplies chaos of order \(2\alpha\) by
\((r^{1/4})^{2\alpha}=r^{\alpha/2}\).
Gaussian convolution makes \(g_r\)
smooth and radial.  For the local representatives \(\xi_j\) of
\(\widetilde s_j^N\) in~\eqref{eq:local-Gaussian}, set
\begin{align*}
[Z_{s_j^N}]_r
&:=\mathcal C_0^{N,j}
  +\frac{\sqrt{-1}}{\pi}\partial\bar\partial g_r(\xi_j),
\\
[Z_{s_j^N}]_{r,>n}
&:=\frac{\sqrt{-1}}{\pi}\partial\bar\partial
  (g_r-g_r^{[n]})(\xi_j).
\end{align*}
Radiality makes these random smooth forms globally defined.
\end{defi}

\begin{pro}
\label{prop:Abel-chaos-approximation}
Let \(N\) be sufficiently large, and let \(\varphi\) be of
type~\textup{(S)} or type~\textup{(D)}.
\begin{enumerate}
\item[\textup{(i)}] Each \(\mathsf U_{j,r}\) is a
self-adjoint contraction on \(L^2(\Omega)\) and converges strongly as
\(r\uparrow1\) to the projection onto the even Wiener chaoses in
\(\boldsymbol\zeta_j\).  For every
\(\mathbf q\), with \(F=\widehat X_{N,\mathbf q}^\varphi\) or
\(F=\mathbb E X_N^\varphi+\widehat X_{N,\mathbf q}^\varphi\),
\begin{equation}\label{eq:Abel-strong-convergence}
\prod_{j=1}^k\mathsf U_{j,r}F
\longrightarrow F
\quad\text{in }L^2(\Omega)
\qquad(r\uparrow1).
\end{equation}
\item[\textup{(ii)}] For fixed \(N\), \(0<r<1\),
\(n\in\mathbb Z_{\geqslant0}\), and every integer \(\ell\geqslant0\), the series
\begin{equation}\label{eq:Abel-spatial-series}
\begin{aligned}
\relax[Z_{s_j^N}]_r
&=\mathcal C_0^{N,j}
 +\sum_{\alpha=1}^{\infty}r^{\alpha/2}\mathcal C_\alpha^{N,j},\\
\relax[Z_{s_j^N}]_{r,>n}
&=\sum_{\alpha=n+1}^{\infty}r^{\alpha/2}\mathcal C_\alpha^{N,j}
\end{aligned}
\end{equation}
converge in \(L^2(\Omega;\mathscr C^\ell(M))\).

\item[\textup{(iii)}] For every \(n\in\mathbb Z_{\geqslant0}\), the tail
statistic satisfies
\begin{equation}\label{eq:Abel-tail-statistic-convergence}
\int_M[Z_{s_1^N}]_{r,>n}
\wedge\bigwedge_{j=2}^k[Z_{s_j^N}]_r\wedge\varphi
\longrightarrow \mathcal T_{N,n}^{\varphi,(1)}
\quad\text{in }L^2(\Omega)
\qquad(r\uparrow1).
\end{equation}
\end{enumerate}
\end{pro}

\begin{proof}
For~\textup{(i)}, orthogonality of the projections gives strong convergence of the defining
series and makes each \(\mathsf U_{j,r}\) a self-adjoint contraction.  Parseval's identity gives, for every \(F\in L^2(\Omega)\),
\[
\left\|\mathsf U_{j,r}F-
 \sum_{\alpha=0}^{\infty}\mathsf P_{j,=\alpha}F\right\|_2^2
=\sum_{\alpha=0}^{\infty}
 (1-r^{\alpha/2})^2\|\mathsf P_{j,=\alpha}F\|_2^2
\longrightarrow0.
\]
For either statistic \(F\) in the statement,
\eqref{eq:blockwise-projection-convergence} identifies the projection
limit with \(F\), including for type~\textup{(D)}.
Commutativity in distinct coefficient vectors proves
\eqref{eq:Abel-strong-convergence}.  The same conclusion holds for
\(\mathcal T_{N,n}^{\varphi,(1)}
=(I-\mathsf P_{1,\leqslant n})\widehat X_N^\varphi\), since
this bounded projection commutes with the contractions.

For~\textup{(ii)}, radiality of \(e_\alpha\) makes \(u_\alpha^N(z)=e_\alpha(\xi(z))\)
globally defined smooth random functions, and~\eqref{eq:two-point-radial-chaos} gives
\[
\mathbb E[u_\alpha^N(z)u_\beta^N(w)]
=\delta_{\alpha\beta}\tau_N(z,w)^\alpha.
\]
For fixed \(N\), smoothness of \(\tau_N\) on compact \(M^2\) and
\(0\leqslant\tau_N\leqslant1\) give
\[
\|\nabla^d(\tau_N^\alpha)\|_\infty
\leqslant C_{N,d}(1+\alpha)^d,
\qquad d\in\mathbb Z_{\geqslant0},\quad\alpha\geqslant1.
\]
The product rule differentiates at most \(d\) factors, leaving
nonnegative integer powers of \(\tau_N\) in every surviving term.  Differentiate the covariance in both variables,
then restrict to the diagonal.  With fixed real frames,
\(\mathcal C_\alpha^N=(\sqrt{-1}/\pi)c_{2\alpha}
\partial\bar\partial u_\alpha^N\) and \(|c_{2\alpha}|=1/(2\alpha)\) give
\begin{equation}\label{eq:fixed-N-chaos-Sobolev}
\begin{aligned}
\mathbb E\|u_\alpha^N\|_{H^s(M)}^2
&\leqslant C_{N,s}(1+\alpha)^{2s},\\
\mathbb E\|\mathcal C_\alpha^N\|_{H^s(M)}^2
&\leqslant C_{N,s}(1+\alpha)^{2s+2},
\qquad s\in\mathbb Z_{\geqslant0}.
\end{aligned}
\end{equation}
The form bound uses at most \(2s+4\) covariance derivatives and
the factor \(\alpha^{-2}\).  Spatial differentiation preserves chaos
orthogonality, so
\(\sum_{\alpha\geqslant1}r^\alpha(1+\alpha)^{2s+2}<\infty\)
gives convergence in \(L^2(\Omega;H^s(M))\).
Sobolev embedding for \(s>m+\ell\) gives the asserted
\(\mathscr C^\ell\) convergence.

Likewise, the scalar series
\(c_0+\sum_{\alpha\geqslant1}r^{\alpha/2}c_{2\alpha}u_\alpha^N\)
converges in \(L^2(\Omega;\mathscr C^{\ell+2}(M))\).
The Mehler multiplier and~\eqref{eq:decomposition-in-probability}
identify it pointwise with \(g_r(\xi(z))\); a countable dense set and
continuity give global equality.  Apply
\((\sqrt{-1}/\pi)\partial\bar\partial\) and subtract the finite sum
to obtain~\eqref{eq:Abel-spatial-series}.  Constants here may depend
on \(N,r\).

For~\textup{(iii)}, consider random forms \(A_j\) depending only on the
\(j\)-th independent Gaussian coefficient vector.  The pointwise wedge
bound and independence give
\[
\left\|\bigwedge_{j=1}^k A_j\right\|_{L^2(\Omega;\mathscr C^0)}
\leqslant C\prod_{j=1}^k
\|A_j\|_{L^2(\Omega_j;\mathscr C^0)}.
\]
By telescoping, the exterior products of partial sums
in~\eqref{eq:Abel-spatial-series} converge in
\(L^2(\Omega;\mathscr C^0(M))\).
Integration against either test is continuous in this norm, since
\(f\chi_U\) is bounded in type~\textup{(D)}.
With common cutoff \(L>n\), taking homogeneous components in
\eqref{eq:mixed-statistic-projection} identifies the integrated
product, with \(n<\alpha\leqslant L\)
in the first factor, as
\[
\left(\prod_{j=1}^k\mathsf U_{j,r}\right)
\left(\prod_{j=1}^k\mathsf P_{j,\leqslant L}\right)
\mathcal T_{N,n}^{\varphi,(1)}.
\]
Letting \(L\to\infty\), product convergence and the blockwise
decomposition yield
\begin{equation}\label{eq:Abel-mixed-statistic-identity}
\int_M[Z_{s_1^N}]_{r,>n}
\wedge\bigwedge_{j=2}^k[Z_{s_j^N}]_r\wedge\varphi
=\left(\prod_{j=1}^k\mathsf U_{j,r}\right)
\mathcal T_{N,n}^{\varphi,(1)}.
\end{equation}
Now~\eqref{eq:Abel-strong-convergence} gives the scalar \(L^2\) limit
as \(r\uparrow1\).
\end{proof}

\subsection{Localized projection identities}
\label{sec:localized-projection}

The following lemma localizes the projection identities and the
fixed-\(N\) Abel limit to a neighbourhood of the test support.
These identities are used in Lemma~\ref{lem:automatic-kernel-tail}
and for tests supported away from \(\mathcal P\) in
Section~\ref{sec:finite-conditioning}.

\begin{lem}
\label{lem:localized-projection-Abel}
Fix \(N\), and let \(s_1,\ldots,s_k\) be independent standard complex Gaussian
sections in finite-dimensional Hermitian subspaces of \(H^0(M,L^N)\).
Assume that all divisor intersections are almost surely proper and
that their covariance diagonals \(\mathsf b_j\) are positive on an open
set \(V\subset M\).  Fix a compact set \(K\Subset V\), and let \(\varphi\)
have either form in Definition~\ref{def:admissible-zero-statistics},
without its nondegeneracy requirements, with
\(\operatorname{supp}\varphi\subset K\) in type~\textup{(S)} and
\(\operatorname{supp}(f|_{\overline U})\subset K\) in type~\textup{(D)}.
For a domain test, assume also that, almost surely, the intersection
trace measure of every subcollection assigns zero mass to
\(\partial U\cap K\).

Under these local hypotheses, the projection identities and \(L^2\)
cutoff convergence of Proposition~\ref{prop:blockwise-chaos-projection}
and the fixed-\(N\) Abel conclusions of
Proposition~\ref{prop:Abel-chaos-approximation} hold for these sections,
using the normalization \(s_j/\sqrt{\mathsf b_j}\) on \(V\), integration
over \(V\), and spatial seminorms on \(K\).  All operators act in the
coefficient vectors of \(s_j\); spatial constants may depend on \(N,V,K\).
\end{lem}

\begin{proof}
We localize the conditional pairing argument and the scalar Abel
approximation near the test support.

Use the chaos, truncation and Abel notation above for
\(\widetilde s_j=s_j/\sqrt{\mathsf b_j}\) on \(V\), writing
\(\mathcal C_{\alpha,j}\) for the components.
Set \(Y^\varphi:=\int_V\bigwedge_{j=1}^k[Z_{s_j}]\wedge\varphi\).
For \(\mathbf q\in(\mathbb Z_{\geqslant0}\cup\{\infty\})^k\), the local
mixed projection identity is
\begin{equation}\label{eq:localized-mixed-projection}
Y_{\mathbf q}^{\varphi}
:=\int_V\bigwedge_{j=1}^k[Z_{s_j}]^{\leqslant q_j}\wedge\varphi
=\mathsf P_{\leqslant\mathbf q}Y^\varphi
\quad\text{in }L^2(\Omega).
\end{equation}

For a smooth form \(A\) and an integer \(\ell\geqslant0\), use the coefficient seminorms
\[
 p_{K,\ell}(A):=
 \max_{|\alpha|\leqslant\ell}\sup_{x\in K}|D^\alpha A(x)|
\]
on a fixed finite chart and frame cover near \(K\).
Here \(D^\alpha\) differentiates the local coefficients in real
coordinates, and the maximum also runs over the cover.

For fixed \(\mathbf q\), let \(W\) be the proper intersection of the
divisors with \(q_j=\infty\).  Its cohomology class determines its total
K\"ahler mass.  The comass
estimate~\eqref{eq:comass-domination-on-Y} bounds \(Y_{\mathbf q}^{\varphi}\)
by this mass times the \(p_{K,0}\)-seminorms of the truncated factors.
At fixed \(N\), the truncated factors are polynomials in finite-dimensional
Gaussian coefficient vectors, with coefficients smooth near \(K\).
Their seminorms therefore have finite moments of every order, proving
\(L^2\) membership for smooth \(\varphi\).
The conditional Poincar\'e--Lelong and Fubini argument in
Proposition~\ref{prop:blockwise-chaos-projection} is local near the test
support, so it proves~\eqref{eq:localized-mixed-projection} and the
joint homogeneous projection identities.

Choose an open set \(V'\), a compact set \(K_1\), and a smooth
cutoff \(\zeta\) such that
\[
 K\Subset V'\Subset K_1\Subset V,\qquad
 \zeta\in\mathscr C_c^\infty(V),\qquad
 \zeta=1\ \hbox{on }V',\qquad \operatorname{supp}\zeta\subset K_1.
\]
For the domain test \(\varphi=\varphi_U\), apply the smooth result on
\(K_1\) to \(\zeta\varphi_{U,\varepsilon}\), where
\(\varphi_{U,\varepsilon}=f h_\varepsilon\omega^{m-k}/(m-k)!\),
\(0\leqslant h_\varepsilon\leqslant1\) and
\(h_\varepsilon\to\chi_U\) off \(\partial U\).
Every such intersection \(W\) has trace measure zero on
\(\partial U\cap K\), while \(f=0\) on \(\partial U\setminus K\).
The comass bound with the truncated factors' \(\mathscr C^0(K_1)\)-norms
therefore gives convergence for each fixed \(\mathbf q\), almost surely
and in \(L^2\).  Projection contractivity passes the identities to \(\varphi_U\).

For \(n\in\mathbb Z_{\geqslant0}\), set
\(\mathcal T_n^{\varphi,(1)}=(I-\mathsf P_{1,\leqslant n})
(Y^\varphi-\mathbb EY^\varphi)\).
For the Abel sums with fixed \(0<r<1\), localize only the scalar functions
\(e_\alpha(\widetilde s_j)\) by \(\zeta\).
Derivatives in both variables of their covariance kernels on
\(\operatorname{supp}\zeta\) satisfy the fixed-\(N\) polynomial bounds
of~\eqref{eq:fixed-N-chaos-Sobolev}.  Sobolev embedding then gives
convergence in \(L^2(\Omega)\) for the seminorm \(p_{K,\ell}\).
Since \(\zeta=1\) near \(K\), the pairings are unchanged and involve
the original closed currents.  Independence, telescoping, and
\eqref{eq:localized-mixed-projection} identify the products as
in~\eqref{eq:Abel-mixed-statistic-identity}.  Strong Abel convergence gives
\begin{equation}\label{eq:localized-Abel-tail}
\int_V[Z_{s_1}]_{r,>n}\wedge
\bigwedge_{j=2}^k[Z_{s_j}]_r\wedge\varphi
=\left(\prod_{j=1}^k\mathsf U_{j,r}\right)\mathcal T_n^{\varphi,(1)}
\longrightarrow\mathcal T_n^{\varphi,(1)}
\quad\text{in }L^2(\Omega)
\end{equation}
as \(r\uparrow1\).
\end{proof}

\subsection{Derivative bounds and truncation errors}

The Abel second moment involves the following full and tail kernels.
For all sufficiently large \(N\) and \(0<r<1\), set
\begin{equation*}
Q_{N,r}:=\frac{1}{4\pi^2}\sum_{\alpha=1}^{\infty}
 \frac{r^\alpha}{\alpha^2}\tau_N^\alpha,
\qquad
A_{N,r}:=\mathscr D_{12}Q_{N,r},
\end{equation*}
and, for \(n\in\mathbb Z_{\geqslant0}\),
\begin{equation*}
Q_{N,r}^{>n}
:=\frac{1}{4\pi^2}\sum_{\alpha=n+1}^{\infty}
\frac{r^\alpha}{\alpha^2}\tau_N^\alpha.
\end{equation*}

These kernels are smooth; \(Q_{N,r},Q_{N,r}^{>n}\) are scalar, and
\(A_{N,r}\) is a \((2,2)\)-form on \(M^2\).

For the second-moment integrals, set
\(\Phi_\varphi:=\pi_1^*\varphi\wedge\pi_2^*\varphi\).
The derivatives below act on the indicated factor of \(M^2\).

\begin{lem}
\label{lem:Abel-two-jet-tail}
For \(q\geqslant1\) and \(0<r<1\), set
\[
E_{N,q,r}:=
\int_M[Z_{s_1^N}]_{r,>q}
\wedge\bigwedge_{j=2}^k[Z_{s_j^N}]_r\wedge\varphi.
\]
\begin{enumerate}
\item[\textup{(i)}] At each fixed \(N\) and \(0<r<1\),
\begin{equation}\label{eq:fixed-Abel-variance-expansion}
\begin{aligned}
\mathbb E E_{N,q,r}^2
&=\int_{M^2}\mathscr D_{12}Q_{N,r}^{>q}
 \wedge(D_N+A_{N,r})^{\wedge(k-1)}\wedge\Phi_\varphi\\
&=\sum_{t=0}^{k-1}\binom{k-1}{t}
 \int_{M^2}\mathscr D_{12}Q_{N,r}^{>q}
 \wedge A_{N,r}^{\wedge t}\wedge D_N^{\wedge(k-1-t)}
 \wedge\Phi_\varphi.
\end{aligned}
\end{equation}
The corresponding tail norm satisfies
\begin{equation}\label{eq:Abel-tail-variance-identity}
\|\mathcal T_{N,q}^{\varphi,(1)}\|_2^2
=\lim_{r\uparrow1}\mathbb E E_{N,q,r}^2.
\end{equation}
\item[\textup{(ii)}] Write \(\nabla_j\) for covariant differentiation in
the \(j\)-th factor of \(M^2\).  The notation \(\nabla^d\) denotes total covariant order \(d\).
If \(\sup_{S_\varphi^2}J_N\leqslant C\) uniformly in \(N\), then
\(\|\mathcal C_0^N\|_{L^\infty(S_\varphi)}\leqslant CN\), and
\begin{align}
|\nabla_1^a\nabla_2^b\tau_N|
&\leqslant C N^{(a+b)/2}J_N^2,
&&0\leqslant a,b\leqslant2,\quad 1\leqslant a+b\leqslant4,
\label{eq:weighted-tau-jets}\\
|\nabla\tau_N|
&\leqslant C\sqrt N\sqrt{g_N}
&&\text{on }S_\varphi^2.
\label{eq:weighted-gap-gradient}
\end{align}
\item[\textup{(iii)}] Under the same uniform bound for \(J_N\),
uniformly in \(0<r<1\), \(q\geqslant1\), and \(g_N>0\),
\begin{align}
|A_{N,r}|&\leqslant CN^2\frac{J_N}{g_N},
\label{eq:weighted-pointwise-A}\\
|Q_{N,r}^{>q}|&\leqslant Cq^{-1}J_N,
\label{eq:weighted-pointwise-scalar-tail}\\
|\bar\partial_1\bar\partial_2Q_{N,r}^{>q}|
&\leqslant C_\beta Nq^{-\beta}\frac{J_N}{g_N^\beta},
\qquad \beta>0 \text{ fixed}.
\label{eq:weighted-pointwise-boundary-tail}
\end{align}
The scalar bound~\eqref{eq:weighted-pointwise-scalar-tail} also holds
at \(g_N=0\).
\end{enumerate}
The constants in~\textup{(ii)}--\textup{(iii)} may depend on the uniform
bound for \(J_N\) and, in~\eqref{eq:weighted-pointwise-boundary-tail}, on
\(\beta\), but not on \(N,q,r\).
\end{lem}

\begin{proof}
For~\textup{(i)}, at fixed \(r\), Proposition~\ref{prop:Abel-chaos-approximation} and
Cauchy--Schwarz give convergence of the external products in
\(L^1(\Omega;\mathscr C^0(M^2))\).
Proposition~\ref{prop:finite-cutoff-two-point-functional}\textup{(ii)} yields
\[
\mathbb E[[Z_{s^N}]_r\boxtimes[Z_{s^N}]_r]=D_N+A_{N,r},
\qquad
\mathbb E[[Z_{s^N}]_{r,>q}\boxtimes[Z_{s^N}]_{r,>q}]
=\mathscr D_{12}Q_{N,r}^{>q}.
\]
At fixed \(N,r\), the differentiated series converge uniformly: the chain rule, applied twice in each variable, bounds their terms by
\(C_N\alpha^2r^\alpha\).
Independence and Fubini, including for the bounded type~\textup{(D)}
test, now give~\eqref{eq:fixed-Abel-variance-expansion}.
The convergence in~\eqref{eq:Abel-tail-statistic-convergence} proves
\eqref{eq:Abel-tail-variance-identity}.  Thus the subsequent applications
of Stokes' theorem use smooth forms before passage to the
\(L^2(\Omega)\) limit.

For~\textup{(ii)}, the fibre phases cancel in \(\tau_N\).  The horizontal Leibniz rule,
Lemma~\ref{lem:Bergman-horizontal-comparison}, and the fixed base frames
give~\eqref{eq:weighted-tau-jets}.
Lemma~\ref{lem:weighted-mean-current-bound} then gives the asserted
\(O(N)\) mean-current bound on \(S_\varphi\).

To obtain the gap-gradient bound, use the unit-vector representation
of the kernel.  For each \(p\in S_\varphi\), write
\(S_j^N=f_{j,p}^N e_{a,p}^{\otimes N}\) in a centred frame at \(p\),
and set
\[
v_p(z):=\frac{(f_{1,p}^N(z),\dots,f_{d_N,p}^N(z))}
                 {\bigl(\sum_j|f_{j,p}^N(z)|^2\bigr)^{1/2}}.
\]
Holding the centre and frame fixed, these are unit vectors and
\(R_N(z,w)=\langle v_p(z),v_q(w)\rangle\) in the corresponding
endpoint frames.
For a real unit tangent direction \(D\) at \(p\), let
\(a=v_p(p)\) and \(u=Dv_p(p)\).  Using the same \(p\)-centred frame
at both points, the mixed derivative with one real derivative in each
variable gives
\[
\|u\|^2=(D_zD_w R_N)(p,p)\leqslant CNJ_N(p,p)\leqslant CN.
\]
For \(q\in S_\varphi\), write \(b=v_q(q)\),
\(\rho=\langle a,b\rangle\), and
\[
u=u^\perp+\langle u,a\rangle a,\qquad
b=\overline\rho\,a+b^\perp,\qquad
\|b^\perp\|^2=1-|\rho|^2=g_N(p,q).
\]
Since \(\operatorname{Re}\langle u,a\rangle=0\), the component parallel
to \(a\) drops out, and differentiation of \(|\rho|^2\) gives
\[
|D_z\tau_N(p,q)|
\leqslant2|\rho|\,\|u^\perp\|\sqrt{g_N(p,q)}.
\]
Summing in both variables proves~\eqref{eq:weighted-gap-gradient},
including at \(g_N=0\), using only derivatives on \(S_\varphi^2\).

For~\textup{(iii)}, suppress the constant factor \(1/(4\pi^2)\)
in the scalar series and set, for \(0\leqslant u<1\),
\[
b_r(u)=\sum_{\alpha\geqslant1}\frac{r^\alpha u^\alpha}{\alpha^2},
\qquad
b_{q,r}(u)=\sum_{\alpha>q}\frac{r^\alpha u^\alpha}{\alpha^2}.
\]
Uniformly in \(0<r<1\),
\[
|b_r'(u)|\leqslant C(1+|\log(1-u)|),\qquad
|b_r^{(j)}(u)|\leqslant C_j(1-u)^{1-j},
\quad 2\leqslant j\leqslant4.
\]
On \(\tau_N\geqslant1/2\), apply
\eqref{eq:weighted-tau-jets}--\eqref{eq:weighted-gap-gradient}
to the four derivatives in \(\mathscr D_{12}b_r(\tau_N)\).
The terms with derivative distributions \(1+1+1+1\), \(2+1+1\),
and \(2+2\) among the \(\tau_N\)-factors are bounded by
\(CN^2g_N^{-1}\).  The remaining distributions \(3+1\) and \(4\)
are bounded by
\[
 CN^2g_N^{-1/2},\qquad CN^2(1+|\log g_N|),
\]
respectively.
Each term uses at most two derivatives in either variable, although the
total order is four.
Since \(J_N\geqslant\sqrt{\tau_N}\), this proves
\eqref{eq:weighted-pointwise-A} in that region.  On
\(\tau_N\leqslant1/2\), the derivatives of \(b_r\) are bounded and every
chain-rule term retains a factor \(J_N^2\); boundedness of \(J_N\) gives
the same estimate.  The elementary bound
\(b_{q,r}(u)\leqslant Cq^{-1}u\), together with
\(\tau_N\leqslant J_N^2\), proves
\eqref{eq:weighted-pointwise-scalar-tail}.

For the mixed second derivative of \(Q_{N,r}^{>q}\), use, on
\(1/2\leqslant u<1\),
\[
(1-u)|b_{q,r}''(u)|\leqslant Ce^{-q(1-u)},\qquad
|b_{q,r}'(u)|\leqslant2E_1(q(1-u)),\qquad
E_1(x):=\int_x^\infty e^{-s}\frac{\mathrm ds}{s}.
\]
For every \(x>0\) and fixed \(\beta>0\),
\(e^{-x}+E_1(x)\leqslant C_\beta x^{-\beta}\).
The \((1,0),(0,1),(1,1)\) cases of
\eqref{eq:weighted-tau-jets}, together with the gradient bound
\eqref{eq:weighted-gap-gradient}, give
\eqref{eq:weighted-pointwise-boundary-tail} on this region.
On \(u\leqslant1/2\), \(b_{q,r}'(u)\) and \(b_{q,r}''(u)\)
are \(O(e^{-cq})\).  After composition with \(\tau_N\), each term in
the chain rule contains a factor \(J_N^2\), which gives the same global bound.
\end{proof}

The pointwise bounds are uniform in the Abel parameter \(r\).
Integration by parts transfers all four outer derivatives to the
test forms in type~\textup{(S)}, and leaves one derivative on the
tail kernel at each variable in type~\textup{(D)}.  Integrating
the resulting bounds yields the \(L^2\) tail estimates below.

\begin{lem}
\label{lem:automatic-kernel-tail}
For each \(N\), let \(s_{1,N},\ldots,s_{k,N}\) be independent,
identically distributed centred circularly symmetric complex Gaussian
holomorphic sections.  Use the kernel
notation \(R_N,\tau_N,J_N\) of Section~\ref{sec:Bergman-setting}
for this ensemble, and set \(g_N=1-\tau_N\).
Let \(S_N\) be a
closed subset of \(M\) in type~\textup{(S)}, or of \(\overline U\) for
a fixed domain \(U\) in type~\textup{(D)}.  Assume that the normalized
fields are defined near \(S_N\) and that the projection and Abel
summation identities of Lemma~\ref{lem:localized-projection-Abel}
hold for tests supported there.  Use the scale \(a_N\)
in~\eqref{eq:kernel-integration-scales}, and write
\(\mu=\omega^m/m!\) in type~\textup{(S)} and
\(\mu=\mu_U\) in type~\textup{(D)}.

Let \(\varrho_N>0\).  In the following integrals, use \(g_N^0=1\);
for positive exponents the integrand is \(+\infty\) at \(g_N=0\).
Suppose, with constants uniform in \(N\), that
\begin{equation}\label{eq:uniform-test-row-input}
\begin{cases}
\displaystyle
 \sup_{S_N^2}J_N\leqslant C,\quad
 \sup_{z\in S_N}\int_{S_N}\frac{J_N(z,w)}{g_N(z,w)^{k-1}}
 \,\mathrm d\mu(w)\leqslant C\varrho_N,
 &\textup{(S)},\\[8pt]
\displaystyle
 \sup_{S_N^2}J_N\leqslant C,\quad
 \sup_{z\in S_N}\int_{S_N}
 \frac{J_N(z,w)}{g_N(z,w)^{k-1+\beta}}
 \,\mathrm d\mu(w)\leqslant C_\beta\varrho_N,
 &\textup{(D)},
\end{cases}
\end{equation}
where \(\beta>0\) in the last line.  For a test \(\varphi\), set
\[
 Y_N^\varphi:=\int_M[Z_{s_{1,N}}]\wedge\cdots\wedge
 [Z_{s_{k,N}}]\wedge\varphi,
 \qquad
 \mathcal T_{N,q}^{\varphi,(1)}:=(I-\mathsf P_{1,\leqslant q})
 (Y_N^\varphi-\mathbb EY_N^\varphi).
\]
Then every real \(\mathscr C^3\) \((m-k,m-k)\)-form \(\varphi\) supported
in \(S_N\) satisfies
\begin{equation*}
 \|\mathcal T_{N,q}^{\varphi,(1)}\|_2^2
 \leqslant C\|\varphi\|_{\mathscr C^2(M)}^2
 a_N^2\varrho_Nq^{-1},
 \qquad q\geqslant1.
\end{equation*}
For
\(\varphi_U=f\chi_U\omega^{m-k}/(m-k)!\), where
\(f\in\mathscr C^3(M,\mathbb R)\) and
\(\operatorname{supp}(f|_{\overline U})\subset S_N\),
\begin{equation*}
 \|\mathcal T_{N,q}^{\varphi_U,(1)}\|_2^2
 \leqslant C_\beta\|f\|_{\mathscr C^1(M)}^2
 a_N^2\varrho_Nq^{-\beta},
 \qquad q\geqslant1.
\end{equation*}
The constants may depend on \(m,k\), the fixed geometry, the constants
in the assumed kernel bounds, and, in type~\textup{(D)},
on \(U\) and \(\beta\).  They are independent of the test, \(N\), and \(q\), with
test dependence appearing only through the displayed norms.  The
measure of \(S_N\) is bounded by the fixed ambient volume or by
\(\mu_U(\overline U)\).
\end{lem}

\begin{proof}
For the present covariance kernel, the assumed local projection and Abel
identities give~\eqref{eq:fixed-Abel-variance-expansion} and
\eqref{eq:Abel-tail-variance-identity}.  The local covariance calculations
in Lemma~\ref{lem:Abel-two-jet-tail}\textup{(ii)}--\textup{(iii)} give
the derivative bounds for this kernel on \(S_N^2\).
The diagonal bound \(J_N(z,z)\leqslant C\) on \(S_N\) and
Lemma~\ref{lem:weighted-mean-current-bound}, applied to the local
covariance kernel of this ensemble, give
\(\|\mathcal C_0^N\|_{L^\infty(S_N)}\leqslant CN\).
Apply~\eqref{eq:fixed-Abel-variance-expansion} to the test \(\varphi\), with
\(D_N\), \(A_{N,r}\), and \(Q_{N,r}^{>q}\) defined from
the common covariance kernel.  The factors \(D_N,A_{N,r}\) are separately
\(\partial_j\)- and \(\bar\partial_j\)-closed, \(j=1,2\), where the kernel
is defined.  To apply Stokes, multiply
\(Q_{N,r}^{>q} A_{N,r}^{\wedge t}\wedge D_N^{\wedge(k-1-t)}\) by the
product of the two pullbacks of a smooth cutoff supported in the kernel
neighbourhood and equal to one near \(S_N\), and extend the product by
zero.  The test form and its derivatives in type~\textup{(S)} are supported
in \(S_N\).  In type~\textup{(D)}, \(f\) and \(\partial f\) restricted to
\(\overline U\) are supported there, since \(f\) vanishes on
\(\overline U\setminus S_N\) and its derivatives extend continuously
from \(U\) to \(\partial U\).  Thus the differentiated cutoff factors
vanish on the relevant supports.

For a smooth test, Stokes transfers \(\mathscr D_{12}\) entirely to
the two test forms.
Equations~\eqref{eq:weighted-pointwise-A} and
\eqref{eq:weighted-pointwise-scalar-tail} bound the result by
\[
 Cq^{-1}\|\varphi\|_{\mathscr C^2}^2N^{2k-2}
 \int_{S_N^2}\frac{J_N}{g_N^{k-1}}\,\mathrm d\mu^{\otimes2}.
\]
For the domain test \(\varphi=\varphi_U\),
Lemma~\ref{lem:boundary-Stokes}\textup{(ii)} leaves at each
variable only \(f\) or \(\partial f\), and
\eqref{eq:weighted-pointwise-boundary-tail} bounds the four resulting
integrals over \(U^2\), \(U\times\partial U\),
\(\partial U\times U\), and \((\partial U)^2\) by
\[
 C_\beta q^{-\beta}\|f\|_{\mathscr C^1}^2N^{2k-1}
 \int_{S_N^2}\frac{J_N}{g_N^{k-1+\beta}}
 \,\mathrm d\mu^{\otimes2}.
\]
Bounded \(J_N\) absorbs extra powers, and \(g_N\leqslant1\) permits
replacing \(t\) by \(k-1\).  For positive denominator exponents, the
assumed bounds for the integral in the second variable make
\(\{w:g_N(z,w)=0\}\) \(\mu\)-null for every
\(z\in S_N\).  For \(k=1\) in type~\textup{(S)}, only \(t=0\)
occurs: no \(A_{N,r}\) factor is present, and the scalar tail bound
also holds at \(g_N=0\).
Each double integral is bounded by the supremum of the integral in
the second variable times \(\mu(S_N)\);
the powers \(N^{2k-2},N^{2k-1}\) equal the respective \(a_N^2\).
The bounds are uniform in \(0<r<1\).  Sum over \(t\) and let
\(r\uparrow1\) at fixed \(N,q\) in~\eqref{eq:Abel-tail-variance-identity}.
\end{proof}

\subsection{Sublevel-set bounds and proof of the correlation-volume criterion}

The pointwise and integral bounds for \(J_N\), together with the
sublevel-set bound, control the weighted negative moments of \(g_N\)
over \(S_\varphi\).  Lemma~\ref{lem:automatic-kernel-tail} bounds the
squared \(L^2\) tail for one coefficient vector by \(a_N^2\varrho_N\)
times a factor that tends to zero with \(q\), uniformly in \(N\).
The variance lower bound and Lemma~\ref{lem:common-cutoff-removal}
give a uniform comparison of the full and truncated statistics, each
standardized by its own standard deviation.  We can then take
\(N\to\infty\) at fixed sufficiently large \(q\) before letting
\(q\to\infty\).

\begin{lem}
\label{lem:cv-negative-moments}
Assume the first two bounds in~\eqref{eq:cv-basic-bounds} and
\eqref{eq:cv-small-correlation-balls}, with fixed \(\nu>0\) and
\(t_0\in(0,1)\).  Then, for every \(0\leqslant p<\nu\),
\[
\sup_{z\in S_\varphi}\int_{S_\varphi}
 \frac{J_N(z,w)}{g_N(z,w)^p}\,\mathrm d\mu_\varphi(w)
\leqslant C_p\varrho_N.
\]
Moreover, \(\{w:g_N(z,w)=0\}\) is \(\mu_\varphi\)-null for every
\(z\in S_\varphi\).  The constants are uniform in \(N,z\).
\end{lem}

\begin{proof}
Fix \(z\).  Letting \(t\downarrow0\) in
\eqref{eq:cv-small-correlation-balls} proves the null-set assertion.
The set \(\{g_N>t_0\}\) contributes at most \(Ct_0^{-p}\varrho_N\).
On the remaining set, use \(J_N\leqslant C\) and the shells
\(2^{-j-1}t_0<g_N\leqslant2^{-j}t_0\):
\[
\int_{\{0<g_N\leqslant t_0\}}J_Ng_N^{-p}
 \,\mathrm d\mu_\varphi
\leqslant C\varrho_N\sum_{j\geqslant0}
 (2^{j+1}/t_0)^p(2^{-j}t_0)^\nu.
\]
The geometric series converges for \(p<\nu\), uniformly in
\(z\in S_\varphi\); \(g_N^0=1\) when \(p=0\).
\end{proof}

\begin{lem}
\label{lem:common-cutoff-removal}
Assume~\eqref{eq:cv-basic-bounds}, with
\(0<\varrho_N\to0\), and, for a sequence \(\varepsilon_q\downarrow0\),
\begin{equation}\label{eq:abstract-single-block-tail}
\|\mathcal T_{N,q}^{\varphi,(1)}\|_2^2
\leqslant Ca_N^2\varrho_N\varepsilon_q,
\qquad q\geqslant1,
\end{equation}
uniformly for all sufficiently large \(N\).
The full statistic then satisfies
\[
\|\widehat X_N^\varphi\|_2^2\asymp a_N^2\varrho_N,\qquad
\widehat X_N^\varphi/\|\widehat X_N^\varphi\|_2
\xrightarrow{\ d\ }\mathcal N_{\mathbb R}(0,1).
\]
\end{lem}

\begin{proof}
Write \(\sigma_N=\|\widehat X_N^\varphi\|_2\) and
\(\sigma_{N,q}=\|\widehat X_{N,q}^\varphi\|_2\).
Because \(\widehat X_{N,q}^\varphi=(\prod_{j=1}^k\mathsf P_{j,\leqslant q})\widehat X_N^\varphi\)
is an orthogonal projection of \(\widehat X_N^\varphi\),
\[
\sigma_N^2-\sigma_{N,q}^2=\|\widehat X_N^\varphi-\widehat X_{N,q}^\varphi\|_2^2.
\]
By~\eqref{eq:common-cutoff-error} and
\eqref{eq:abstract-single-block-tail}, choose a fixed \(q_0\),
independent of \(N\), such that
\(kC\varepsilon_{q_0}\leqslant c/2\).
Then, for every fixed \(q\geqslant q_0\),
\(\sigma_{N,q}^2\geqslant(c/2)a_N^2\varrho_N\).
At \(q=1\), combine~\eqref{eq:finite-cutoff-row-variance},
\eqref{eq:common-cutoff-error}, and~\eqref{eq:abstract-single-block-tail}:
\[
\sigma_N^2
=\sigma_{N,1}^2+\|\widehat X_N^\varphi-\widehat X_{N,1}^\varphi\|_2^2
\leqslant Ca_N^2\varrho_N.
\]
Together with the assumed lower bound, this proves the variance
comparison.

For every fixed \(q\geqslant q_0\), Proposition~\ref{prop:finite-cutoff-clt} gives
\[
\widehat X_{N,q}^\varphi/\sigma_{N,q}\xrightarrow{\ d\ }\mathcal N_{\mathbb R}(0,1).
\]

For every fixed \(q\geqslant q_0\) and all sufficiently large \(N\),
both standard deviations are positive.  Orthogonality gives
\begin{equation*}
\begin{aligned}
\left\|\frac{\widehat X_N^\varphi}{\sigma_N}
-\frac{\widehat X_{N,q}^\varphi}{\sigma_{N,q}}\right\|_2^2
&=2\left(1-\frac{\sigma_{N,q}}{\sigma_N}\right)\\
&=\frac{2(\sigma_N^2-\sigma_{N,q}^2)}
 {\sigma_N(\sigma_N+\sigma_{N,q})}\\
&\leqslant\frac{2\|\widehat X_N^\varphi-\widehat X_{N,q}^\varphi\|_2^2}{\sigma_N^2}
\leqslant\frac{2kC}{c}\varepsilon_q.
\end{aligned}
\end{equation*}
For each \(t\in\mathbb R\), the inequality
\(|e^{ix}-e^{iy}|\leqslant|x-y|\) and Cauchy--Schwarz give
\[
 \lim_{q\to\infty}\limsup_{N\to\infty}
 \left|\mathbb E e^{it\widehat X_N^\varphi/\sigma_N}
       -\mathbb E e^{it\widehat X_{N,q}^\varphi/\sigma_{N,q}}\right|
 \leqslant\lim_{q\to\infty}|t|
       \sqrt{\frac{2kC}{c}\varepsilon_q}=0.
\]
For every fixed \(q\geqslant q_0\), the second characteristic function
converges to \(e^{-t^2/2}\).  Taking \(N\to\infty\) first and then
\(q\to\infty\) proves the CLT by L\'evy's continuity theorem.
\end{proof}

\begin{proof}[Proof of Proposition~\ref{thm:correlation-volume-criterion}]
Lemma~\ref{lem:cv-negative-moments} with \(p=k-1<\nu\) gives
\eqref{eq:uniform-test-row-input} for type~\textup{(S)}.
For type~\textup{(D)}, fix \(0<\beta<\nu-k+1\) and use
\(p=k-1+\beta<\nu\).
The projection and Abel identities required in
Lemma~\ref{lem:automatic-kernel-tail} hold by
Propositions~\ref{prop:blockwise-chaos-projection}
and~\ref{prop:Abel-chaos-approximation}.  Applying the lemma with
\(S_N=S_\varphi\) and absorbing the fixed test norm gives
\begin{equation*}
\|\mathcal T_{N,q}^{\varphi,(1)}\|_2^2
\leqslant C_\beta a_N^2\varrho_N
\begin{cases}
q^{-1},&\textup{(S)},\\
q^{-\beta},&\textup{(D)},
\end{cases}
\qquad q\geqslant1,
\end{equation*}
uniformly for all sufficiently large \(N\).  In type~\textup{(S)} the
constant is independent of \(\beta\).

Apply Lemma~\ref{lem:common-cutoff-removal} with
\(\varepsilon_q=q^{-1}\) in type~\textup{(S)} and
\(\varepsilon_q=q^{-\beta}\) in type~\textup{(D)};
its remaining hypotheses are~\eqref{eq:cv-basic-bounds}.
\end{proof}

\section{Variance lower bounds from the Bargmann--Fock limit}
\label{sec:two-point-nondegeneracy}

Assume \textup{(BF)}--\textup{(OD)} of
Theorem~\ref{thm:two-estimate-BF-criterion} and fix an admissible test
\(\varphi\).  We treat smooth statistics with \(2\leqslant k\leqslant m\)
and weighted-domain statistics with \(1\leqslant k\leqslant m\).
The estimates below verify the kernel hypotheses of
Proposition~\ref{thm:correlation-volume-criterion}.  For each test,
a positive Bargmann--Fock limit for one joint chaos component
supplies the variance lower bound.

\subsection{Variance scales and selected chaos components}

The correlation length \(N^{-1/2}\) gives integration factors of orders
\(N^{-m}\) in the bulk and \(N^{-m+1/2}\) on the boundary.  After Stokes,
the scalar covariance factor in the boundary main term retains one
derivative at each endpoint, contributing an additional factor \(N\)
compared with the smooth pairing.
With \(a_N\) as in~\eqref{eq:kernel-integration-scales}, the scales are:
\begin{equation*}
\begin{array}{c|c|c|c}
 &d_\varphi&r_N=N^{-d_\varphi/2}&v_N=a_N^2r_N\\ \hline
\textup{(S)}&2m&N^{-m}&N^{2k-m-2}\\
\textup{(D)}&2m-1&N^{-(m-\frac12)}&N^{2k-m-\frac12}.
\end{array}
\end{equation*}

Set \(q_N:=Q_N^{\leqslant1}=\tau_N/(4\pi^2)\).
For \(1\leqslant l\leqslant k\), set
\begin{equation}\label{eq:W-terms}
\mathcal{W}_l^{N}(\varphi) :=
\int_{M^2}
\bigl[\mathscr D_{12}q_N\bigr]^{\wedge(k-l+1)}
\wedge \bigl[ \pi_1^*\mathcal{C}^N_0 \wedge \pi_2^*\mathcal{C}^N_0 \bigr]^{\wedge (l-1)}
\wedge \pi_1^*\varphi \wedge \pi_2^*\varphi.
\end{equation}
Proposition~\ref{prop:variance-lower-bound} identifies
\(\mathcal W_l^N(\varphi)\) as the variance of the joint component with
chaos order two in the first \(k-l+1\) coefficient vectors and order
zero in the others; its total chaos order is \(2(k-l+1)\).

\begin{pro}
\label{prop:natural-scale-variance-lower-bound}
In type~\textup{(S)}, let \(l_*\) be the least
\(l\in\{1,\dots,k\}\) with
\(\omega^l\wedge\sqrt{-1}\,\partial\bar\partial\varphi\equiv0\);
in type~\textup{(D)}, set \(l_*=k\).
The component variances in~\eqref{eq:W-terms} satisfy
\begin{equation}\label{eq:selected-component-positive-limit}
\lim_{N\to\infty}v_N^{-1}\mathcal W_{l_*}^N(\varphi)\in(0,\infty).
\end{equation}
Consequently, there are constants \(c_\varphi>0\) and \(N_\varphi\) such that
\[
\operatorname{Var}\widehat X_{N,\mathbf q}^\varphi
\geqslant c_\varphi v_N,\qquad N\geqslant N_\varphi,
\]
simultaneously for all
\(q_1,\dots,q_k\in\mathbb Z_{\geqslant1}\cup\{+\infty\}\).
\end{pro}

\subsection{Analytic consequences of \textup{(BF)}--\textup{(OD)}}

The first-derivative bound in \textup{(OD)} strengthens the \(L^2\)
convergence in \textup{(BF)} to uniform convergence on compact scaled
sets.  A holomorphic quotient and Cauchy estimates then give convergence
of the derivatives of the squared kernel modulus.  We use these local
limits, the off-diagonal decay and the volume bounds to compute the
component-variance limits and verify the kernel hypotheses of the
correlation-volume criterion.

\begin{lem}
\label{lem:weak-BF-jet-recovery}
Set \(\Omega_T:=D_T\times D_T\).
For every integer \(r\geqslant0\) and every fixed \(T>0\),
\begin{equation}\label{eq:weak-BF-tau-jets}
\sup_{x\in M}
\bigl\||\mathcal R_{N,x}(u,v)|^2-e^{-|u-v|^2}\bigr\|_{C^r(\Omega_T)}
\longrightarrow0.
\end{equation}
Here the \(C^r\)-norm contains all ordinary derivatives.
\end{lem}

\begin{proof}
Fix \(T<T_1<T_2\).  We first obtain a uniform Lipschitz bound on
\(\Omega_{T_2}\), then recover all derivatives of
\(|\mathcal R_{N,x}|^2\) through a holomorphic quotient.  By
Lemma~\ref{lem:Bergman-horizontal-comparison} and the full \(J_N\)-bound
in~\textup{(OD)}, every first horizontal derivative of \(\rho_N\), after
division by \(\sqrt N\), is uniformly bounded.  In the circle coordinates
associated with the common preferred frame at \(x\), a base-coordinate
field has the form
\[
D_\mu=D_\mu^{\mathrm H}-a_{\mu,x}\partial_\theta,
\qquad |a_{\mu,x}(z)|\leqslant C|z|,
\]
up to a uniformly bounded change of base-coordinate fields.  Here
\(D_\mu^{\mathrm H}\) is horizontal; the estimate for
\(a_{\mu,x}\) follows uniformly from the preferred-frame normalization.
Since \(\partial_\theta\rho_N=\pm iN\rho_N\) at the two endpoints and
\(|\rho_N|\leqslant1\), for \(|u|,|v|\leqslant T_2\), the identity
\eqref{eq:Bergman-Szego-bridge} gives
\[
N^{-1/2}|D_\mu\rho_N|
\leqslant N^{-1/2}|D_\mu^{\mathrm H}\rho_N|
     +\sqrt N\,|a_{\mu,x}|\,|\rho_N|
\leqslant C_{T_2},
\]
because \(z=u/\sqrt N\) or \(v/\sqrt N\).  Thus
\(\mathcal R_{N,x}\), and hence
\(f_{N,x}:=\mathcal R_{N,x}-\mathcal R_{\mathrm{BF}}\), is uniformly
Lipschitz on \(\Omega_{T_2}\), independently of \(x\) and \(N\).

A value of modulus \(a\) forces modulus at least \(a/2\) on a ball
of radius proportional to \(a\).  Since the real dimension is \(4m\),
the squared \(L^2\)-norm is at least \(ca^{4m+2}\).
Thus \(\overline{\Omega_{T_1}}\subset\Omega_{T_2}\) gives, for large \(N\),
\[
\sup_{x\in M}\|f_{N,x}\|_{L^\infty(\Omega_{T_1})}
\leqslant C_{T_1,T_2}
\left(
\sup_{x\in M}\|f_{N,x}\|_{L^2(\Omega_{T_2})}
\right)^{1/(2m+1)}.
\]
Consequently, \(\mathcal R_{N,x}\to\mathcal R_{\mathrm{BF}}\) uniformly
on \(\Omega_{T_1}\), uniformly in \(x\).

By uniform convergence, both \(\mathcal R_{N,x}(u,0)\) and
\(\mathcal R_{N,x}(0,v)\) have modulus at least
\(\tfrac12e^{-T_1^2/2}\), uniformly in \(x\) and
\(|u|,|v|\leqslant T_1\), for large \(N\).  These bounds allow us to form
the quotient
\[
A_{N,x}(u,v):=
\frac{\mathcal R_{N,x}(u,v)}
 {\mathcal R_{N,x}(u,0)\mathcal R_{N,x}(0,v)}.
\]
Write \(B_{N,x}\) for the unnormalized Bergman coefficient in the
same preferred frame.  The quotient cancels the nonholomorphic diagonal
normalizing factors:
\[
A_{N,x}(u,v)=
\frac{B_{N,x}(u/\sqrt N,v/\sqrt N)B_{N,x}(0,0)}
 {B_{N,x}(u/\sqrt N,0)B_{N,x}(0,v/\sqrt N)}.
\]
Hence \(A_{N,x}\) is holomorphic in \(u\) and antiholomorphic in \(v\),
and the uniform convergence already proved gives
\[
A_{N,x}(u,v)\longrightarrow e^{u\cdot\bar v}
\quad\text{uniformly on }\Omega_{T_1}.
\]
Cauchy estimates in \((u,\bar v)\) on the smaller set \(\Omega_T\) give
uniform convergence of every derivative, with constants independent of
\(x\).  To recover the squared modulus of the original normalized kernel, use
\[
A_{N,x}(u,u)=|\mathcal R_{N,x}(u,0)|^{-2}>0,
\qquad
|\mathcal R_{N,x}(u,v)|^2
=\frac{|A_{N,x}(u,v)|^2}
 {A_{N,x}(u,u)A_{N,x}(v,v)}.
\]
Applying the product and quotient rules proves
\eqref{eq:weak-BF-tau-jets}, since the limiting quotient is
\(e^{-|u-v|^2}\).
\end{proof}

\begin{lem}
\label{lem:C2-boundary-geometry}
Let \(U\subsetneq M\) have \(\mathscr C^2\) boundary.  There are
constants \(C_U,\rho_0>0\) such that
\begin{equation}\label{eq:boundary-volume-upper}
\sigma_{\partial U}(\partial U\cap B(x,\rho))
 \leqslant C_U\rho^{2m-1},
 \qquad x\in M,\quad 0<\rho<\rho_0.
\end{equation}
At every \(x\in\partial U\), holomorphic normal coordinates centred at
\(x\) exist on a ball of uniform radius and satisfy
\begin{equation}\label{eq:uniform-normal-distance}
\operatorname{dist}(x,z)=|z|+O_U(|z|^2).
\end{equation}
After a unitary change of coordinates, the boundary and the \(U\)-side
are given by \(\operatorname{Im}z_1+\psi_x(z)=0\) and
\(\operatorname{Im}z_1+\psi_x(z)>0\), respectively, where
\begin{equation}\label{eq:uniform-boundary-graph-bounds}
\psi_x(0)=0,\quad \mathrm d\psi_x(0)=0,\quad
|\psi_x(z)|\leqslant C_U|z|^2,\quad
|\mathrm D\psi_x(z)|\leqslant C_U|z|,\quad
|\mathrm D^2\psi_x(z)|\leqslant C_U.
\end{equation}
Here \(\mathrm D\) and \(\mathrm D^2\) denote the first and second
real differentials.  All radii and constants are independent of \(x\).
\end{lem}

\begin{proof}
The compact \(\mathscr C^2\) hypersurface \(\partial U\) has a finite
cover by graph charts with uniformly bounded \(\mathscr C^2\) norms.
The area formula and uniform metric comparison give
\eqref{eq:boundary-volume-upper} for centres on \(\partial U\).
For an arbitrary centre \(x\in M\), a nonempty intersection
\(\partial U\cap B(x,\rho)\) lies in \(\partial U\cap B(y,2\rho)\)
for some \(y\) in that intersection.  Decreasing \(\rho_0\) and
enlarging \(C_U\) give the bound for all centres.  In particular,
\(\partial U\) has zero ambient volume.

Compactness of \(M\) gives the uniform normal-coordinate radius and
\eqref{eq:uniform-normal-distance}.  A unitary rotation aligns the
tangent hyperplane with \(\{\operatorname{Im}z_1=0\}\); the uniform
implicit-function theorem gives the indicated graph and side on a
smaller fixed ball.  Taylor's formula and the common \(\mathscr C^2\)
bound give~\eqref{eq:uniform-boundary-graph-bounds}.
\end{proof}

Set \(\mathfrak h_N(z,w)=(1+\sqrt N\,\operatorname{dist}(z,w))^{-A}\), where \(A>2m\)
is the exponent in~\textup{(OD)}.

\begin{lem}
\label{lem:BF-OD-consequences}
For all sufficiently large \(N\),
\begin{align}
\sup_{z,w\in S_\varphi}J_N(z,w)&\leqslant C,
\label{eq:BF-uniform-jet}\\
\sup_{z\in S_\varphi}\int_{S_\varphi}J_N(z,w)
 \,\mathrm d\mu_\varphi(w)&\leqslant Cr_N.
\label{eq:BF-ordinary-row}
\end{align}
The mean current \(\mathcal C_0^N\) and the function
\(g_N=1-|R_N|^2\) satisfy
\begin{equation}\label{eq:BF-mean-limit}
N^{-1}\mathcal C_0^N\longrightarrow\omega/\pi
\quad\text{uniformly on }M,
\end{equation}
\begin{equation}\label{eq:BF-correlation-gap}
g_N(z,w)\geqslant c\min\{N \operatorname{dist}(z,w)^2,1\},
\end{equation}
and, consequently, for some fixed \(t_0>0\) and all \(0<t\leqslant t_0\),
\begin{equation}\label{eq:BF-correlation-volume}
\sup_{z\in S_\varphi}
\mu_\varphi\{w\in S_\varphi:g_N(z,w)\leqslant t\}
\leqslant Cr_Nt^{d_\varphi/2}.
\end{equation}
The function \(q_N\) also satisfies
\begin{equation}\label{eq:BF-q-jet-tail}
|\nabla^j q_N(z,w)|\leqslant C N^{j/2}\mathfrak h_N(z,w)
\quad(0\leqslant j\leqslant2),\qquad
|\mathscr D_{12}q_N(z,w)|\leqslant C N^2\mathfrak h_N(z,w).
\end{equation}
The operator \(\mathscr D_{12}\) has order two at each endpoint.
Finally, set
\(\mathrm d\mu_{2m}=\omega^m/m!\) and
\(\mathrm d\mu_{2m-1}=\mathrm d\sigma_{\partial U}\).  For
\(D\in\{2m,2m-1\}\) and \(L\geqslant1\),
\begin{equation}\label{eq:BF-absolute-tail}
\begin{aligned}
N^{D/2}\sup_{z\in M}
\int_{\{\sqrt N \operatorname{dist}(z,w)>L\}}\mathfrak h_N(z,w)\,\mathrm d\mu_D(w)
&\leqslant C\varepsilon_D(L),\\
\varepsilon_D(L)&:=\int_L^\infty s^{D-1}(1+s)^{-A}\,\mathrm ds
\longrightarrow0.
\end{aligned}
\end{equation}
For the surface measure the estimate includes interior base points.
\end{lem}

\begin{proof}
The diagonal identity~\eqref{eq:weighted-mean-from-kernel} and the
derivative convergence in
Lemma~\ref{lem:weak-BF-jet-recovery} give
\eqref{eq:BF-mean-limit}: in unitary normal coordinates the squared
modulus of the limiting kernel is \(e^{-|u-v|^2}\).  Its mixed derivatives
\(\partial_{u_a}\partial_{\bar u_b}\), evaluated at \(u=v\), give
\(-\delta_{ab}\), and~\eqref{eq:weighted-mean-from-kernel} gives
\(\omega/\pi\).

Condition~\textup{(OD)} gives~\eqref{eq:BF-uniform-jet}.  For
\(0<\rho<\min\{\rho_0,1\}\), the volume bound and
\eqref{eq:boundary-volume-upper} give, in type~\textup{(D)},
\[
\mu_U(B(z,\rho))
\leqslant C\bigl(\rho^{2m}+\rho^{2m-1}\bigr)
\leqslant C\rho^{2m-1},\qquad z\in M.
\]
Thus both test classes satisfy
\(\mu_\varphi(S_\varphi\cap B(z,\rho))\leqslant C\rho^{d_\varphi}\)
uniformly in \(z\).  Finite total mass extends the bound to all
\(\rho>0\) after enlarging the constant.
Integration over dyadic annuli proves~\eqref{eq:BF-ordinary-row} and
\eqref{eq:BF-absolute-tail}.  More precisely, for \(D\in\{2m,2m-1\}\),
\[
\int_{\{\sqrt N\,\operatorname{dist}(z,w)>L\}}\mathfrak h_N\,\mathrm d\mu_D
\leqslant CN^{-D/2}\sum_{j\geqslant0}
 (2^jL)^D(1+2^jL)^{-A}
\leqslant CN^{-D/2}\varepsilon_D(L).
\]
This uses \(A>D\); \eqref{eq:boundary-volume-upper} applies to every
centre \(z\in M\), including \(z\in U\).

To prove the correlation gap, first work near the diagonal, where
\(g_N\) and its first derivatives vanish.
Lemma~\ref{lem:weak-BF-jet-recovery} makes the real Hessian in \(v\)
of \(1-|\mathcal R_{N,z}(0,v)|^2\) uniformly positive definite for
\(|v|\leqslant c_0\), with \(c_0>0\) fixed.  Taylor's formula gives
\(g_N(z,w)\geqslant cN \operatorname{dist}(z,w)^2\) in this region.
On every fixed scaled annulus, the same lemma bounds \(g_N\) below by
a positive constant.  Finally, choose the outer radius so large
that~\textup{(OD)} gives \(|R_N|\leqslant1/2\) beyond it.  These three
regions cover all pairs, and uniform coordinate--metric comparison
gives~\eqref{eq:BF-correlation-gap} with scaled radii and constants
independent of \(N\).  Fix \(t_0<c\), where \(c\) is the constant in that bound, and
then take \(N\) sufficiently large for these choices.
The set in~\eqref{eq:BF-correlation-volume} is therefore contained in a
ball of radius \(C\sqrt{t/N}\), and the preceding volume bound proves the
claim.  Here
\(d_\varphi/2=m>k-1\) in type~\textup{(S)} and
\(d_{\varphi_U}/2=m-1/2>k-1\) in type~\textup{(D)}.

Finally, differentiate \(q_N=R_N\overline{R_N}/(4\pi^2)\) in the
centred endpoint frames.  The product rule, the comparison in
Lemma~\ref{lem:Bergman-horizontal-comparison}, and \(\mathfrak h_N\leqslant1\)
give~\eqref{eq:BF-q-jet-tail} from~\textup{(OD)}.
\end{proof}

\subsection{A joint chaos component and its local scaling limit}

For each \(1\leqslant l\leqslant k\), every cutoff with \(q_j\geqslant1\)
in all coefficient vectors retains the joint component of variance
\(\mathcal W_l^N(\varphi)\).  Orthogonality gives the variance comparison,
and the rescaled kernel formulas provide the local models for the two
test classes.

\begin{pro}
\label{prop:variance-lower-bound}
For all sufficiently large \(N\), every \(1\leqslant l\leqslant k\), and
all \(q_1,\dots,q_k\in\mathbb Z_{\geqslant1}\cup\{+\infty\}\),
\[
\operatorname{Var}\widehat X_{N,\mathbf q}^\varphi
\geqslant\mathcal W_l^N(\varphi).
\]
The threshold is independent of \(l\) and the cutoffs.
\end{pro}
\begin{proof}
Set \(a=k-l+1\), and let \(\vec1_a\in\{0,1\}^k\) have its first
\(a\) entries equal to one and all remaining entries equal to zero.
The corresponding statistic and joint projection are
\[
Y:=\int_M\bigwedge_{j=1}^a\mathcal C_1^{N,j}
 \wedge\bigwedge_{i=a+1}^k\mathcal C_0^{N,i}\wedge\varphi,
\qquad
\mathsf P_{=\vec1_a}=
 \left(\prod_{j=1}^a\mathsf P_{j,=1}\right)
 \left(\prod_{i=a+1}^k\mathsf P_{i,=0}\right).
\]
Every indicated truncation retains \(\vec1_a\), and \(a\geqslant1\)
makes its projection annihilate constants.  Thus
Proposition~\ref{prop:blockwise-chaos-projection} gives
\(Y=\mathsf P_{=\vec1_a}\widehat X_{N,\mathbf q}^\varphi\), also for infinite
cutoffs, and orthogonality bounds the variance below by \(\mathbb E[Y^2]\).

The integrand defining \(Y\) has coefficients polynomial in the
Gaussian vectors, so independence and Fubini allow us to pair the
currents one section at a time.
Proposition~\ref{prop:finite-cutoff-two-point-functional}\textup{(ii)}
assigns \(\mathscr D_{12}q_N\) to each paired \(\mathcal C_1\).
The resulting integral is
\(\mathbb E[Y^2]=\mathcal W_l^N(\varphi)\), by~\eqref{eq:W-terms}.
\end{proof}

\subsubsection{Mixed-derivative expansion}

Fix a holomorphic normal coordinate chart
\(\kappa:U_{z_0}\to V_{z_0}\subset\mathbb C^m\), with
\(\kappa(z_0)=0\), and write \(z,w\) for the coordinates on its two
copies.  The map from rescaled coordinates to \(M\) is
\[
\delta_{N,z_0}:\sqrt N V_{z_0}\longrightarrow U_{z_0},\qquad
\delta_{N,z_0}(v):=\kappa^{-1}(v/\sqrt N).
\]
We scale the integration variable while retaining the covectors at
the base point.  For a differential form \(\eta\) on \(U_{z_0}\times U_{z_0}\),
\(\operatorname{Sc}_{N,z_0}\eta\) means pullback in the second variable
by \(\delta_{N,z_0}\), followed by evaluation of the coefficients at
\(z=0\); the first-variable covectors remain at \(z_0\).
Explicitly, extend linearly the rule
\[
\begin{aligned}
&\operatorname{Sc}_{N,z_0}
 \bigl[a(z,w)\,\mathrm dz_I\wedge\mathrm d\bar z_J
       \wedge\mathrm dw_K\wedge\mathrm d\bar w_L\bigr]\\
&\quad=N^{-(|K|+|L|)/2}a(0,v/\sqrt N)
 (\mathrm dz_I\wedge\mathrm d\bar z_J)_{z_0}
 \wedge\mathrm dv_K\wedge\mathrm d\bar v_L.
\end{aligned}
\]
Here \(I,J,K,L\) are increasing coordinate index sets, and
\(\mathrm dz_I=\bigwedge_{i\in I}\mathrm dz_i\), with the same
convention for the other coordinate forms.  The result is a form in
\(v\) with values in the exterior algebra of the complexified cotangent
space at \(z_0\); first-variable covectors are written before the
\(v\)-covectors.  In this coefficient evaluation, \(\mathrm dz_i\) is retained as
\((\mathrm dz_i)_{z_0}\).  All derivatives of \(\eta\) are taken
before this operation.  Below we suppress the subscript \(z_0\) on
these retained covectors.

In these coordinates, use the product Dolbeault operators in the
following forms, taking all derivatives before evaluating their
coefficients at \(z=0\) and retaining the \(z\)-covectors:
\begin{equation}\label{eq:I-123}
\begin{aligned}
I_1 &:=\partial(\bar v\!\cdot\! z)\wedge\bar\partial(v\!\cdot\!\bar z)
      \wedge\partial|v|^{2}\wedge\bar\partial|v|^{2},\\[2pt]
I_{2,1}&:=\partial\bar\partial|z|^{2}\wedge\partial|v|^{2}\wedge\bar\partial|v|^{2},\\[2pt]
I_{2,2}&:=\bar\partial(v\!\cdot\!\bar z)\wedge\partial|v|^{2}
      \wedge\partial\bar\partial(z\!\cdot\!\bar v),\\[2pt]
I_{2,3}&:=\partial(\bar v\!\cdot\! z)\wedge\bar\partial\partial(v\!\cdot\!\bar z)
      \wedge\bar\partial|v|^{2},\\[2pt]
I_{2,4}&:=\partial(\bar v\!\cdot\! z)\wedge\bar\partial(v\!\cdot\!\bar z)
      \wedge\partial\bar\partial|v|^{2},\\[2pt]
I_{3,1}&:=\bar\partial\partial(v\!\cdot\!\bar z)\wedge\partial\bar\partial(z\!\cdot\!\bar v),\\[2pt]
I_{3,2}&:=\partial\bar\partial|z|^{2}\wedge\partial\bar\partial|v|^{2}.
\end{aligned}
\end{equation}

The model \((2,2)\)-form is
\begin{equation}\label{def:scaled-Q-model}
\mathcal A_{\mathrm{BF}}^{z_0}(v)
:=-\frac{e^{-|v|^2}}{4\pi^2}
 \left(I_1-\sum_{a=1}^4I_{2,a}+I_{3,1}+I_{3,2}\right).
\end{equation}

\begin{pro}
\label{prop:derivative-asymptotic}
For every fixed \(L>0\), uniformly in \(z_0\in M\) and \(|v|\leqslant L\),
coefficient by coefficient in the covectors
\((\mathrm dz_a)_{z_0},(\mathrm d\bar z_a)_{z_0},
\mathrm dv_a,\mathrm d\bar v_a\),
\begin{align}
q_N\Bigl(0,\frac v{\sqrt N}\Bigr)
&=\frac{e^{-|v|^2}}{4\pi^2}+o(1),
\label{eq:first-sector-scalar-model}\\
\operatorname{Sc}_{N,z_0}(\bar\partial_1\bar\partial_2q_N)
&=-\frac{\sqrt N\,e^{-|v|^2}}{4\pi^2}
 \bar\partial(\bar z\cdot v)\wedge\bar\partial|v|^2
 +o(\sqrt N),
\label{eq:first-sector-two-jet-model}\\
\operatorname{Sc}_{N,z_0}(\mathscr D_{12}q_N)
&=N\bigl[\mathcal A_{\mathrm{BF}}^{z_0}(v)+o(1)\bigr].
\label{eq:first-sector-four-jet-model}
\end{align}
\end{pro}

\begin{proof}
The scaled changes from the coordinates in \textup{(BF)} to the
holomorphic normal coordinates converge in every fixed \(C^r\) norm
to their unitary linear parts.  The chain rule and unitary invariance
of \(e^{-|u-v|^2}\) therefore transfer
Lemma~\ref{lem:weak-BF-jet-recovery} to these coordinates, including
the unitary rotations used at boundary points.  The squared normalized
kernels therefore satisfy
\[
\frac1{4\pi^2}
\left|\mathcal R_{N,z_0}(u,v)\right|^2
\longrightarrow
\frac{e^{-|u-v|^2}}{4\pi^2}
\quad\text{in }C^{2,2}
\]
on every fixed set \(D_T\times D_T\) of scaled variables; here \(C^{2,2}\) means ordinary derivatives
of order at most two in each endpoint.  At \(u=0\), direct differentiation of
the right-hand side gives the following identities, with \(u\)
temporarily replacing \(z\) in the forms~\eqref{eq:I-123}:
\begin{align*}
\bar\partial_u\bar\partial_v\frac{e^{-|u-v|^2}}{4\pi^2}
&=-\frac{e^{-|v|^2}}{4\pi^2}
 \bar\partial(\bar u\cdot v)\wedge\bar\partial|v|^2,\\
\partial_u\partial_v\bar\partial_u\bar\partial_v
\frac{e^{-|u-v|^2}}{4\pi^2}
&=\mathcal A_{\mathrm{BF}}^{z_0}(v).
\end{align*}
In \(\operatorname{Sc}_{N,z_0}\), only the second factor is pulled back
by dilation.  Returning the first-factor covectors from \(u=\sqrt N z\)
to \(z\) therefore gives one factor \(\sqrt N\) in the first line and
one factor \(N\) in the second.
This gives~\eqref{eq:first-sector-scalar-model}--
\eqref{eq:first-sector-four-jet-model}.
\end{proof}

For the integrations by parts in both test classes, set
\begin{equation}\label{eq:K-W}
K^{\mathcal{W}}_{N,l}(z^1,z^2):=
q_N(z^1,z^2)\bigl[\mathscr D_{12}q_N(z^1,z^2)\bigr]^{\wedge(k-l)},
\qquad 1 \leqslant l \leqslant k.
\end{equation}
Closedness of \(\mathscr D_{12}q_N\) under all four Dolbeault
operators gives, with
\(D_N=\pi_1^*\mathcal C_0^N\wedge\pi_2^*\mathcal C_0^N\) and
\(\Phi_\varphi=\pi_1^*\varphi\wedge\pi_2^*\varphi\),
\begin{equation}\label{eq:W-int-K}
\mathcal W_l^N(\varphi)
 =\int_{M^2}\mathscr D_{12}K_{N,l}^{\mathcal W}
 \wedge D_N^{\wedge(l-1)}\wedge\Phi_\varphi.
\end{equation}

\subsection{Smooth statistics in higher codimension}\label{sec:completes-the-comparison}

Throughout this subsection, \(2\leqslant k\leqslant m\).

\paragraph{Choice of the nonzero component}
For \(\eta=\sqrt{-1}\,\partial\bar\partial\varphi\), the last nonzero
K\"ahler wedge of \(\eta\) determines \(l_*\in\{1,\dots,k\}\) by
\begin{equation}\label{eq:terminal-wedge-power}
\omega^{l_*-1}\wedge\eta\not\equiv0,\qquad
\omega^{l_*}\wedge\eta\equiv0.
\end{equation}
Admissibility gives \(\eta\not\equiv0\), and bidegree gives
\(\omega^k\wedge\eta=0\), so \(l_*\) is the unique least vanishing power,
independent of the base point.  When \(l_*=k\), the last nonzero wedge
has top degree and the limiting variance is a positive multiple of the
integral of a square.  When \(l_*<k\), wedging once more with the K\"ahler form
gives zero.  This vanishing cancels terms in the kernel model before
the Gaussian contraction.  The selected component has total chaos
order \(2(k-l_*+1)\).

The case \(l_*<k\) occurs already for \(m=3\), \(k=2\).
On a product of three compact Riemann surfaces \(C_j\), choose
positive Hermitian holomorphic line bundles \((L_j,h_j)\to C_j\).
Let \(\omega_j\) be the pullback of \(\pi c_1(L_j,h_j)\).
Their external tensor product, with the product metric, has
\(\omega=\omega_1+\omega_2+\omega_3\).  Take
\(\varphi=f(z_1)(\omega_2-\omega_3)\), with \(f\) a real, smooth,
nonconstant function on the first factor.  Then \(\eta\not\equiv0\) but
\(\omega\wedge\eta=0\), since \(\omega_j^2=0\) and the two mixed
terms cancel.  Thus \(l_*=1<k\), and the lower bound uses the joint
component of total chaos order four.

\paragraph{Scaling limit for smooth forms}
For a type~\textup{(S)} test form, set \(d_l=m-k+l\) and
\begin{equation*}
\alpha_l:=\omega^{l-1}\wedge
 \sqrt{-1}\,\partial\bar\partial\varphi
=\sum_{|I|=|J|=m-k+l}\psi^l_{IJ}(z)\,
 \mathrm dz_I\wedge\mathrm d\bar z_J,\qquad 1\leqslant l\leqslant k.
\end{equation*}
Evaluating their \(\mathscr C^1\) coefficients at \(z_0\) in normal
coordinates gives the constant-coefficient forms
\begin{equation}\label{eq:frozen-smooth-coupling-forms}
\begin{aligned}
\alpha_{l,z_0}^{(z)}
&:=\sum_{|A|=|B|=m-k+l}\psi^l_{AB}(z_0)\,
 \mathrm dz_A\wedge\mathrm d\bar z_B,\\
\alpha_{l,z_0}^{(v)}
&:=\sum_{|I|=|J|=m-k+l}\psi^l_{IJ}(z_0)\,
 \mathrm dv_I\wedge\mathrm d\bar v_J.
\end{aligned}
\end{equation}
Closedness of \(\mathcal C_0^N\) allows both second-order endpoint operators in
\eqref{eq:W-int-K} to be transferred to the test forms.  The sign from
\(\mathscr D_{12}=-\partial_1\bar\partial_1\partial_2\bar\partial_2\)
is included in the two factors \(\sqrt{-1}\,\partial\bar\partial\varphi\).
The mean-current limit~\eqref{eq:BF-mean-limit} gives the following
uniform convergence on \(M\):
\begin{equation*}
\Psi_N^l:= (\mathcal C_0^N)^{l-1}\wedge
 \sqrt{-1}\,\partial\bar\partial\varphi,
\qquad
N^{-(l-1)}\Psi_N^l\longrightarrow\pi^{-(l-1)}\alpha_l.
\end{equation*}
After the transfer, the integral becomes
\[
\mathcal W_l^N(\varphi)
=\int_{M^2}K_{N,l}^{\mathcal W}\wedge
 \pi_1^*\Psi_N^l\wedge\pi_2^*\Psi_N^l.
\]
Evaluating the test coefficients at \(z_0\) as
in~\eqref{eq:frozen-smooth-coupling-forms} and applying
Proposition~\ref{prop:derivative-asymptotic}, we obtain, for fixed \(L\)
and \(|v|\leqslant L\),
\[
\operatorname{Sc}_{N,z_0}K^{\mathcal W}_{N,l}
=N^{k-l}\left[
\frac{e^{-|v|^2}}{4\pi^2}
[\mathcal A_{\mathrm{BF}}^{z_0}(v)]^{\wedge(k-l)}+o(1)\right]
\]
uniformly in the base point, including the second-variable pullback.
The bidegree \((d_l,d_l)\) of \(\alpha_l\) determines the remaining
scaling factor.  The kernel contributes \(N^{k-l}\), the two
mean-current powers contribute \(N^{2(l-1)}\), and the pullback of the
second test form contributes \(N^{-d_l}\).  The resulting scale is
\(N^{k-l+2(l-1)-d_l}=N^{2k-m-2}\).

Outside the scaled \(L\)-ball,
Lemma~\ref{lem:BF-OD-consequences} bounds the paired integrand by
\(CN^{2k-2}\mathfrak h_N\) in absolute value, and
\eqref{eq:BF-absolute-tail} bounds its integral by
\(CN^{2k-m-2}\varepsilon_{2m}(L)\).  The limiting coefficients are
polynomials times a Gaussian.  We may therefore pass first to
\(N\to\infty\) and then to \(L\to\infty\).  A finite coordinate cover
and partition of unity give
\begin{equation}\label{eq:W-l-over-N}
\begin{aligned}
\lim_{N\to\infty}N^{-2k+m+2}\mathcal W_l^N(\varphi)
={}&\frac1{\pi^{2(l-1)}}\int_M\alpha_{l,z_0}^{(z)}\\
&\wedge\int_{\mathbb C^m}
\frac{e^{-|v|^2}}{4\pi^2}
[\mathcal A_{\mathrm{BF}}^{z_0}(v)]^{\wedge(k-l)}
\wedge\alpha_{l,z_0}^{(v)}.
\end{aligned}
\end{equation}
Use the Euclidean volume convention
\begin{equation*}
\mathrm d\mathsf{Vol}_{\mathbb C^m}(v)
:=\Bigl(\frac{\sqrt{-1}}2\Bigr)^m
\bigwedge_{a=1}^m\mathrm dv_a\wedge\mathrm d\bar v_a.
\end{equation*}

\paragraph{The top-degree case \(l_*=k\)}

Write \(\mathrm dV=\omega^m/m!\).  If \(l_*=k\), then
\(\alpha_k=a_\varphi\,\mathrm dV\) for a nonzero real continuous
function \(a_\varphi\).  In unitary normal coordinates the coefficients
of \(\mathrm dV\) at \(z_0\) give
\(\mathrm d\mathsf{Vol}_{\mathbb C^m}\), so
\eqref{eq:W-l-over-N} becomes
\begin{equation}\label{eq:top-smooth-leading-density}
\begin{aligned}
\lim_{N\to\infty}N^{-2k+m+2}\mathcal W_k^N(\varphi)
&=\frac1{4\pi^{2k}}\int_M a_\varphi^2\,\mathrm dV
 \int_{\mathbb C^m}e^{-|v|^2}\,
             \mathrm d\mathsf{Vol}_{\mathbb C^m}(v)\\
&=\frac{\pi^{m-2k}}4\int_M a_\varphi^2\,\mathrm dV>0.
\end{aligned}
\end{equation}
This proves~\eqref{eq:selected-component-positive-limit} when \(l_*=k\).

\paragraph{Forms satisfying \(\omega\wedge\alpha_l=0\)}
For \(l_*<k\), we write the exterior product as a matrix contraction
and apply the following Gaussian identity.

\begin{lem}
\label{lem:exterior-Gaussian-contraction}
Let \(1\leqslant t\leqslant m\), and equip
\(\bigwedge^\bullet\mathbb C^m\) with its standard Hermitian structure.
For the standard orthonormal basis \(e_1,\ldots,e_m\) of \(\mathbb C^m\),
let \(\varepsilon_a:\bigwedge^r\mathbb C^m\to
\bigwedge^{r+1}\mathbb C^m\) denote exterior multiplication by \(e_a\),
and let
\(\iota_a=\varepsilon_a^*:\bigwedge^r\mathbb C^m\to
\bigwedge^{r-1}\mathbb C^m\).  For a Hermitian endomorphism \(B\) of
\(\bigwedge^t\mathbb C^m\), set
\[
 \operatorname{ptr}(B):=\sum_{b=1}^m\iota_bB\varepsilon_b
 \in\operatorname{End}(\bigwedge^{t-1}\mathbb C^m).
\]
This endomorphism is Hermitian.  For the Gaussian integral, use the operators
\[
 M_t(v):=\bigwedge^t(I-vv^*)=I-\varepsilon_v\iota_v,\qquad
 \varepsilon_v=\sum_av_a\varepsilon_a,\qquad
 \iota_v=\sum_b\bar v_b\iota_b.
\]
Write \(\|\cdot\|_{\mathrm{HS}}\) for the Hilbert--Schmidt norm.
Then, for every \(\lambda>0\),
\begin{equation}\label{eq:abstract-Gaussian-contraction}
\begin{aligned}
&\int_{\mathbb C^m}e^{-\lambda|v|^2}
 \operatorname{Tr}\bigl(BM_t(v)BM_t(v)\bigr)
 \,\mathrm d\mathsf{Vol}_{\mathbb C^m}(v)\\
&\qquad=\Bigl(\frac\pi\lambda\Bigr)^m
 \left[\Bigl(1-\frac t\lambda\Bigr)^2\|B\|_{\mathrm{HS}}^2
 +\lambda^{-2}\|\operatorname{ptr}(B)\|_{\mathrm{HS}}^2\right].
\end{aligned}
\end{equation}
In particular, if \(\operatorname{ptr}(B)=0\), then
\begin{equation}\label{eq:trace-free-Gaussian-contraction}
 \int_{\mathbb C^m}e^{-(t+1)|v|^2}
 \operatorname{Tr}\bigl(BM_t(v)BM_t(v)\bigr)
 \,\mathrm d\mathsf{Vol}_{\mathbb C^m}(v)
 =\frac{\pi^m}{(t+1)^{m+2}}\|B\|_{\mathrm{HS}}^2.
\end{equation}
Under the condition \(\operatorname{ptr}(B)=0\), the integral is positive precisely
when \(B\ne0\).
\end{lem}

\begin{proof}
Set \(R_{ab}=\varepsilon_a\iota_b\) and
\[
 \mathsf A_v=\varepsilon_v\iota_v
 =\sum_{a,b}v_a\bar v_bR_{ab}.
\]
Let \(\mathbb E_\lambda\) denote expectation under the normalized
Gaussian measure
\( (\lambda/\pi)^m e^{-\lambda|v|^2}\,
\mathrm d\mathsf{Vol}_{\mathbb C^m}(v)\).
Its second and fourth moments are
\[
 \mathbb E_\lambda[v_a\bar v_b]=\lambda^{-1}\delta_{ab},\qquad
 \mathbb E_\lambda[v_a\bar v_bv_c\bar v_d]
 =\lambda^{-2}(\delta_{ab}\delta_{cd}+\delta_{ad}\delta_{cb}).
\]
Since \(\sum_aR_{aa}=tI\), expansion gives
\[
\begin{aligned}
\mathbb E_\lambda\operatorname{Tr}(B(I-\mathsf A_v)B(I-\mathsf A_v))
&=\Bigl(1-\frac t\lambda\Bigr)^2\operatorname{Tr}(B^2)\\
&\quad+\lambda^{-2}\sum_{a,b}\operatorname{Tr}(BR_{ab}BR_{ba}).
\end{aligned}
\]
To evaluate the remaining sum, use the trace identity for maps between
\(\bigwedge^t\mathbb C^m\) and \(\bigwedge^{t-1}\mathbb C^m\):
\[
\begin{aligned}
&\operatorname{Tr}_{\bigwedge^t\mathbb C^m}
 \bigl(B\varepsilon_a\iota_bB\varepsilon_b\iota_a\bigr)\\
&\qquad=\operatorname{Tr}_{\bigwedge^{t-1}\mathbb C^m}
 \bigl(\iota_aB\varepsilon_a\iota_bB\varepsilon_b\bigr).
\end{aligned}
\]
Summing over \(a,b\) gives
\[
\begin{aligned}
\sum_{a,b}\operatorname{Tr}(BR_{ab}BR_{ba})
&=\sum_a\operatorname{Tr}\bigl(
 \iota_aB\varepsilon_a\operatorname{ptr}(B)\bigr)\\
&=\operatorname{Tr}\bigl(\operatorname{ptr}(B)^2\bigr)
 =\|\operatorname{ptr}(B)\|_{\mathrm{HS}}^2.
\end{aligned}
\]
Both \(B\) and \(\operatorname{ptr}(B)\) are Hermitian, proving
\eqref{eq:abstract-Gaussian-contraction} after undoing the Gaussian
normalization.  Taking \(\lambda=t+1\) and \(\operatorname{ptr}(B)=0\)
gives~\eqref{eq:trace-free-Gaussian-contraction}.
\end{proof}

\smallskip\noindent\emph{Application to \(\alpha_l\).}

Suppose \(l=l_*\leqslant k-1\), and set
\(t=k-l\), \(d_l=m-t=m-k+l\); thus \(1\leqslant t\leqslant m-1\).
The nonzero real form \(\alpha_l\) satisfies \(\omega\wedge\alpha_l=0\).
At every normal-coordinate centre,
\begin{equation*}
\alpha_{l,z_0}^{(z)}\wedge\partial\bar\partial|z|^2=0,\qquad
\alpha_{l,z_0}^{(v)}\wedge\partial\bar\partial|v|^2=0.
\end{equation*}
The definitions in~\eqref{eq:I-123} consequently give
\begin{equation}\label{eq:wedge-I=0}
\alpha_{l,z_0}^{(z)}\wedge\alpha_{l,z_0}^{(v)}\wedge T=0,
\qquad T\in\{I_{2,1},I_{2,4},I_{3,2}\},
\end{equation}
since each of these forms contains \(\partial\bar\partial|z|^2\),
\(\partial\bar\partial|v|^2\), or both as an exterior factor.

Use the differential forms
\[
\begin{aligned}
\mathsf E&=\sum_a\mathrm d\bar z_a\wedge\mathrm dv_a,
&\mathsf F&=\sum_a\mathrm dz_a\wedge\mathrm d\bar v_a,\\
\mathsf Z&=\sum_a\bar v_a\,\mathrm dz_a,
&\overline{\mathsf Z}&=\sum_av_a\,\mathrm d\bar z_a,\\
\mathsf V&=\sum_a\bar v_a\,\mathrm dv_a,
&\overline{\mathsf V}&=\sum_av_a\,\mathrm d\bar v_a.
\end{aligned}
\]
The surviving part of the model factors as
\[
 I_1-I_{2,2}-I_{2,3}+I_{3,1}
 =(\mathsf E-\overline{\mathsf Z}\wedge\mathsf V)
  \wedge(\mathsf F-\mathsf Z\wedge\overline{\mathsf V}).
\]
Indeed, \(I_1=\mathsf Z\wedge\overline{\mathsf Z}\wedge\mathsf V
\wedge\overline{\mathsf V}\),
\(I_{2,2}=\overline{\mathsf Z}\wedge\mathsf V\wedge\mathsf F\),
\(I_{2,3}=\mathsf Z\wedge\mathsf E\wedge\overline{\mathsf V}\), and
\(I_{3,1}=\mathsf E\wedge\mathsf F\).
In every exterior power of \(\mathcal A_{\mathrm{BF}}^{z_0}\), any term
containing \(I_{2,1}\), \(I_{2,4}\), or \(I_{3,2}\) vanishes after pairing
with \(\alpha_{l,z_0}^{(z)}\) and \(\alpha_{l,z_0}^{(v)}\),
by~\eqref{eq:wedge-I=0}.  The two factors in the display have even degree
and hence commute.  Thus, with \(t=k-l\), the local integral
in~\eqref{eq:W-l-over-N} is
\begin{equation}\label{eq:whole-form}
\begin{aligned}
&\frac1{\pi^{2(l-1)}}\alpha_{l,z_0}^{(z)}\wedge
\int_{\mathbb C^m}\frac{e^{-|v|^2}}{4\pi^2}
[\mathcal A_{\mathrm{BF}}^{z_0}(v)]^{\wedge t}
\wedge\alpha_{l,z_0}^{(v)}
\\
&\qquad=\frac{(-1)^t}{4^{t+1}\pi^{2k}}\alpha_{l,z_0}^{(z)}\wedge
\int_{\mathbb C^m}e^{-(t+1)|v|^2}
(\mathsf E-\overline{\mathsf Z}\wedge\mathsf V)^{\wedge t}\\
&\qquad\wedge(\mathsf F-\mathsf Z\wedge\overline{\mathsf V})^{\wedge t}
\wedge\alpha_{l,z_0}^{(v)}.
\end{aligned}
\end{equation}
The surviving exterior terms can be written as a matrix contraction.
For every \(t\)-subset \(D\subset\{1,\dots,m\}\), let \(\sigma(D)\) satisfy
\(\mathrm dz_{D^c}\wedge\mathrm dz_D
=\sigma(D)\mathrm dz_1\wedge\cdots\wedge\mathrm dz_m\).
On the standard orthonormal basis of \(\bigwedge^t\mathbb C^m\), set
\[
(B_l(z_0))_{D,C}
:=(\sqrt{-1})^{d_l}\sigma(D)\sigma(C)\psi^l_{D^c,C^c}(z_0).
\]
Use the exterior multiplication and contraction operators from
Lemma~\ref{lem:exterior-Gaussian-contraction}.

At the coordinate centre, set
\[
\Omega_0:=\prod_{a=1}^m
\sqrt{-1}\,\mathrm dz_a\wedge\mathrm d\bar z_a,\qquad
\omega_0:=\frac{\sqrt{-1}}2
\sum_a\mathrm dz_a\wedge\mathrm d\bar z_a.
\]

\begin{lem}
\label{lem:co-primitive-BF-contraction}
The matrix \(B_l(z_0)\) is Hermitian, and
\[
\|B_l(z_0)\|_{\mathrm{HS}}^2
=\sum_{|A|=|B|=d_l}|\psi^l_{AB}(z_0)|^2.
\]
Moreover,
\(\omega_0\wedge\alpha_{l,z_0}^{(z)}=0\) is equivalent to
\[
\sum_{b=1}^m\iota_bB_l(z_0)\varepsilon_b=0,
\]
and, under these equivalent conditions,
\begin{align}
&(-1)^t\alpha_{l,z_0}^{(z)}\wedge
\int_{\mathbb C^m}e^{-(t+1)|v|^2}
(\mathsf E-\overline{\mathsf Z}\wedge\mathsf V)^{\wedge t}
\nonumber\\
&\qquad\wedge(\mathsf F-\mathsf Z\wedge\overline{\mathsf V})^{\wedge t}
\wedge\alpha_{l,z_0}^{(v)}
\nonumber\\
&\qquad=
\frac{2^m(t!)^2\pi^m}{(t+1)^{m+2}}
\|B_l(z_0)\|_{\mathrm{HS}}^2\,\Omega_0.
\label{eq:primitive-Gaussian-contraction}
\end{align}
Its scalar coefficient is positive precisely when \(\alpha_l(z_0)\ne0\).
\end{lem}

\begin{proof}
Write \(B=B_l(z_0)\).  Reality gives
\(\psi^l_{BA}=(-1)^{d_l}\overline{\psi^l_{AB}}\); therefore \(B\) is
Hermitian, and complementation gives its displayed Hilbert--Schmidt norm.

Appendix~\ref{app:BF-complementary-indices} converts the factored exterior product
into the trace formula~\eqref{eq:compound-matrix-contraction}.  The same
complementary-index calculation identifies
\(\omega_0\wedge\alpha_{l,z_0}^{(z)}=0\) with
\(\operatorname{ptr}(B)=0\), which is the hypothesis needed for the
Gaussian identity.

\smallskip\noindent\emph{Gaussian integration.}
The undifferentiated \(q_N\) in~\eqref{eq:K-W} supplies one extra
Gaussian factor, so the exponent is \(t+1\).
Under \(\operatorname{ptr}(B)=0\), the coefficient
\((1-t/\lambda)^2\) in~\eqref{eq:abstract-Gaussian-contraction}
becomes \((t+1)^{-2}>0\) at \(\lambda=t+1\).
Substitution of~\eqref{eq:trace-free-Gaussian-contraction} into
\eqref{eq:compound-matrix-contraction} proves
\eqref{eq:primitive-Gaussian-contraction}.
\end{proof}

\paragraph{Positive limit for the selected smooth component}
Substituting~\eqref{eq:primitive-Gaussian-contraction} into~\eqref{eq:whole-form} gives
\begin{equation}\label{eq:degenerate-smooth-leading-density}
\begin{aligned}
&\frac1{\pi^{2(l-1)}}\alpha_{l,z_0}^{(z)}\wedge
\int_{\mathbb C^m}\frac{e^{-|v|^2}}{4\pi^2}
[\mathcal A_{\mathrm{BF}}^{z_0}(v)]^{\wedge(k-l)}
\wedge\alpha_{l,z_0}^{(v)}
\\
&\qquad=\frac1{4^{k-l+1}\pi^{2k}}
\frac{2^m((k-l)!)^2\pi^m}{(k-l+1)^{m+2}}
\|B_l(z_0)\|_{\mathrm{HS}}^2\,\Omega_0.
\end{aligned}
\end{equation}
The coordinate-independent expression for both limits uses
\(t=k-l_*\) and the Hermitian norm on complex exterior forms induced by the
real Riemannian metric.  In unitary normal coordinates,
\[
 |\alpha_{l_*}|^2
 =4^{m-t}\sum_{|A|=|B|=m-t}|\psi^{l_*}_{AB}|^2,
 \qquad \Omega_0=2^m\,\mathrm dV.
\]
Thus~\eqref{eq:W-l-over-N} and
\eqref{eq:degenerate-smooth-leading-density}, together with
\eqref{eq:top-smooth-leading-density} for \(t=0\), give
\begin{equation}\label{eq:intrinsic-smooth-component-limit}
 \lim_{N\to\infty}v_N^{-1}\mathcal W_{l_*}^N(\varphi)
 =\frac{\pi^{m-2k}(t!)^2}{4(t+1)^{m+2}}
   \int_M|\alpha_{l_*}|^2\,\mathrm dV>0.
\end{equation}
At \(t=0\), \(\alpha_k=a_\varphi\,\mathrm dV\) and
\(|\alpha_k|^2=a_\varphi^2\).
Since \(\alpha_{l_*}\not\equiv0\), this proves
\eqref{eq:selected-component-positive-limit} for smooth statistics.

\subsection{Weighted-domain statistics}
\label{sec:numerical-statistics}

Fix an admissible domain test \(\varphi_U\) with weight \(f\) as in
Definition~\ref{def:admissible-zero-statistics}\textup{(D)}.
A boundary integral gives the leading term of \(\mathcal W_k^N(\varphi_U)\),
the variance of the component with chaos order two in one coefficient
vector and order zero in the others.  Terms involving derivatives of
\(f\) are \(O(N^{2k-m-1})\), smaller than the leading scale
\(N^{2k-m-1/2}\).  For the boundary calculation, set
\[
\Theta_N:=\frac{\omega^{m-k}\wedge
 (\mathcal C_0^N)^{\wedge(k-1)}}{(m-k)!}.
\]
By~\eqref{eq:BF-mean-limit},
\begin{equation*}
N^{-(k-1)}\Theta_N\longrightarrow
\frac{\omega^{m-1}}{\pi^{k-1}(m-k)!}
\quad\text{uniformly on }M;
\end{equation*}
in particular, \(\|\Theta_N\|_\infty=O(N^{k-1})\).
We use the surface measure \(\sigma_{\partial U}\) from
Definition~\ref{def:admissible-zero-statistics}\textup{(D)} and the
boundary and product orientations of Lemma~\ref{lem:boundary-Stokes}.
The repeated Stokes identity is part~\textup{(i)} of that lemma.

\paragraph{Reduction to a boundary integral}
For a smooth scalar function \(q\), two integrations by parts give
\begin{equation}\label{eq:weighted-sector-Stokes}
\int_U f\,\partial\bar\partial q\wedge\Theta_N
=\int_{\partial U}(f\bar\partial q+q\partial f)\wedge\Theta_N
 +\int_U q\,\partial\bar\partial f\wedge\Theta_N.
\end{equation}
Indeed, apply Stokes first to \(f\bar\partial q\wedge\Theta_N\) and
then to \(q\partial f\wedge\Theta_N\); closedness of \(\Theta_N\)
removes its derivatives.  Applying~\eqref{eq:weighted-sector-Stokes} in
both variables of the integral over \(U^2\) in~\eqref{eq:W-int-K},
with the product orientation fixed above, gives
\[
\mathcal W_k^N(\varphi_U)=\mathcal B_N(f)+\mathcal E_N(f).
\]
The boundary term retains one kernel derivative at each endpoint:
\begin{equation}\label{eq:weighted-sector-boundary-pair}
\mathcal B_N(f):=
\int_{(\partial U)^2}f(z)f(w)
 [-\bar\partial_1\bar\partial_2q_N(z,w)]
 \wedge\pi_1^*\Theta_N\wedge\pi_2^*\Theta_N.
\end{equation}
In
\eqref{eq:weighted-sector-Stokes}, \(B\) denotes
\(\int_{\partial U}f\bar\partial q\wedge\Theta_N\),
\(D\) denotes \(\int_{\partial U}q\partial f\wedge\Theta_N\), and
\(I\) denotes \(\int_Uq\partial\bar\partial f\wedge\Theta_N\).
The error \(\mathcal E_N(f)\) consists of all pairs other than \(BB\),
each with a derivative on the weight.
Apply~\eqref{eq:BF-q-jet-tail} and \eqref{eq:BF-absolute-tail},
integrating over \(U\) first for mixed \(U\)-boundary terms:
\begin{equation*}
\begin{array}{c|c|c|c}
\textup{pair}&\textup{kernel derivatives}&\textup{first integration}
 &\textup{bound}\\ \hline
BD&1&\partial U&O(N^{2k-m-1})\\
DD&0&\partial U&O(N^{2k-m-3/2})\\
BI&1&U&O(N^{2k-m-3/2})\\
DI&0&U&O(N^{2k-m-2})\\
II&0&U&O(N^{2k-m-2})
\end{array}
\end{equation*}
The transposed pairs have the same bounds: the two \(\Theta_N\)
factors contribute \(O(N^{2k-2})\), each \(B\) contributes \(N^{1/2}\),
and the first integration contributes \(O(N^{-m+1/2})\) over
\(\partial U\) or \(O(N^{-m})\) over \(U\).  These bounds give the total
error estimate
\begin{equation}\label{eq:weighted-sector-error-total}
\mathcal E_N(f)=O(N^{2k-m-1})
=o(N^{2k-m-1/2}).
\end{equation}

\paragraph{Rescaled boundary charts}
Fix \(L>0\).  For \(z_0\in\partial U\), use normal coordinates
centred at \(z_0\) and set
\begin{equation*}
z^1=0,\quad z^2=v/\sqrt N,\qquad
\Sigma_{N,z_0}(2L)=\{v\in B(0,2L):v/\sqrt N\in\partial U\}.
\end{equation*}
Uniformly for \(|v|\leqslant2L\),
\eqref{eq:uniform-normal-distance} gives
\begin{equation}\label{eq:normal-coordinate-distance-comparison}
\sqrt N\,\operatorname{dist}(z_0,v/\sqrt N)
 =|v|+O_L(N^{-1/2}).
\end{equation}
Orient
\(H=\{\operatorname{Im}w_1=0\}=\mathbb R\times\mathbb C^{m-1}\)
as the boundary of \(\{\operatorname{Im}w_1>0\}\).  Write
\(\sigma_H\) for its Euclidean surface measure and
\(\mathrm d\mathsf{Vol}_H\) for the associated oriented volume form.

\begin{lem}
\label{lem:uniform-scaled-boundary-chart}
For fixed \(L>0\), all sufficiently large \(N\), and every
\(z_0\in\partial U\), choose the normal coordinates and oriented graph
\(\psi_{z_0}\) of Lemma~\ref{lem:C2-boundary-geometry}.  The real
\(\mathscr C^2\) map
\begin{equation*}
\mathfrak F_{N,z_0}(v)
 =\bigl(v_1+\sqrt{-1}\sqrt N\psi_{z_0}(v/\sqrt N),v_2,\dots,v_m\bigr)
\end{equation*}
maps the rescaled boundary to \(H\) and the rescaled \(U\)-side to
\(\{\operatorname{Im}w_1>0\}\).  It is an orientation-preserving
diffeomorphism on \(B(0,2L)\) onto
its image.  The rescaled boundary image satisfies
\begin{equation}\label{eq:boundary-condition}
H\cap B(0,L)\subset
\mathfrak F_{N,z_0}(\Sigma_{N,z_0}(2L))\subset H.
\end{equation}
Its inverse \(\mathfrak G=\mathfrak G_{N,z_0}\) satisfies on
\(H\cap B(0,L)\), uniformly in \(z_0\),
\[
|\mathfrak G(w)-w|\leqslant C|w|^2/\sqrt N,
\qquad \|\mathrm D\mathfrak G(w)-I\|\leqslant C|w|/\sqrt N.
\]
The Jacobian of \(\mathfrak G|_H\), computed with respect to Euclidean
surface measure, is \(1+O(|w|/\sqrt N)\).
\end{lem}

\begin{proof}
Take \(N\) so large that \(v/\sqrt N\) lies in the chart of
Lemma~\ref{lem:C2-boundary-geometry} whenever \(|v|\leqslant2L+1\).
For \(\mathfrak F=I+E_N\), the graph bounds give
\begin{equation}\label{eq:quantitative-F-bounds}
|E_N(v)|\leqslant C|v|^2/\sqrt N,\quad
\|\mathrm DE_N(v)\|\leqslant C|v|/\sqrt N,\quad
\|\mathrm D^2E_N(v)\|\leqslant C/\sqrt N.
\end{equation}
Integration along segments, with
\(\delta_{N,L}=C_L/\sqrt N<1/4\), yields
\begin{equation*}
(1-\delta_{N,L})|v-w|\leqslant|\mathfrak F(v)-\mathfrak F(w)|
 \leqslant(1+\delta_{N,L})|v-w|.
\end{equation*}
The inverse function theorem gives a \(\mathscr C^2\) inverse.  Since
\(\operatorname{Im}\mathfrak F(v)_1
=\operatorname{Im}v_1+\sqrt N\psi_{z_0}(v/\sqrt N)\), the rescaled
\(U\)-side maps to \(\{\operatorname{Im}w_1>0\}\).  Its derivative is joined to the identity through
invertible matrices; its positive real determinant therefore preserves
both ambient and outward-normal-first boundary orientations.

For \(w\in H\cap B(0,L)\), the map \(v\mapsto w-E_N(v)\) is a
contraction of \(\overline B(w,C_L/\sqrt N)\) into itself for a suitable
\(C_L\), by~\eqref{eq:quantitative-F-bounds}.
This ball lies in \(B(0,2L)\) for large \(N\).  Its fixed point satisfies
\(\mathfrak F(v)=w\) and therefore lies on the scaled boundary,
proving~\eqref{eq:boundary-condition}.

The identity \(\mathfrak G(w)-w=-E_N(\mathfrak G(w))\) and the Neumann
series give the asserted bounds for \(\mathfrak G\) and
\(\mathrm D\mathfrak G\).  Only the first imaginary coordinate changes:
for \(w=(x_1,w')\in H\),
\(\mathfrak G(w)=(x_1+\sqrt{-1}h(w),w')\), where
\(|h|\leqslant C|w|^2/\sqrt N\), \(|\mathrm Dh|\leqslant C|w|/\sqrt N\).
Thus its surface Jacobian is \(1+O(|w|/\sqrt N)\),
\(\mathfrak G^*\mathrm dv_1=\mathrm dx_1+\sqrt{-1}\mathrm dh\) and
\(\mathfrak G^*\mathrm d\bar v_1=\mathrm dx_1-\sqrt{-1}\mathrm dh\),
with the other coordinate one-forms unchanged.  All constants are
uniform in the boundary base point.
\end{proof}

\begin{lem}
\label{lem:weighted-boundary-localization}
Let \(\Lambda_N\) be smooth forms on \(M^2\) of total degree
\(2m-1\) in each variable.  On the oriented boundary product, write
\[
 \Lambda_N=\lambda_N(z,w)\,
 \mathrm d\mathsf{Vol}_{\partial U}(z)
 \wedge\mathrm d\mathsf{Vol}_{\partial U}(w).
\]
Suppose that \(|\lambda_N(z,w)|\leqslant
CN^{2k-1}\mathfrak h_N(z,w)\), with \(C\) independent of \(N,z,w\).
In the charts of Lemma~\ref{lem:uniform-scaled-boundary-chart}, suppose
also that, for every fixed \(L>0\),
\[
 N^{-(2k-1)}\lambda_N
 \bigl(z,\delta_{N,z}(\mathfrak G_{N,z}(v))\bigr)
 \longrightarrow b(v)
\]
uniformly for \(z\in\partial U\) and \(v\in H\cap B(0,L)\), where
\(b\in L^1(H,\mathrm d\sigma_H)\) is independent of \(z\).
Then every fixed real \(f\in\mathscr C^1(M)\) satisfies
\[
\begin{aligned}
 &\lim_{N\to\infty}N^{-2k+m+1/2}
 \int_{(\partial U)^2}f(z)f(w)\Lambda_N\\
 &\qquad=\left(\int_H b(v)\,\mathrm d\sigma_H(v)\right)
          \int_{\partial U}f^2\,\mathrm d\sigma_{\partial U}.
\end{aligned}
\]
Replacing \(f(z)f(w)\) by \(f(z)^2\) in the integral changes it by
\(O_f(N^{2k-m-1})\).
\end{lem}

\begin{proof}
The area bound~\eqref{eq:boundary-volume-upper} and \(A>2m\) give
\[
\begin{aligned}
 \sup_{z\in\partial U}\int_{\partial U}
 \operatorname{dist}(z,w)\mathfrak h_N(z,w)
 \,\mathrm d\sigma_{\partial U}(w)
 &\leqslant CN^{-m}\int_0^\infty
 t^{2m-1}(1+t)^{-A}\,\mathrm dt\\
 &\leqslant C'N^{-m}.
\end{aligned}
\]
The Lipschitz bound for \(f\), together with the assumed bound for
\(\lambda_N\), proves the stated error when freezing the weight.

Fix \(L\geqslant2\).  Parametrize a boundary neighbourhood by
\[
 \Gamma_{N,z,L}:=(\delta_{N,z}\circ\mathfrak G_{N,z})
                    (H\cap B(0,L)).
\]
The graph and distance estimates in
Lemma~\ref{lem:uniform-scaled-boundary-chart} and
\eqref{eq:normal-coordinate-distance-comparison} give, for all
sufficiently large \(N\), uniformly in \(z\),
\[
 \partial U\cap B\left(z,\frac{L}{2\sqrt N}\right)
 \subset\Gamma_{N,z,L}
 \subset\partial U\cap B\left(z,\frac{2L}{\sqrt N}\right).
\]
The pulled-back surface measure, including the ambient metric, is
\[
 (\delta_{N,z}\circ\mathfrak G_{N,z})^*
       \mathrm d\sigma_{\partial U}
 =N^{-m+1/2}(1+o_L(1))\,\mathrm d\sigma_H,
\]
uniformly in \(z\in\partial U\) and \(v\in H\cap B(0,L)\).
The local density assumption therefore gives
\[
 N^{-2k+m+1/2}\int_{\Gamma_{N,z,L}}
 \lambda_N(z,w)\,\mathrm d\sigma_{\partial U}(w)
 \longrightarrow\int_{H\cap B(0,L)}b(v)\,\mathrm d\sigma_H(v),
\]
uniformly in \(z\).  For the remaining part,
\eqref{eq:BF-absolute-tail} gives
\begin{equation*}
 N^{-2k+m+1/2}\sup_{z\in\partial U}
 \int_{\partial U\setminus\Gamma_{N,z,L}}
 |\lambda_N(z,w)|\,\mathrm d\sigma_{\partial U}(w)
 \leqslant C\varepsilon_{2m-1}(L/2).
\end{equation*}
First let \(N\to\infty\), then \(L\to\infty\).
Integrability of \(b\) gives uniform convergence of the full integral
in \(w\).  Integrating against \(f(z)^2\,\mathrm d\sigma_{\partial U}(z)\)
and using the freezing error proves the formula.
\end{proof}

\paragraph{The limiting integrand on the tangent hyperplane}
All Dolbeault derivatives below are taken before the real boundary
pullback.  Proposition~\ref{prop:derivative-asymptotic} gives, uniformly
for \(|v|\leqslant L\),
\begin{equation}\label{eq:ddK-W}
\begin{aligned}
\operatorname{Sc}_{N,z_0}(-\bar\partial_1\bar\partial_2q_N)
&=\frac{\sqrt N}{4\pi^2}\Bigl[
 e^{-|v|^2}\bar\partial(\bar z\cdot v)\wedge\bar\partial|v|^2
 +o_L(1)\Bigr].
\end{aligned}
\end{equation}
In the scaled coordinates of \(\delta_{N,z_0}\), the Euclidean
K\"ahler form is
\(\omega_{\mathrm E}=(\sqrt{-1}/2)\sum_a\mathrm dv_a\wedge\mathrm d\bar v_a\).
The uniform normal-coordinate expansion
\(\omega=\frac{\sqrt{-1}}2\sum_{a,b=1}^m
(\delta_{ab}+O(|z|^2))\,\mathrm dz_a\wedge\mathrm d\bar z_b\)
and the mean-current limit~\eqref{eq:BF-mean-limit} give the pullback expansions
\begin{align}
\delta_{N,z_0}^*\omega
 &=N^{-1}[\omega_{\mathrm E}+O(|v|^2/N)],
\label{eq:scaled-Kahler-form-correct}\\
\delta_{N,z_0}^*\Theta_N
 &=\frac{N^{k-m}}{\pi^{k-1}(m-k)!}
  [\omega_{\mathrm E}^{\wedge(m-1)}+o_L(1)].
\notag
\end{align}
The coefficients of \(\Theta_N(z_0)\) and
\(\delta_{N,z_0}^*\Theta_N\) have respective orders
\(N^{k-1}\) and \(N^{k-m}\).
Combining~\eqref{eq:BF-mean-limit}, \eqref{eq:ddK-W},
and~\eqref{eq:scaled-Kahler-form-correct} gives
\begin{equation}\label{eq:Upsilon-W}
\begin{aligned}
&\operatorname{Sc}_{N,z_0}[-\bar\partial_1\bar\partial_2q_N
 \wedge\pi_1^*\Theta_N\wedge\pi_2^*\Theta_N]\\
&\qquad=N^{2k-m-\frac12}
 [e^{-|v|^2}\mathcal K_{\infty,k}(z_0,v)
      +o_L(1)].
\end{aligned}
\end{equation}
The limiting kernel form is
\begin{equation}\label{eq:K-infinity}
\begin{aligned}
\mathcal K_{\infty,k}(z_0,v)
={}&\Bigl(\frac1{2\pi^k(m-k)!}\Bigr)^2
 \bar\partial(\bar z\cdot v)\wedge\bar\partial|v|^2\\
&\wedge\Bigl[\frac{\sqrt{-1}}2\partial\bar\partial|z|^2
       \wedge\frac{\sqrt{-1}}2\partial\bar\partial|v|^2\Bigr]^{\wedge(m-1)}.
\end{aligned}
\end{equation}

Writing \(v_j=x_j+\sqrt{-1}y_j\), the outward normal to the side
\(y_1>0\) is \(-\partial_{y_1}\).  The \(x_1\)-direction is tangent
to \(H\), along the intersection of \(H\) with the complex normal line.
The ordered frame
\((-\partial_{y_1},\partial_{x_1},\partial_{x_2},\partial_{y_2},\dots)\)
has the complex orientation, so
\begin{equation*}
\mathrm d\mathsf{Vol}_H
 =\mathrm dx_1\wedge\prod_{j=2}^m
   \frac{\sqrt{-1}}2\mathrm dv_j\wedge\mathrm d\bar v_j.
\end{equation*}
After restriction to \(T_{z_0}\partial U\),
\begin{equation*}
\mathrm d\mathsf{Vol}_{\partial U}(z_0)
 =\mathrm d\bar z_1\wedge\prod_{j=2}^m
    \frac{\sqrt{-1}}2\mathrm dz_j\wedge\mathrm d\bar z_j
 =\frac{\mathrm d\bar z_1\wedge\omega(z_0)^{m-1}}{(m-1)!}.
\end{equation*}
Wedging each one-form in~\eqref{eq:K-infinity} with its corresponding
K\"ahler power and restricting to the tangent hyperplane leaves only
the term with index \(1\).  Every other term contains
\(\mathrm dx_1\wedge\mathrm dy_1\), which vanishes on restriction.
The two K\"ahler powers each contribute \((m-1)!\), while \(v_1=x_1\)
on \(H\).  The resulting constant is
\begin{equation*}
 \kappa_{m,k}
 :=\Bigl(\frac{(m-1)!}{2\pi^k(m-k)!}\Bigr)^2>0.
\end{equation*}
The K\"ahler factors have even degree, so regrouping by variable
introduces no sign.  Restriction of the kernel form therefore gives
\begin{equation}\label{eq:K-infinity-boundary-density}
\left.\mathcal K_{\infty,k}(z_0,v)\right|_{T_{z_0}\partial U\times H}
 =\kappa_{m,k}x_1^2\,
  \mathrm d\mathsf{Vol}_{\partial U}(z_0)\wedge\mathrm d\mathsf{Vol}_H(v).
\end{equation}

\paragraph{Positive limit for the selected domain component}
Apply Lemma~\ref{lem:weighted-boundary-localization} to the kernel form
in~\eqref{eq:weighted-sector-boundary-pair}.
The bounds for \(\Theta_N\) and~\eqref{eq:BF-q-jet-tail} give its
coefficient bound \(CN^{2k-1}\mathfrak h_N\).
Equations~\eqref{eq:Upsilon-W} and
\eqref{eq:K-infinity-boundary-density} identify its scalar limit as
\[
 b(v)=\kappa_{m,k}e^{-|v|^2}x_1^2.
\]
Indeed, \eqref{eq:Upsilon-W} already includes pullback in the second
variable; dividing by the surface Jacobian
\(N^{-m+1/2}(1+o_L(1))\) gives the scalar coefficient normalization
\(N^{2k-1}\) required by the lemma.
The Gaussian integral is
\[
 \int_H e^{-|v|^2}x_1^2\,\mathrm d\sigma_H(v)
 =\frac{\pi^{m-1/2}}2.
\]
Together with~\eqref{eq:weighted-sector-error-total}, the lemma yields
\begin{equation}\label{eq:weighted-sector-positive-limit}
\begin{aligned}
\lim_{N\to\infty}N^{-2k+m+1/2}\mathcal W_k^N(\varphi_U)
&=\Bigl(\frac{(m-1)!}{2\pi^k(m-k)!}\Bigr)^2
  \frac{\pi^{m-1/2}}2\\
&\qquad\times\int_{\partial U}f^2\,
 \mathrm d\sigma_{\partial U}>0.
\end{aligned}
\end{equation}

\subsection{Completion of the Bergman kernel criterion}

In each test class, the positive component limit and orthogonality
give a lower bound for the variance of every indicated truncation and
of the full statistic.  The full-statistic lower bound, together with
the kernel estimates, verifies the hypotheses of
Proposition~\ref{thm:correlation-volume-criterion}.

\begin{proof}[Proof of Proposition~\ref{prop:natural-scale-variance-lower-bound}]
For a smooth test, \eqref{eq:intrinsic-smooth-component-limit} gives
\eqref{eq:selected-component-positive-limit} for the index
in~\eqref{eq:terminal-wedge-power}.
For the domain test \(\varphi_U\), \eqref{eq:weighted-sector-positive-limit} gives
\eqref{eq:selected-component-positive-limit} with \(l_*=k\).
Proposition~\ref{prop:variance-lower-bound} then transfers the lower
bound to every indicated truncation, with constants independent of
the cutoffs.
\end{proof}

\begin{proof}[Proof of Theorem~\ref{thm:two-estimate-BF-criterion}]
Lemma~\ref{lem:BF-OD-consequences} gives the first two bounds
in~\eqref{eq:cv-basic-bounds} and~\eqref{eq:cv-small-correlation-balls} with
\(\varrho_N=r_N=N^{-d_\varphi/2}\) and exponent
\(\nu=d_\varphi/2>k-1\).  Taking all truncation indices equal to \(\infty\) in
Proposition~\ref{prop:natural-scale-variance-lower-bound} gives the third
bound in~\eqref{eq:cv-basic-bounds}, since \(v_N=a_N^2r_N\).
Proposition~\ref{thm:correlation-volume-criterion} therefore proves
the two-sided variance estimate and the central limit theorem.
\end{proof}

\begin{rmk}
\label{rmk:piecewise-boundary-extension}
The domain conclusions of Theorem~\ref{thm:two-estimate-BF-criterion}
and Main Theorems~\ref{thm:finite-point-conditioned-variance}
and~\ref{thm:finite-point-conditioned-clt} extend
under the following
geometric assumptions, with the same weight and nondegeneracy conditions.
Let \(U\subsetneq M\) be a proper domain with
\(\partial U=(\bigsqcup_{j=1}^J F_j)\sqcup\Sigma\), where the \(F_j\)
are relatively open embedded real \(\mathscr C^2\) hypersurface faces and
\(\Sigma\) is closed of Hausdorff dimension at most \(2m-2\).
Write \(\sigma_{\partial U}\) for the induced measure on the faces and
\(S_\delta=\{x\in\partial U:\operatorname{dist}(x,\Sigma)<\delta\}\).
Assume that, for some \(C_U,\rho_0>0\) and \(A_U>1\):
\begin{enumerate}[label=\textup{(\roman*)}]
\item The outward-oriented faces satisfy
\(\int_U\mathrm d\eta=\sum_j\int_{F_j}\eta\) for every smooth
\((2m-1)\)-form \(\eta\).
\item For all \(x\in M\) and \(0<\rho,\delta<\rho_0\),
\[
 \sigma_{\partial U}(\partial U\cap B(x,\rho))\leqslant C_U\rho^{2m-1},
 \qquad \sigma_{\partial U}(S_\delta)\leqslant C_U\delta.
\]
\item If \(x\in\partial U\setminus\Sigma\), \(0<\rho<\rho_0\), and
\(\operatorname{dist}(x,\Sigma)\geqslant A_U\rho\), then \(\partial U\cap B(x,\rho)\)
is a single graph, with \(U\) on one side, in the normal coordinates
of Lemma~\ref{lem:C2-boundary-geometry}, satisfying the uniform bounds
\eqref{eq:uniform-boundary-graph-bounds}.
\end{enumerate}
When \(\Sigma=\varnothing\), take \(S_\delta=\varnothing\) and impose
the graph condition at every boundary point without the distance restriction.
The \(L^2\) stability estimates in
Proposition~\ref{prop:finite-conditioning-stability} and the domain
conclusions of Proposition~\ref{prop:quantitative-conditional-comparison}
also hold, with the convention \(\chi_U=0\) on \(\partial U\).

These assumptions hold for the piecewise \(\mathscr C^2\) domains
without cusps considered in
\cite[Section~1 and the proof of Theorem~1.4]{MR2465693}.
At each boundary point, the polyhedral local model consists of a real \(\mathscr C^2\)-diffeomorphism
\(\rho:V\to\rho(V)\subset\mathbb R^{2m}\) and a closed, possibly
nonconvex, polyhedral cone \(K\) such that
\[
 \rho(U\cap V)=\rho(V)\cap\operatorname{int}K,
 \qquad
 \rho(\partial U\cap V)=\rho(V)\cap\partial K.
\]
See Appendix~\ref{app:piecewise-boundary-extension} for the verification
and proof of the extension.
\end{rmk}

\section{Kernel estimates and variance asymptotics}
\label{sec:gaussian-sections-verification}

The local Szeg\H{o} kernel expansion and the horizontal derivative
bounds verify \textup{(BF)} and \textup{(OD)} for the Gaussian section
ensemble.  For weighted domains, a boundary calculation at a fixed
chaos cutoff and a uniform tail estimate give the leading term of
the variance.

\subsection{Local expansion and off-diagonal kernel estimates}

Choose, on a finite cover, holomorphic normal coordinates and preferred
holomorphic frames depending smoothly on the centre \(z_0\), with the
normalization used in \textup{(BF)}.  These choices give Heisenberg
coordinates on the unit circle bundle \(\pi:X\to M\):
\[
 \phi_{x_0}(z,\theta)
 =e^{\sqrt{-1}\theta}\frac{e_L^*(z)}{|e_L^*(z)|_{h^*}},
 \qquad \phi_{x_0}(0,0)=x_0,\quad \pi(x_0)=z_0.
\]
All derivatives of any fixed order of the transition maps, metric
coefficients, and connection coefficients are uniformly bounded by compactness.
Write the lifted kernels in these coordinates as
\[
\Pi_{x_0}^N=\Pi_N\circ(\phi_{x_0}\times\phi_{x_0}),\qquad
\rho_{x_0}^N=\rho_N\circ(\phi_{x_0}\times\phi_{x_0}).
\]
Here \(\Pi_N\) is the Szeg\H{o} kernel on the circle bundle and \(\rho_N\)
is its diagonal normalization.  At zero fibre angles, the normalized
lifted kernel in the common preferred frame equals the normalized
Bergman kernel \(R_{N,z_0}\) by~\eqref{eq:Bergman-Szego-bridge}.
Derivatives \(D_{u,v}^\gamma\) are taken in the scaled spatial variables
\(u,v\), including their antiholomorphic directions.

The local expansion and off-diagonal horizontal derivative bounds for
\(\Pi_N\) in \cite[Theorem~2.4]{MR2465693}, based on
\cite{MR1794066,MR1887895}, give \textup{(i)} after diagonal
normalization and \textup{(ii)} directly.  We derive
\textup{(iii)}--\textup{(iv)} from these estimates.

\begin{thm}
\label{thm:expansion-szego-kernel}
For the positive line-bundle ensemble above, uniformly in \(z_0\), the
following hold.
\begin{enumerate}
\item[\textup{(i)}] \emph{Local expansion of the normalized kernel.}
For every \(b,\varepsilon>0\) and fixed integer
\(j\geqslant0\),
\[
 \rho_{x_0}^N\left(\frac u{\sqrt N},\frac\theta N;
                   \frac v{\sqrt N},\frac\vartheta N\right)
 =e^{\sqrt{-1}(\theta-\vartheta)+u\cdot\bar v
       -( |u|^2+|v|^2)/2}\,[1+\mathscr E_{N,x_0}(u,v)],
\]
\[
 |D_{u,v}^\gamma\mathscr E_{N,x_0}(u,v)|
 \leqslant C_{j,\varepsilon,b}N^{-1/2+\varepsilon},
 \quad |\gamma|=j,\quad |u|+|v|<b\sqrt{\log N}.
\]
\item[\textup{(ii)}] \emph{Off-diagonal decay of the unnormalized kernel.}
For \(j,q\geqslant0\), \(b>\sqrt{j+2q+2m}\), and every ordered
composition \(\mathcal H\) with \(|\mathcal H|=j\),
\[
 |\mathcal H\Pi_N(x,y)|=O(N^{-q})
 \qquad\text{if }\quad
 \operatorname{dist}(\pi(x),\pi(y))\geqslant b\sqrt{\frac{\log N}{N}},
\]
uniformly in the endpoints.
\item[\textup{(iii)}] \emph{Global derivative bound for the normalized kernel.}
For every \(B>0\), there are \(c,C>0\) and
\(D\geqslant0\) such that
\begin{equation}\label{eq:standard-horizontal-envelope}
 N^{-|\mathcal H|/2}|\mathcal H\rho_N(x,y)|
 \leqslant C\left[
 (1+\sqrt N\,\operatorname{dist}(z,w))^D e^{-cN \operatorname{dist}(z,w)^2}+N^{-B}\right]
\end{equation}
for \(z=\pi(x)\), \(w=\pi(y)\), whenever \(\mathcal H\) contains at
most two horizontal derivatives in each of the variables \(x\) and \(y\).
\item[\textup{(iv)}] \emph{Off-diagonal decay of normalized kernel derivatives.}
For every \(q\geqslant0\), one can choose \(\beta_q>0\) so that every
ordered composition \(\mathcal H\) with at most two horizontal derivatives
at each endpoint satisfies
\begin{equation}\label{eq:horizontal-normalized-kernel-far}
 |\mathcal H\rho_N(x,y)|=O(N^{-q})
 \quad\text{if}\quad
 \operatorname{dist}(\pi(x),\pi(y))\geqslant
 \beta_q\sqrt{\frac{\log N}{N}},
\end{equation}
uniformly over the horizontal lifts of the coordinate fields in the
fixed finite atlas.
\end{enumerate}
\end{thm}

\begin{proof}
The scaled Szeg\H{o}-kernel estimate
\cite[Theorem~2.4]{MR2465693} gives, for each fixed integer
\(K\geqslant1\) and \(|u|+|v|<b\sqrt{\log N}\),
\[
\begin{aligned}
 &N^{-m}\Pi_{x_0}^N\left(\frac u{\sqrt N},\frac\theta N;
                          \frac v{\sqrt N},\frac\vartheta N\right)\\
 &\qquad=\frac{e^{\sqrt{-1}(\theta-\vartheta)+u\cdot\bar v
        -( |u|^2+|v|^2)/2}}{\pi^m}
  \left[1+\sum_{\ell=1}^K N^{-\ell/2}p_\ell(u,v)
        +N^{-(K+1)/2}\mathscr E_{N,K}(u,v)\right],
\end{aligned}
\]
where the polynomials \(p_\ell\) may depend on the centre, with
uniformly bounded coefficients, and every fixed derivative of
\(\mathscr E_{N,K}\) is \(O(N^\varepsilon)\), uniformly
in the centre.  The off-diagonal estimate in that theorem is
precisely~\textup{(ii)}.  See also the
heat-kernel approach in~\cite{MR2339952}.

The diagonal Bergman expansion
\cite{MR1616718,zbMATH01341543} gives a smooth coefficient \(b_1\)
such that, for every fixed \(r\),
\[
 \mathsf b_N(z)=\frac{N^m}{\pi^m}
\left(1+\frac{b_1(z)}N+O_{C^r}(N^{-2})\right),
\qquad
 |D_z^\alpha(\mathsf b_N^{-1/2})|=O(N^{-m/2})\quad(|\alpha|\leqslant r).
\]
Divide the scaled Szeg\H{o} kernel by the square roots of its two
diagonal values.  The quotient rule and absorption of powers of
\(\log N\) into \(N^\varepsilon\) prove~\textup{(i)}.

To prove~\textup{(iii)}, we use this normalized expansion near the
diagonal and~\textup{(ii)} outside that region.  Choose \(q\geqslant B\)
and a threshold \(b_{\mathrm{far}}\) large enough for~\textup{(ii)}
with \(j\leqslant4\).  For each pair, centre holomorphic normal
coordinates at the first base point \(\pi(x)\).
Use~\textup{(i)} for \(|u|+|v|<b_{\mathrm{loc}}\sqrt{\log N}\), with
\(b_{\mathrm{loc}}>2b_{\mathrm{far}}\).  Uniform comparison of coordinate
norms and geodesic distance shows that, for all large \(N\), every pair
with \(\operatorname{dist}(z,w)\leqslant b_{\mathrm{far}}\sqrt{\log N/N}\)
is represented in this scaled window.
In these coordinates, the horizontal lift of a coordinate field has the form
\(D_\mu^{\mathrm H}=D_\mu+a_\mu(z,\bar z)\partial_\theta\), with \(a_\mu(0)=0\).
Under \((z,\theta)=(u/\sqrt N,\vartheta/N)\),
\[
\begin{aligned}
 N^{-1/2}D_\mu^{\mathrm H}
 &=D_{u,\mu}+b_{\mu,N}(u)\partial_\vartheta,\qquad
 |b_{\mu,N}(u)|\leqslant C|u|,\\
 |D_u^\gamma b_{\mu,N}(u)|&\leqslant C_s
 \quad(|\gamma|=s,\ 1\leqslant s\leqslant4).
\end{aligned}
\]
Here \(D_{u,\mu}\) is the \(u\)-coordinate field corresponding to
\(D_\mu\), either \(\partial/\partial u_a\) or
\(\partial/\partial\bar u_a\).
Use~\textup{(i)} with a fixed \(0<\varepsilon<1/2\).
By the product rule, at most two horizontal derivatives in each
variable multiply the Gaussian bound by a polynomial in \(|u|+|v|\),
with degree and coefficients uniform in \(N\) and the centre.
Setting the first scaled variable to zero after differentiation gives
\[
 N^{-|\mathcal H|/2}|\mathcal H\rho_N(x,y)|
 \leqslant C(1+|v|)^D e^{-|v|^2/2}
\]
for \(|v|<b_{\mathrm{loc}}\sqrt{\log N}\).  Uniform comparison of
the coordinate norm \(|v|/\sqrt N\) and geodesic distance gives the
Gaussian term in~\eqref{eq:standard-horizontal-envelope}.  With the
centre held fixed, the horizontal coordinate fields of the fixed atlas
are linear combinations of the centred fields with uniformly bounded
coefficients and derivatives on the relevant overlaps.  Expanding a
composition of order \(j\) gives terms of order \(h\leqslant j\), with at most two
derivatives at each endpoint.  Each term obeys the same envelope after
normalization, since \(N^{-j/2}N^{h/2}\leqslant1\).

For the remaining pairs, the choices of \(b_{\mathrm{loc}}\) and
\(b_{\mathrm{far}}\) allow us to apply~\textup{(ii)}.
Apply Leibniz' rule to
\[
\rho_N(x,y)=\Pi_N(x,y)
\mathsf b_N(\pi(x))^{-1/2}\mathsf b_N(\pi(y))^{-1/2}.
\]
The diagonal derivative estimate and~\textup{(ii)} give
the \(N^{-B}\) term.  Hermitian symmetry handles the exchanged
endpoints, proving~\textup{(iii)}.

For~\textup{(iv)}, apply~\textup{(iii)} with \(B>q+2\).
Since \(|\mathcal H|\leqslant4\), multiplying by
\(N^{|\mathcal H|/2}\) contributes at most \(N^2\).  Choosing \(\beta_q\) sufficiently large makes the
Gaussian term, including its polynomial factor, \(O(N^{-q-2})\)
when \(\sqrt N\,\operatorname{dist}(z,w)\geqslant\beta_q\sqrt{\log N}\).
This proves~\eqref{eq:horizontal-normalized-kernel-far}.
\end{proof}

\subsection{Verification of \textup{(BF)} and \textup{(OD)}}

For \textup{(BF)}, the scaled kernel uses one preferred frame at both
endpoints.  For \textup{(OD)}, we compare horizontal derivatives with
ordinary derivatives in frames centred separately at the two endpoints.

\begin{rmk}
Ordinary kernel derivatives depend on the local frame.  For example,
in the standard affine Fubini--Study frame on \(\mathbb{CP}^1\),
\(\left.\partial_z\rho_N(z,0;w,0)\right|_{z=w=a}
=N\bar a/[2(1+|a|^2)]\).
In the separately centred endpoint frames used in the criterion,
Lemma~\ref{lem:Bergman-horizontal-comparison} compares ordinary and
horizontal derivatives at the \(N^{1/2}\) scale, with at most two
derivatives in each variable.
\end{rmk}

Setting \(\theta=\vartheta=0\) in
Theorem~\ref{thm:expansion-szego-kernel}\textup{(i)} and using
\eqref{eq:Bergman-Szego-bridge} gives, uniformly in \(z_0\),
\[
 \mathcal R_{N,z_0}(u,v)
 =\mathcal R_{\mathrm{BF}}(u,v)[1+\mathscr E_{N,x_0}(u,v)].
\]
With \(0<\varepsilon<1/2\), the relative error tends to zero uniformly
in \(z_0\) and \((u,v)\in D_T\times D_T\) for each fixed \(T\).  The
bounded model kernel then gives the \(L^2\) convergence in \textup{(BF)}.

For~\textup{(OD)}, fix \(A>2m\), choose \(B>A/2\), and set
\(s=\sqrt N\,\operatorname{dist}(z,w)\).  The normalized kernel derivatives satisfy
\[
 J_N(z,w)\leqslant
 C\bigl[(1+s)^D e^{-cs^2}+N^{-B}\bigr]
\]
by Lemma~\ref{lem:Bergman-horizontal-comparison} and
\eqref{eq:standard-horizontal-envelope}.  The first term is \(O((1+s)^{-A})\).  Since \(s\leqslant\operatorname{diam}(M)\sqrt N\),
\(N^{-B}(1+s)^A\leqslant C_MN^{A/2-B}\leqslant C_M\).
Thus \(J_N(z,w)\leqslant C_A(1+s)^{-A}\).  This proves
condition~\textup{(OD)}, including derivatives of order at most two in
each variable.

Theorem~\ref{thm:two-estimate-BF-criterion} therefore gives the two-sided
variance estimates and CLT for smooth tests with \(k\geqslant2\) and
domain tests with \(k\geqslant1\).
The smooth \(k=1\) case follows from
\cite[Theorems~1.1--1.2]{MR2742043}.

\subsection{Leading variance for weighted domains}

For a fixed cutoff \(q\), the variance divided by \(N^{2k-m-1/2}\)
tends to a universal constant times the boundary integral of \(f^2\).
A tail estimate uniform in \(N\) identifies the limit of these constants
as \(q\to\infty\) from the unweighted case \(f\equiv1\) and gives the
leading term of the full variance for general \(f\).

\begin{pro}
\label{prop:weighted-domain-variance-asymptotics}
For every fixed admissible domain test \(\varphi_U\) with weight \(f\),
\[
 \operatorname{Var}X_N^{\varphi_U}
 =N^{2k-m-\frac12}\left[
 \nu_{m,k}\int_{\partial U}f^2\,\mathrm d\sigma_{\partial U}+o(1)\right],
\]
where \(\nu_{m,k}>0\) is the constant in
Theorem~\ref{thm:SZ-variance-asymptotics}\textup{(ii)}.
\end{pro}

\begin{proof}
\smallskip\noindent\emph{Boundary reduction.}
Fix an integer \(q\geqslant1\).  With \(Q_N^{\leqslant q}\)
from~\eqref{eq:finite-bi-potential}, set
\[
 A_{N,q}:=\mathscr D_{12}Q_N^{\leqslant q},\qquad
 \Phi_0:=\frac{\pi_1^*\omega^{m-k}\wedge\pi_2^*\omega^{m-k}}
                   {((m-k)!)^2}.
\]
The projection identity and
Proposition~\ref{prop:finite-cutoff-two-point-functional}\textup{(ii)} give
\[
 \operatorname{Var}\widehat X_{N,q}^{\varphi_U}
 =\sum_{j=1}^k\binom{k}{j}
  \int_{U^2}f(z)f(w)A_{N,q}^{\wedge j}
      \wedge D_N^{\wedge(k-j)}\wedge\Phi_0.
\]
In each summand, extract \(\mathscr D_{12}Q_N^{\leqslant q}\)
from one factor \(A_{N,q}\); all remaining factors are separately
\(\partial_i\)- and \(\bar\partial_i\)-closed.  Lemma~\ref{lem:boundary-Stokes}\textup{(ii)}
leaves the boundary--boundary contribution
\(\int_{(\partial U)^2}f(z)f(w)\Lambda_{N,q}\), where
\begin{equation}\label{eq:full-cutoff-boundary-form}
 \Lambda_{N,q}:=-\sum_{j=1}^k\binom{k}{j}
  \bar\partial_1\bar\partial_2Q_N^{\leqslant q}
  \wedge A_{N,q}^{\wedge(j-1)}
  \wedge D_N^{\wedge(k-j)}\wedge\Phi_0.
\end{equation}
The minus sign follows from
\(\mathscr D_{12}=-\partial_1\bar\partial_1\partial_2\bar\partial_2\).
Every other term contains a derivative of \(f\) and at least one
integration over \(U\).

Recall \(\mathfrak h_N(z,w)=(1+\sqrt N\,\operatorname{dist}(z,w))^{-A}\),
where \(A>2m\).  Since \(Q_N^{\leqslant q}\) is a polynomial in \(\tau_N\),
\textup{(OD)}, \eqref{eq:weighted-tau-jets}, and the mean-current bound give
\begin{equation}\label{eq:full-cutoff-boundary-bound}
 |\Lambda_{N,q}(z,w)|\leqslant C_qN^{2k-1}\mathfrak h_N(z,w).
\end{equation}
The differentiated kernel forms in the other terms have the same bound.
Integrating first over a volume variable gives a factor \(O(N^{-m})\),
by~\eqref{eq:BF-absolute-tail} and the small-ball volume bound.
Their total is therefore \(O_{q,f}(N^{2k-m-1})\).
Combining the boundary and interior contributions gives
\[
 \operatorname{Var}\widehat X_{N,q}^{\varphi_U}
 =\int_{(\partial U)^2}f(z)f(w)\Lambda_{N,q}
   +O_{q,f}(N^{2k-m-1}).
\]

\smallskip\noindent\emph{Limit at a fixed cutoff.}
Write the restriction of \(\Lambda_{N,q}\) to the oriented boundary
product as
\[
 \Lambda_{N,q}
 =\lambda_{N,q}(z,w)\,
 \mathrm d\mathsf{Vol}_{\partial U}(z)
 \wedge\mathrm d\mathsf{Vol}_{\partial U}(w).
\]
We verify the local density assumption of
Lemma~\ref{lem:weighted-boundary-localization}.  Fix \(L>1\), use normal coordinates centred at
\(z\), and pull the second variable back to the tangent hyperplane
\(H=\mathbb R\times\mathbb C^{m-1}\) using
Lemma~\ref{lem:uniform-scaled-boundary-chart}.
In~\eqref{eq:full-cutoff-boundary-form}, the mixed second derivative
has scale \(N\), and each of the other \(k-1\) two-point factors
\(A_{N,q}\) or \(D_N\) has scale \(N^2\).  This gives the scalar
density scale \(N^{2k-1}\) before pulling back the surface measure.
The second-variable surface Jacobian contributes \(N^{-m+1/2}\)
in the integral, giving the variance scale \(N^{2k-m-1/2}\).
Lemma~\ref{lem:weak-BF-jet-recovery} and~\eqref{eq:BF-mean-limit} give
\[
 N^{-(2k-1)}\lambda_{N,q}
 \bigl(z,\delta_{N,z}(\mathfrak G_{N,z}(v))\bigr)
 =b_{m,k,q}(v)+o_{q,L}(1),
\]
uniformly in \(z\in\partial U\) and \(v\in H\cap B(0,L)\).
Here \(\delta_{N,z}\) is the coordinate rescaling map and
\(\mathfrak G_{N,z}\) the inverse boundary chart of that lemma.
The scalar profile \(b_{m,k,q}\) is the density of the limiting form
obtained from~\eqref{eq:full-cutoff-boundary-form}
by using the Euclidean K\"ahler form in both variables, its quotient by
\(\pi\) for each mean-current factor, and the potential
\[
 Q_q^{\mathrm{BF}}(u,v):=\frac1{4\pi^2}
   \sum_{\alpha=1}^q\frac{e^{-\alpha|u-v|^2}}{\alpha^2}
\]
in place of \(Q_N^{\leqslant q}\), with
\(A_{N,q}\) replaced by \(\mathscr D_{12}Q_q^{\mathrm{BF}}\).
The Dolbeault derivatives are taken before restriction to \(H\times H\);
the first-variable coefficients are then evaluated at \(u=0\).
Each resulting coefficient is a finite sum of polynomials times
decaying Gaussians, and is integrable in \(v\in H\).
The Euclidean K\"ahler form and \(e^{-|u-v|^2}\) are invariant under
simultaneous unitary changes of coordinates.  Since \(U(m)\) acts
transitively on unit real normals, with the domain side and boundary
orientation transported together,
\(c_{m,k,q}:=\int_H b_{m,k,q}(v)\,\mathrm d\sigma_H(v)\)
depends only on \(m,k,q\).
Lemma~\ref{lem:weighted-boundary-localization}, applied with this
profile and the bound~\eqref{eq:full-cutoff-boundary-bound}, therefore gives
\begin{equation}\label{eq:weighted-cutoff-variance-limit}
 \lim_{N\to\infty}N^{-2k+m+1/2}
 \operatorname{Var}\widehat X_{N,q}^{\varphi_U}
 =c_{m,k,q}\int_{\partial U}f^2\,\mathrm d\sigma_{\partial U}.
\end{equation}

\smallskip\noindent\emph{Removal of the cutoff and identification of the constant.}
Apply Lemmas~\ref{lem:cv-negative-moments}
and~\ref{lem:automatic-kernel-tail}, with the bounds of
Lemma~\ref{lem:BF-OD-consequences} and \(\beta=1/4\).
This choice is permitted since \(m-k+\tfrac12\geqslant\tfrac12\).
Orthogonality and~\eqref{eq:common-cutoff-error} give, uniformly in \(q\geqslant1\),
\begin{equation}\label{eq:weighted-variance-cutoff-error}
\begin{aligned}
0\leqslant\operatorname{Var}X_N^{\varphi_U}
 -\operatorname{Var}\widehat X_{N,q}^{\varphi_U}
 &=\|\widehat X_N^{\varphi_U}-\widehat X_{N,q}^{\varphi_U}\|_2^2\\
 &\leqslant C_fN^{2k-m-1/2}q^{-1/4}.
\end{aligned}
\end{equation}
Apply~\eqref{eq:weighted-cutoff-variance-limit} and
\eqref{eq:weighted-variance-cutoff-error} first with \(f\equiv1\).
The unweighted asymptotic in
Theorem~\ref{thm:SZ-variance-asymptotics}\textup{(ii)}, together with
\(\sigma_{\partial U}(\partial U)>0\), yields
\[
 0\leqslant\nu_{m,k}-c_{m,k,q}\leqslant C_Uq^{-1/4}.
\]
For the original weight, the same two formulas therefore imply
\[
\begin{aligned}
 &\limsup_{N\to\infty}\left|
 N^{-2k+m+1/2}\operatorname{Var}X_N^{\varphi_U}
 -\nu_{m,k}\int_{\partial U}f^2\,\mathrm d\sigma_{\partial U}\right|\\
 &\qquad\leqslant C_fq^{-1/4}
 +|c_{m,k,q}-\nu_{m,k}|
    \int_{\partial U}f^2\,\mathrm d\sigma_{\partial U}
 \longrightarrow0\qquad(q\to\infty).
\end{aligned}
\]
\end{proof}

\section{Stability under finite-point conditioning}\label{sec:finite-conditioning}

Retain the conditional ensemble on \(\mathcal H_{N,\mathcal P}\), the statistic
\(X_{N,\mathcal P}^{\varphi}\), and the conventions of the Introduction.
Fix \(1\leqslant k\leqslant m\), an admissible test form \(\varphi\), and
a finite set \(\mathcal P\) of distinct points, independent of \(N\).
We prove Main Theorems~\ref{thm:finite-point-conditioned-variance}
and~\ref{thm:finite-point-conditioned-clt} by showing that the \(L^2\)
difference between the conditional and unconditioned statistics under
an orthogonal coupling is small relative to the unconditioned standard
deviation.  The same coupling supplies the identity used to compute the
mean correction in Section~\ref{sec:conditional-mean-asymptotics}.

\subsection{Conditional intersections and coupling}

We first state the quantitative comparison, then introduce the coupling
and record the conditional intersection facts used in its proof.

\begin{pro}
\label{prop:quantitative-conditional-comparison}
Under the hypotheses of
Main Theorem~\ref{thm:finite-point-conditioned-variance},
the variance ratio satisfies
\begin{equation}\label{eq:conditional-unconditional-variance-ratio}
 \frac{\operatorname{Var}_{\mathcal P}X_{N,\mathcal P}^{\varphi}}
      {\operatorname{Var}X_N^\varphi}
 =1+
 \begin{cases}
 O(N^{-m/2}),&\textup{(S)},\\
 O\bigl(N^{-m/2+1/4}(\log N)^m\bigr),&\textup{(D)}.
 \end{cases}
\end{equation}
The implied constants may depend on \(\varphi\), \(\mathcal P\), and
the fixed geometry, including \(U,f\) in type~\textup{(D)}, but not
on \(N\).  The conditional and unconditioned means satisfy
\[
 \bigl|\mathbb E_{\mathcal P}X_{N,\mathcal P}^{\varphi}
       -\mathbb EX_N^\varphi\bigr|
 =o\!\left(\sqrt{\operatorname{Var}X_N^\varphi}\right).
\]
The convergence in~\eqref{eq:conditional-statistic-clt} also holds
with the unconditioned mean, standard deviation, or both.
\end{pro}

Use the domain convention of
Definition~\ref{def:admissible-zero-statistics}\textup{(D)}.
When \(k=m\), prescribed zeros are deterministic simple atoms almost
surely (Lemma~\ref{lem:conditional-intersection-facts}).
Changing \(\chi_U\) at one of these points on \(\partial U\) adds a
deterministic term and preserves the centred law.

For all sufficiently large \(N\), finite-point conditioning removes a
fixed-dimensional subspace of sections.  After normalization by the
unconditioned diagonal densities, the kernel of this subspace decays
away from the prescribed points.

For the subspace \(\mathcal H_{N,\mathcal P}\) in
\eqref{eq:finite-conditioning-subspace}, choose independent standard
Gaussian sections
\(s_{j,N}^{\mathcal P}\in\mathcal H_{N,\mathcal P}\) and
\(c_{j,N}^{\mathcal P}\in\mathcal H_{N,\mathcal P}^{\perp}\),
for \(1\leqslant j\leqslant k\), with the pairs independent for
different \(j\).  The coupled realizations of the unconditioned
sections \(s_j^N\) are
\begin{equation}\label{eq:orthogonal-conditional-coupling}
 s_j^N=s_{j,N}^{\mathcal P}+c_{j,N}^{\mathcal P},
 \qquad 1\leqslant j\leqslant k.
\end{equation}
The statistics \(X_{N,\mathcal P}^{\varphi}\) and \(X_N^\varphi\)
are coupled using these sections.

Let \(d_{N,\mathcal P}:=\dim\mathcal H_{N,\mathcal P}\) and choose an
orthonormal basis
\(\{S_1^{N,\mathcal P},\ldots,S_{d_{N,\mathcal P}}^{N,\mathcal P}\}\)
of \(\mathcal H_{N,\mathcal P}\).  The conditional diagonal density is
\[
 \mathsf b_N^{\mathcal P}(z)
 :=\sum_{j=1}^{d_{N,\mathcal P}}|S_j^{N,\mathcal P}(z)|_{h_N}^2.
\]
Shiffman, Zelditch and Zhong~\cite[Lemma~3.9]{MR2806467} obtained the
single-divisor identity
\[
 \mathcal C_{0,\mathcal P}^N
 :=\mathbb E_{\mathcal P}\bigl[[Z_{s_N^{\mathcal P}}]\bigr]
 =\frac{N}{\pi}\omega
 +\frac{\sqrt{-1}}{2\pi}\partial\bar\partial\log\mathsf b_N^{\mathcal P}.
\]
It holds as currents on \(M\) for all sufficiently large \(N\) and
reduces to~\eqref{eq:C0} when \(\mathcal P=\varnothing\).

\begin{lem}
\label{lem:conditional-intersection-facts}
For all sufficiently large \(N\), any collection of at most \(m\)
independent sections, each drawn from the conditional or unconditioned
ensemble, has almost-sure proper intersections.  An entirely conditional
system of \(m\) sections has a simple zero at every point of \(\mathcal P\);
a system containing an unconditioned section almost surely misses
\(\mathcal P\).

The mean conditional divisor current \(\mathcal C_{0,\mathcal P}^N\)
is smooth on \(M\setminus\mathcal P\).  For every
\(I\subset\{1,\ldots,k\}\), one has, on this open set,
\[
 \mathbb E_{\mathcal P}\left[\bigwedge_{i\in I}
 [Z_{s_{i,N}^{\mathcal P}}]\right]
 = (\mathcal C_{0,\mathcal P}^N)^{\wedge|I|}.
\]
The trace measures of these random intersections almost surely assign
zero mass to \(\partial U\setminus\mathcal P\), for each fixed domain
\(U\) of type~\textup{(D)}.
\end{lem}

\begin{proof}
Let \(\mathcal I_{\mathcal P}\) be the ideal sheaf of the reduced
finite set \(\mathcal P\).  For large \(N\), Serre's theorem makes
\(\mathcal I_{\mathcal P}\otimes L^N\) globally generated.  At
\(p\in\mathcal P\), its fibre is
\((\mathfrak m_p/\mathfrak m_p^2)\otimes L_p^N\), where
\(\mathfrak m_p\) is the maximal ideal at \(p\).  Thus, for each prescribed point \(p\), the map sending a section
to its first derivatives at \(p\) is surjective.  For \(h\leqslant m\)
independent conditional sections, the gradient matrix at each such
point has rank \(h\) almost surely.  Since \(\mathcal P\) is finite,
this holds simultaneously at every prescribed point; when \(h=m\),
the intersection multiplicity is one.  Global generation and
independence give proper intersections away from \(\mathcal P\).
A system containing an unconditioned section almost surely has no
common zero in \(\mathcal P\), since the value of that section at each
prescribed point has positive variance.

Away from \(\mathcal P\), global generation makes the mean divisor
current smooth.  For compactly supported test forms, local
Poincar\'e--Lelong and successive conditioning, using the pairing
argument in Proposition~\ref{prop:blockwise-chaos-projection}, give the
displayed expectation identity.  The expected trace measure has a smooth density;
positivity therefore gives zero boundary mass on a countable relatively
compact cover of \(M\setminus\mathcal P\).
\end{proof}

\begin{pro}
\label{prop:finite-conditioning-stability}
Under this orthogonal coupling, for all sufficiently large \(N\),
\begin{equation}\label{eq:finite-conditioning-stability}
 \left\|X_{N,\mathcal P}^{\varphi}-X_N^\varphi\right\|_2
 \leqslant
 \begin{cases}
  C_{\varphi,\mathcal P}N^{k-m-1},&\textup{(S)},\\
  C_{\varphi,\mathcal P}N^{k-m}(\log N)^m,&\textup{(D)}.
 \end{cases}
\end{equation}
The same bounds hold for the difference of the centred statistics.
\end{pro}

The proofs of Propositions~\ref{prop:quantitative-conditional-comparison}
and~\ref{prop:finite-conditioning-stability} are given in
Section~\ref{sec:comparison-coupled}.

\subsection{The Schur complement and logarithmic norm differences}

The Schur complement gives the ratio of the conditional to the
unconditioned diagonal density.  Away from the prescribed points, we use
this ratio to prove convergence of the normalized conditional mean divisor
current and to compute the pointwise \(L^2\) norm of the logarithmic
difference under the orthogonal coupling.

For nonempty \(\mathcal P\), write
\begin{equation*}
 \theta_{N,\mathcal P}(z)
 :=\frac{\mathsf b_N^{\mathcal P}(z)}{\mathsf b_N(z)},
 \qquad \operatorname{dist}_{\mathcal P}(z):=\operatorname{dist}(z,\mathcal P).
\end{equation*}
When \(\mathcal P=\varnothing\), the conditional ensemble is the
unconditioned ensemble, and we take \(\theta_{N,\varnothing}\equiv1\).
For \(\mathcal P=\{p_1,\ldots,p_r\}\neq\varnothing\), choose unit
vectors in the relevant fibres to represent \(R_N\) as a scalar kernel,
and set
\[
 (\mathbf G_N)_{ij}:=R_N(p_i,p_j),\qquad
 (\mathbf r_N(z))_i:=R_N(p_i,z),\qquad 1\leqslant i,j\leqslant r.
\]
The Schur-complement identity is
\begin{equation}\label{eq:conditional-Schur-complement}
 1-\theta_{N,\mathcal P}(z)
 =\mathbf r_N(z)^*\mathbf G_N^{-1}\mathbf r_N(z).
\end{equation}
This identity is independent of the chosen unit vectors in the fibres.

\begin{proof}[Proof of Theorem~\ref{thm:conditional-mean-equidistribution}]
For \(\mathcal P=\varnothing\), this is the unconditioned mean limit
from the Introduction.  For nonempty \(\mathcal P\), the
Schur-complement formula~\eqref{eq:conditional-Schur-complement}
(cf.~\cite[pp.~771--772]{MR2806467}) and the off-diagonal derivative estimates
of Theorem~\ref{thm:expansion-szego-kernel}\textup{(ii)} give
\(\mathsf b_N^{\mathcal P}/\mathsf b_N=1+O(N^{-Q})\) in every fixed
\(\mathscr C^r\)-norm on compact subsets of \(M\setminus\mathcal P\),
for any \(Q>0\).  Thus
\(\mathcal C_{0,\mathcal P}^N/N\to\omega/\pi\) smoothly there.
Lemma~\ref{lem:conditional-intersection-facts} identifies the expected
intersection current on this open set with
\((\mathcal C_{0,\mathcal P}^N)^{\wedge k}\), giving the asserted
smooth limit.

The normalized expected intersection currents are positive and have
cohomology class \(c_1(L)^{\smile k}\).  Their trace masses, obtained by
pairing with \(\omega^{m-k}/(m-k)!\), therefore equal the trace mass of
\(\omega^k/\pi^k\).  This mass bound gives weak subsequential limits.
Each limit agrees with \(\omega^k/\pi^k\) off \(\mathcal P\), where the
smooth limiting trace measure already has the entire mass.  The trace
measure of the limit assigns zero mass to \(\mathcal P\), so positivity
identifies the limit on all of \(M\).
\end{proof}

\begin{lem}
\label{lem:finite-defect-log-profile}
Fix \(\mathcal P\neq\varnothing\) and \(A>2m\).  There are constants
\(c,C>0\), allowed to depend on \(A\), such that, for all sufficiently
large \(N\),
\begin{align}
 c\min\{N \operatorname{dist}_{\mathcal P}(z)^2,1\}
 &\leqslant\theta_{N,\mathcal P}(z)
 \leqslant C\min\{N \operatorname{dist}_{\mathcal P}(z)^2,1\},
 \label{eq:conditional-diagonal-quadratic}\\
 1-\theta_{N,\mathcal P}(z)
 &\leqslant C(1+\sqrt N\,\operatorname{dist}_{\mathcal P}(z))^{-2A},
 \label{eq:conditional-leverage-decay}
\end{align}

For one pair \((s_N^{\mathcal P},s^N)\) in the orthogonal
coupling~\eqref{eq:orthogonal-conditional-coupling} and
\(z\notin\mathcal P\), set
\[
 \Lambda_{N,\mathcal P}(z):=
 \bigl\|\log|s_N^{\mathcal P}(z)|_{h_N}
       -\log|s^N(z)|_{h_N}\bigr\|_{L^2(\Omega)}.
\]
This \(L^2\)-norm satisfies
\begin{equation}\label{eq:log-defect-profile}
 \Lambda_{N,\mathcal P}(z)\leqslant C
 \begin{cases}
  1+|\log(\sqrt N\,\operatorname{dist}_{\mathcal P}(z))|,
      &\sqrt N\,\operatorname{dist}_{\mathcal P}(z)\leqslant2,\\
  (1+\sqrt N\,\operatorname{dist}_{\mathcal P}(z))^{-A}
       \sqrt{\log(2+\sqrt N\,\operatorname{dist}_{\mathcal P}(z))},
      &\sqrt N\,\operatorname{dist}_{\mathcal P}(z)>2.
 \end{cases}
\end{equation}
\end{lem}

\begin{proof}
Because the points of \(\mathcal P\) are fixed and distinct,
off-diagonal decay gives \(\mathbf G_N=I+o(1)\), with uniformly bounded
inverse.  Equation~\eqref{eq:conditional-Schur-complement} then implies
\eqref{eq:conditional-leverage-decay}.  In holomorphic normal coordinates
\(\kappa_p\) centred at \(p\in\mathcal P\), use the scaling map
\(\delta_{N,p}(u)=\kappa_p^{-1}(u/\sqrt N)\) introduced in
Section~\ref{sec:two-point-nondegeneracy}.
The scaled Szeg\H{o}-kernel expansion and the same Schur complement
give, uniformly on every fixed ball,
\[
 \theta_{N,\mathcal P}(\delta_{N,p}(u))
 =1-e^{-|u|^2}+o_{C^2}(1).
\]
For \(e_{N,p}(u):=\theta_{N,\mathcal P}(\delta_{N,p}(u))
-(1-e^{-|u|^2})\), we have
\(e_{N,p}(0)=0\), \(\mathrm De_{N,p}(0)=0\), and
\(\|\mathrm D^2e_{N,p}\|_{\mathscr C^0}=o(1)\).
Taylor's theorem gives \(|e_{N,p}(u)|\leqslant o(1)|u|^2\) on a
small fixed ball, proving comparability with \(|u|^2\) up to the
prescribed point.  Compactness on the remaining fixed annulus and
\eqref{eq:conditional-leverage-decay} outside a sufficiently large fixed
ball prove~\eqref{eq:conditional-diagonal-quadratic}.

At a fixed \(z\notin\mathcal P\), divide the two scalar section values in
\eqref{eq:orthogonal-conditional-coupling} by their respective standard
deviations.  Their correlation has
squared modulus \(\theta=\theta_{N,\mathcal P}(z)\), while the
difference between the logarithms of their standard deviations is
\(\frac12\log\theta\).  We use the notation
\[
\operatorname{Li}_2(t):=\sum_{\alpha=1}^\infty\frac{t^\alpha}{\alpha^2},
\qquad 0\leqslant t\leqslant1,
\]
for the dilogarithm on the real interval used here.  The expansion in
the Wick polynomials \(:|\xi|^{2\alpha}:\) from
Section~\ref{sec:projection-foundation} gives
\(\operatorname{Cov}(\log|G_1|,\log|G_2|)
=\frac14\operatorname{Li}_2(\theta)\)
for the resulting pair \(G_1,G_2\) of standard complex Gaussian
variables.  The \(L^2\)-norm of the logarithmic difference consequently satisfies
\begin{equation}\label{eq:exact-log-defect-L2}
 \Lambda_{N,\mathcal P}(z)^2
 =\frac14\log^2\theta
  +\frac12\{\operatorname{Li}_2(1)-\operatorname{Li}_2(\theta)\}.
\end{equation}
Near \(\mathcal P\), use
\eqref{eq:conditional-diagonal-quadratic}.  For
\(0<\lambda\leqslant1/2\),
\[
 \operatorname{Li}_2(1)-\operatorname{Li}_2(1-\lambda)
 \leqslant C\lambda(1+|\log\lambda|),
 \qquad |\log(1-\lambda)|\leqslant C\lambda.
\]
The bounds extend continuously to \(\lambda=0\).  Set
\(s=\sqrt N\,\operatorname{dist}_{\mathcal P}(z)\).  By
\eqref{eq:conditional-leverage-decay}, there is a fixed \(R>2\)
such that \(\lambda=1-\theta\leqslant1/2\) for \(s\geqslant R\).
The displayed inequalities then give the second branch of
\eqref{eq:log-defect-profile} there.  On \(2\leqslant s\leqslant R\),
\eqref{eq:conditional-diagonal-quadratic} bounds \(\theta\) away from
zero, so \eqref{eq:exact-log-defect-L2} bounds \(\Lambda_{N,\mathcal P}\).
Increasing \(C\) covers this fixed annulus and proves
\eqref{eq:log-defect-profile}.
\end{proof}

\subsection{Local moments for conditional and unconditioned divisors}

The unweighted local mass bound applies to intersections of every
codimension \(q\leqslant m\).  For domain tests we use it on the
intersection of \(k\) divisors near the prescribed points, where it
also bounds the mass of prescribed atoms when \(k=m\).  The smooth
comparison uses a logarithmically weighted mass bound on mixed
intersections of \(k-1\) divisors, which have positive dimension.

\begin{lem}
\label{lem:mixed-local-intersection-mass}
Fix \(0\leqslant q\leqslant m\).  For this lemma, take \(q\)
independent copies \((s_{\ell,N}^{\mathcal P},s_\ell^N)\),
\(1\leqslant\ell\leqslant q\), of the coupling
\eqref{eq:orthogonal-conditional-coupling}.  For each \(\ell\), let
\(r_{\ell,N}\) be either member of the \(\ell\)-th pair, and set
\[
 \mu_{N}^{(q)}
 :=[Z_{r_{1,N}}]\wedge\cdots\wedge[Z_{r_{q,N}}]
       \wedge\frac{\omega^{m-q}}{(m-q)!}.
\]
There are \(r_0,C>0\), depending only on the fixed geometry,
\(\mathcal P\), and \(q\), such that, for all sufficiently large \(N\),
uniformly over the choices of conditional and unconditioned sections,
\begin{equation}\label{eq:mixed-local-mass-L2}
 \sup_{x\in M}
 \bigl\|\mu_N^{(q)}(B(x,r_0N^{-1/2}))\bigr\|_{L^2(\Omega)}
 \leqslant C N^{q-m}.
\end{equation}
Moreover, uniformly for \(x\in M\) and \(R\geqslant1\),
\begin{equation}\label{eq:mixed-growing-ball-mass-L2}
 \bigl\|\mu_N^{(q)}(B(x,RN^{-1/2}))\bigr\|_{L^2(\Omega)}
 \leqslant C N^{q-m}(1+R)^{2m}.
\end{equation}
If \(q\leqslant m-1\), then also, uniformly for \(p\in\mathcal P\),
\begin{equation}\label{eq:weighted-mixed-local-mass-L2}
 \left\|
  \int_{B(p,r_0N^{-1/2})}
  \bigl(1+|\log(\sqrt N\,\operatorname{dist}(z,p))|\bigr)
  \,\mathrm d\mu_N^{(q)}(z)
 \right\|_{L^2(\Omega)}
 \leqslant C N^{q-m}.
\end{equation}
For \(q=0\), the measure \(\mu_N^{(0)}\) is the deterministic
K\"ahler volume measure.
\end{lem}

\begin{proof}
The case \(q=0\) is immediate; assume \(q\geqslant1\).
For this proof, \(\Omega_\ell\) denotes the probability factor
carrying the \(\ell\)-th coupled pair, \(1\leqslant\ell\leqslant q\).
In holomorphic normal coordinates and a preferred holomorphic frame
\(e_x\) at \(x\), set
\(z=\delta_{N,x}(u)\) and divide each local coefficient by
\(\sqrt{\mathsf b_N(x)}\), obtaining \(f_{\ell,N}^{\mathrm{loc}}(u)\).
Using the unconditioned normalization for both ensembles leaves the
diagonal ratio \(\theta_{N,\mathcal P}\) explicit.
Write \(B_R:=B(0,R)\subset\mathbb C^m\) for a Euclidean ball in the
scaled coordinates.  We work on \(B_6\) and estimate currents on \(B_4\).
A conditional representative satisfies
\[
 \mathbb E|f_{\ell,N}^{\mathrm{loc}}(u)|^2
 =\frac{\mathsf b_N(z)}{\mathsf b_N(x)}
 |e_x(z)|_h^{-2N}\theta_{N,\mathcal P}(z),
 \qquad z=\delta_{N,x}(u).
\]
The unconditioned formula has \(\theta_{N,\mathcal P}\) replaced by one.
The diagonal expansion and the preferred-frame normalization bound the
first two factors above and below uniformly on \(B_4\).  In particular,
\(c\leqslant\mathbb E|f_{\ell,N}^{\mathrm{loc}}(u)|^2\leqslant C\)
for an unconditioned representative.  For large \(N\), at most one point
of \(\mathcal P\) has scaled coordinate in \(B_5\).  If its scaled
coordinate is \(a\), Lemma~\ref{lem:finite-defect-log-profile}
and uniform comparison of coordinate norms and geodesic distance give,
on \(B_4\),
\begin{equation*}
 c\min\{|u-a|^2,1\}
 \leqslant\mathbb E|f_{\ell,N}^{\mathrm{loc}}(u)|^2
 \leqslant C\min\{|u-a|^2,1\}.
\end{equation*}
If no point of \(\mathcal P\) has scaled coordinate in \(B_5\), the
unit separation between \(B_4\) and the complement of \(B_5\), together
with Lemma~\ref{lem:finite-defect-log-profile}, gives uniform upper and
lower variance bounds on \(B_4\).
Set \(\log_+t:=\max\{\log t,0\}\) for \(t>0\).
A centred circularly symmetric complex Gaussian variable of variance
\(\sigma^2\) has the law of \(\sigma G\), where
\(G\sim\mathcal N_{\mathbb C}(0,1)\).  The \(L^2\)-norm of the logarithm therefore satisfies
\begin{equation}\label{eq:local-Gaussian-log-bound}
\|\log|f_{\ell,N}^{\mathrm{loc}}(u)|\|_{L^2(\Omega)}
 \leqslant C\biggl(1+\log_+\frac1{|u-a|}\biggr),
 \qquad u\in B_4,
\end{equation}
when \(a\in B_5\); without such an \(a\), the left-hand side is
uniformly bounded.  The value at \(u=a\) is immaterial.

By Lemma~\ref{lem:conditional-intersection-facts}, we may work on the
probability-one event of proper intersections and set
\(T_j=[Z_{f_{1,N}^{\mathrm{loc}}}]\wedge\cdots\wedge
[Z_{f_{j,N}^{\mathrm{loc}}}]\), with \(T_0=1\).
Applying the Poincar\'e--Lelong formula to
\(f_{j,N}^{\mathrm{loc}}\) on the proper intersection cycle \(T_{j-1}\)
gives \(T_j\)
\cite[Chapter~III, Corollaries~4.10--4.11 and Proposition~4.12]
{DemaillyCADG}.
Replacing \(\log|f_{j,N}^{\mathrm{loc}}|\) by
\(\tfrac12\log(|f_{j,N}^{\mathrm{loc}}|^2+\varepsilon^2)\) gives local
\(L^1\) convergence as \(\varepsilon\downarrow0\) against the trace
measure of \(T_{j-1}\);
continuity of differentiation therefore gives the same current \(T_j\).

To integrate the logarithmic singularity against an intersection
current, we need a deterministic mass estimate.  Let \(T\) be a
positive closed current of bidimension \((d,d)\), with
\(d\geqslant1\), on a neighborhood of \(\overline B_R\), and write
\(\mu_T=T\wedge\beta^d\), where
\(\beta=(\sqrt{-1}/2)\sum du_a\wedge d\bar u_a\).  For fixed
\(0<t<R<4\), we claim, uniformly for \(a\in B_5\),
\begin{equation}\label{eq:Lelong-log-weight}
 \int_{B_t}\biggl(1+\log_+\frac1{|u-a|}\biggr)
       \,\mathrm d\mu_T(u)
 \leqslant C_{t,R,d}\,\mu_T(B_R).
\end{equation}
Indeed, set \(h=\min\{(R-t)/4,1/4\}\).  If
\(\operatorname{dist}(a,B_t)\geqslant h\), the weight is bounded.  In
the other case \(B(a,2h)\Subset B_R\).  By the mass monotonicity formula
\cite[Chapter~III, Section~5, (5.8)]{DemaillyCADG}, for every \(0<s<h\) we have
\[
 \mu_T(B(a,s))
 \leqslant\left(\frac{s}{2h}\right)^{2d}\mu_T(B_R).
\]
Consequently, by integrating the measures of the level sets,
\[
 \int_{B(a,h)}\log\frac h{|u-a|}\,\mathrm d\mu_T(u)
 =\int_0^h\frac{\mu_T(B(a,s))}{s}\,\mathrm ds
 \leqslant C_d\mu_T(B_R).
\]
On \(B(a,h)\), the weight is at most
\(1+\log(1/h)+\log(h/|u-a|)\); outside it, the weight is bounded
in terms of \(t,R\).  This proves~\eqref{eq:Lelong-log-weight}.

Given
\(0<r<t<R<4\), take
\(0\leqslant\chi\leqslant1\) in
\(\mathscr C_c^\infty(B_t)\), equal to one on \(B_r\), with
\(-C_{r,t}\beta\leqslant\sqrt{-1}\,\partial\bar\partial\chi
\leqslant C_{r,t}\beta\).
Poincar\'e--Lelong, closedness of \(T_{j-1}\), and Stokes, first for
\(\tfrac12\log(|f_{j,N}^{\mathrm{loc}}|^2+\varepsilon^2)\) and then by
the preceding local \(L^1\) convergence, give
\begin{equation*}
 M_j(r)
 \leqslant C_{r,t}
 \int_{B_t}|\log|f_{j,N}^{\mathrm{loc}}(u)||\,
       T_{j-1}\wedge\beta^{m-j+1},
 \qquad
 M_j(s):=\int_{B_s}T_j\wedge\beta^{m-j}.
\end{equation*}
Condition on \(T_{j-1}\) and first take the norm in the remaining
probability factor.  Independence of
\(f_{j,N}^{\mathrm{loc}}\) from this current allows us to use
Minkowski's inequality with~\eqref{eq:local-Gaussian-log-bound}.
The logarithmic mass estimate~\eqref{eq:Lelong-log-weight} then yields
\begin{equation}\label{eq:recursive-local-current-L2}
 \|M_j(r)\|_{L^2(\Omega_1\times\cdots\times\Omega_j)}
 \leqslant C_{r,R,m}
 \|M_{j-1}(R)\|_{L^2(\Omega_1\times\cdots\times\Omega_{j-1})}.
\end{equation}
If no point of \(\mathcal P\) has scaled coordinate in \(B_5\), the
uniform bound in~\eqref{eq:local-Gaussian-log-bound} suffices.
Otherwise \eqref{eq:Lelong-log-weight} applies to \(T_{j-1}\),
which has bidimension \((m-j+1,m-j+1)\), with \(m-j+1\geqslant1\)
even at the last step \(j=m\).

Finally choose \(1=\rho_0<\rho_1<\cdots<\rho_q=3\).  Apply
\eqref{eq:recursive-local-current-L2} successively for
\(j=q,q-1,\ldots,1\), with
\(r=\rho_{q-j}\) and \(R=\rho_{q-j+1}\).  This gives
\[
 \|M_q(1)\|_2
 \leqslant C\|M_{q-1}(\rho_1)\|_2
 \leqslant\cdots\leqslant C M_0(\rho_q)\leqslant C.
\]

The rescaled K\"ahler form satisfies
\[
 \delta_{N,x}^*\omega
 =N^{-1}\bigl(\beta+O(N^{-1/2})\bigr)
\]
on fixed balls.  Shrinking \(r_0\) once, the inverse image of
\(B(x,r_0N^{-1/2})\) lies in \(B_1\), uniformly in \(x\).  The
\(m-q\) factors of \(\omega\) in \(\mu_N^{(q)}\) therefore contribute
\(N^{q-m}\), proving
\eqref{eq:mixed-local-mass-L2}.  Uniform ball-volume bounds give a cover
of \(B(x,RN^{-1/2})\) by at most \(C(1+R)^{2m}\) balls of radius
\(r_0N^{-1/2}\), uniformly in \(x\) and \(R\geqslant1\).
For \(R\geqslant\operatorname{diam}(M)\sqrt N\), a cover of all of
\(M\) by \(O(N^m)\) such balls gives the same bound.  Minkowski's
inequality proves~\eqref{eq:mixed-growing-ball-mass-L2}.  Finally, when
\(q\leqslant m-1\), the current \(T_q\) has bidimension
\((m-q,m-q)\), with \(m-q\geqslant1\).  Repeating the preceding
iteration with radii between \(2\) and \(3\) gives
\(\|M_q(2)\|_2\leqslant C\); then
\eqref{eq:Lelong-log-weight}, with \(t=1\), \(R=2\), and \(a=0\),
controls the logarithmic weight.  Since
\(\sqrt N\,\operatorname{dist}(z,p)\asymp |u|\) in these coordinates, this proves
\eqref{eq:weighted-mixed-local-mass-L2} for \(q<m\).  When \(q=m\),
a prescribed point may be an atom of the intersection measure.
\end{proof}

\subsection{Comparison of the coupled statistics}
\label{sec:comparison-coupled}

For smooth tests, we use the logarithmic norm estimate in the global
Poincar\'e--Lelong telescoping identity.  For domain tests, we combine
a kernel comparison away from the prescribed points with local mass
bounds near them.

The remote estimate below is used in the proof of the domain bound in
Proposition~\ref{prop:finite-conditioning-stability}.  Dividing the
coupling bounds by the unconditioned standard deviation gives the
relative errors in Proposition~\ref{prop:quantitative-conditional-comparison}.
These errors also transfer the unconditioned variance asymptotics and
CLT to the conditional ensemble.

\begin{lem}
\label{lem:finite-conditioning-remote}
Assume \(\mathcal P\neq\varnothing\).  Fix an open domain \(U\subsetneq M\)
whose boundary is locally an embedded \(\mathscr C^2\) hypersurface with
\(U\) on one side, and let \(Q>0\) and \(\varkappa\geqslant0\).

There are \(b,C,N_0>0\), depending only on
\(Q,\varkappa,m,k,(M,L,h),\mathcal P\), and \(U\), with the following
property.  For every family \(f_N\in\mathscr C^3(M,\mathbb R)\) satisfying
\[
 \|f_N\|_{\mathscr C^3}\leqslant N^\varkappa,\qquad
 \operatorname{supp}(f_N|_{\overline U})\subset
 \{z:\operatorname{dist}_{\mathcal P}(z)\geqslant b\sqrt{\log N/N}\},
\]
set \(\varphi_{U,N}=f_N\chi_U\omega^{m-k}/(m-k)!\).  Under the orthogonal
coupling~\eqref{eq:orthogonal-conditional-coupling},
\begin{equation}\label{eq:finite-conditioning-remote-bound}
 \left\|X_{N,\mathcal P}^{\varphi_{U,N}}
             -X_N^{\varphi_{U,N}}\right\|_2
 \leqslant C N^{-Q},\qquad N\geqslant N_0.
\end{equation}
The same bound holds after centring.
\end{lem}

\begin{proof}
We estimate the two chaos tails and the difference at finite cutoff
separately.  The estimates determine the cutoff and the exclusion
radius needed to make each error \(O(N^{-Q})\) in \(L^2\).

The regions away from \(\mathcal P\) are
\[
 S_{N,b}:=\{\operatorname{dist}_{\mathcal P}\geqslant b\sqrt{\log N/N}\},\qquad
 S_N:=S_{N,b}\cap\overline U.
\]
Use the domain scales \(a_N=N^{k-1/2}\),
\(r_N=N^{-m+1/2}\), and the measure \(\mu_U\).

\paragraph{Finite-rank comparison}
For the two-point comparison, set
\begin{equation}\label{eq:normalized-finite-rank-correction}
 E_{N,\mathcal P}(z,w)
 :=\mathbf r_N(z)^*\mathbf G_N^{-1}\mathbf r_N(w).
\end{equation}
Under changes of fibre unit vectors, \(E_{N,\mathcal P}\) acquires
the same unit-modulus factor as \(R_N\); its modulus and diagonal
values are invariant.  Fix the endpoint frames of
Section~\ref{sec:Bergman-setting}, and let \(B_N^\perp\) and
\(B_N^{\mathcal P}\) be the local coefficients of the projection kernels
of \(\mathcal H_{N,\mathcal P}^\perp\) and \(\mathcal H_{N,\mathcal P}\),
respectively.
All kernel coefficients below use these same endpoint frames, with each
diagonal coefficient evaluated in its own endpoint frame.  The normalized
correction is therefore
\begin{equation*}
 E_{N,\mathcal P}(z,w)
 =\frac{B_N^\perp(z,w)}{\sqrt{B_N(z,z)B_N(w,w)}}.
\end{equation*}
The orthogonal decomposition gives
\begin{equation}\label{eq:conditional-two-point-Schur-complement}
 B_N^{\mathcal P}(z,w)=B_N(z,w)-B_N^\perp(z,w),
 \qquad
 E_{N,\mathcal P}(z,z)=1-\theta_{N,\mathcal P}(z).
\end{equation}

We prove the derivative estimates on the larger set \(S_{N,b/2}\),
leaving room for the diagonal Taylor argument below.  Given \(L>0\),
choose \(b=b(L)\) large enough that
Theorem~\ref{thm:expansion-szego-kernel}\textup{(iv)} applies to every
pair \((z,p)\), \(p\in\mathcal P\), with \(z\in S_{N,b/2}\).
Since \(\mathcal P\) is fixed and \(\mathbf G_N^{-1}\) is uniformly
bounded, that estimate and~\eqref{eq:normalized-finite-rank-correction} give
\begin{equation}\label{eq:remote-finite-rank-two-jet}
 \max_{|\alpha|,|\beta|\leqslant2}
 N^{-(|\alpha|+|\beta|)/2}
 \sup_{S_{N,b/2}^2}
 |D_z^\alpha D_w^\beta E_{N,\mathcal P}(z,w)|
 \leqslant C_LN^{-L}.
\end{equation}
Use the centred endpoint frames defining \(J_N\);
Lemma~\ref{lem:Bergman-horizontal-comparison} controls the frame
transitions uniformly.
In particular,
\(\theta_{N,\mathcal P}\geqslant1/2\) on \(S_{N,b/2}\) for large \(N\).

For \(z,w\in S_{N,b/2}\), let \(R_N^{aa'}(z,w)\),
\(a,a'\in\{\mathcal P,0\}\), be the correlation of the local
scalar representatives of \(s_N^a(z)\) and \(s_N^{a'}(w)\), each
divided by its own pointwise standard deviation, under the
coupling~\eqref{eq:orthogonal-conditional-coupling}, where \(s_N^0=s^N\).
For the mean currents on \(M\setminus\mathcal P\), use
\(\mathcal C_{0,0}^N=\mathcal C_0^N\) and the notation
\(\mathcal C_{0,\mathcal P}^N\) of
Lemma~\ref{lem:conditional-intersection-facts}.
The orthogonal decomposition of the covariance kernel gives
\begin{equation}\label{eq:four-coupled-correlations}
\begin{aligned}
 R_N^{00}&=R_N,&
 R_N^{\mathcal P\mathcal P}
  &=\frac{R_N-E_{N,\mathcal P}}
   {\sqrt{\theta_{N,\mathcal P}(z)
          \theta_{N,\mathcal P}(w)}},\\
 R_N^{\mathcal P0}
  &=\frac{R_N-E_{N,\mathcal P}}
   {\sqrt{\theta_{N,\mathcal P}(z)}},&
 R_N^{0\mathcal P}
  &=\frac{R_N-E_{N,\mathcal P}}
   {\sqrt{\theta_{N,\mathcal P}(w)}}.
\end{aligned}
\end{equation}
Equations~\eqref{eq:remote-finite-rank-two-jet},
\eqref{eq:four-coupled-correlations}, and the diagonal identity in
\eqref{eq:conditional-two-point-Schur-complement}, followed by the
quotient rule, imply, with \(\epsilon_N=N^{-L}\),
\begin{align}
 \max_{a,a'}\max_{|\alpha|,|\beta|\leqslant2}
 N^{-(|\alpha|+|\beta|)/2}
 \sup_{S_{N,b/2}^2}
 |D_z^\alpha D_w^\beta(R_N^{aa'}-R_N)|
 &\leqslant C_L\epsilon_N,
 \label{eq:remote-four-correlation-jets}\\
 \|\mathcal C_{0,\mathcal P}^N-\mathcal C_0^N\|_{\mathscr C^0(S_{N,b/2})}
 &\leqslant C_LN\epsilon_N.
 \label{eq:remote-mean-current-defect}
\end{align}
The conditional mean current in these estimates satisfies
\(\mathcal C_{0,\mathcal P}^N=\mathcal C_0^N+
(\sqrt{-1}/2\pi)\partial\bar\partial\log\theta_{N,\mathcal P}\).

\paragraph{Bounds for \(1-|R_N^{aa}|^2\) and weighted kernel integrals}
Set \(g_N^a=1-|R_N^{aa}|^2\), and let \(J_N^a\) be the maximum of
normalized derivatives in~\eqref{eq:kernel-jet-envelope}, with \(R_N\)
replaced by \(R_N^{aa}\).
Both \(g_N^a\) and their first derivatives vanish on the diagonal.
For \(z,w\in S_{N,b}\) with \(\operatorname{dist}(z,w)\leqslant cN^{-1/2}\), use the derivative estimate
\eqref{eq:remote-four-correlation-jets} on \(S_{N,b/2}\); for large
\(N\), the short geodesic segment from \((z,z)\) to \((z,w)\) remains
in \(S_{N,b/2}^2\).  Taylor's
formula,~\eqref{eq:remote-four-correlation-jets}, and
\eqref{eq:BF-correlation-gap} give
\(|g_N^{\mathcal P}-g_N^0|\leqslant
C\epsilon_NN\operatorname{dist}(z,w)^2\).  When \(\operatorname{dist}(z,w)>cN^{-1/2}\), the
zeroth-order estimate and the positive lower bound for \(g_N^0\) give
the comparison.  Thus, for \(L\) sufficiently large and then all
sufficiently large \(N\),
\begin{equation}\label{eq:remote-conditional-gap-row}
 J_N^{\mathcal P}
 \leqslant C(J_N^0+\epsilon_N)
 \quad\hbox{on }S_{N,b}^2,
 \qquad
 g_N^{\mathcal P}\geqslant\tfrac12g_N^0.
\end{equation}
For every \(0\leqslant p<m-\frac12\), we claim that
\begin{equation}\label{eq:remote-conditional-negative-row}
 \sup_{z\in S_N}
 \int_{S_N}\frac{J_N^a(z,w)}{(g_N^a(z,w))^p}
 \,\mathrm d\mu_U(w)\leqslant Cr_N,
 \qquad a\in\{\mathcal P,0\}.
\end{equation}
For the unconditioned ensemble, use
Lemma~\ref{lem:cv-negative-moments} with \(\varrho_N=r_N\).
Equation~\eqref{eq:BF-correlation-volume} also gives
\[
 \sup_{z\in S_N}
 \int_{S_N}(g_N^0(z,w))^{-p}\,\mathrm d\mu_U(w)
 \leqslant C_p(1+r_N).
\]
The regions \(g_N^0\leqslant t_0\) and \(g_N^0>t_0\)
contribute \(O(r_N)\) and \(O(1)\), respectively.
Thus the additional \(\epsilon_N\)-term
in~\eqref{eq:remote-conditional-gap-row} is \(O(r_N)\) if
\(\epsilon_N\leqslant r_N\); for \(p=0\), finite mass suffices.

\paragraph{Local projection identities and uniform remainder estimates}
Let \(Y_N^a=\langle[Z_{s_{1,N}^a,\ldots,s_{k,N}^a}],\varphi_{U,N}\rangle\),
and let \(Y_{N,n}^a\) be the statistic obtained by truncating the
Wiener-chaos expansion in each section at the same index \(n\).
Set \(K_N:=\operatorname{supp}(f_N|_{\overline U})\subset S_N\).
For the conditional projection, choose an open set \(V_N\) such that
\[
 K_N\subset S_N\Subset V_N\Subset M\setminus\mathcal P,
 \qquad V_N\subset S_{N,b/2}.
\]
On \(V_N\), \(\theta_{N,\mathcal P}\geqslant1/2\) gives
\(\mathsf b_N^{\mathcal P}>0\).
Lemma~\ref{lem:conditional-intersection-facts} supplies properness and
zero boundary mass on \(K_N\) for every conditional subcollection;
Lemma~\ref{lem:unconditioned-intersection-facts} supplies the same facts
for the unconditioned ensemble.
Apply Lemma~\ref{lem:localized-projection-Abel} with \(K=K_N\) and
an auxiliary cutoff equal to one near \(K_N\).  It gives the \(L^2\)
projection identities for arbitrary truncation indices in the individual
sections, together with the fixed-\(N\) Abel limit.
To make the tail bound uniform, use the weighted integral estimate
\eqref{eq:remote-conditional-negative-row}.
The integrations by parts for domain tests take place over \(U\) and
\(\partial U\); their resulting integrands remain supported in \(S_N\).

Lemma~\ref{lem:weighted-mean-current-bound} and
\eqref{eq:remote-mean-current-defect} give
\(\|\mathcal C_{0,a}^N\|_{\mathscr C^0(S_{N,b})}\leqslant CN\);
\textup{(OD)} and~\eqref{eq:remote-four-correlation-jets} give
\(J_N^a\leqslant C\).
Together with~\eqref{eq:remote-conditional-negative-row} and the
localized projection and Abel identities, these verify
Lemma~\ref{lem:automatic-kernel-tail} on the closed set \(S_N\) for
both ensembles, with \(\varrho_N=r_N\) and \(\beta=1/4\).
This choice is valid since \(k-1+1/4<m-1/2\) for \(k\leqslant m\).
Applying it and
\eqref{eq:common-cutoff-error} in the independent sections gives
\begin{equation}\label{eq:remote-two-tail-estimate}
 \|Y_N^a-Y_{N,n}^a\|_2^2
 \leqslant C N^{2\varkappa}a_N^2r_N n^{-1/4},
 \qquad a\in\{\mathcal P,0\}.
\end{equation}

\paragraph{Comparison at a finite chaos cutoff}
For \(a,a'\in\{\mathcal P,0\}\), set
\[
 Q_{N,n}^{aa'}=\frac1{4\pi^2}\sum_{\alpha=1}^n
       \frac{|R_N^{aa'}|^{2\alpha}}{\alpha^2},
\]
\[
 K_{N,n}^{aa'}=\pi_1^*\mathcal C_{0,a}^N
  \wedge\pi_2^*\mathcal C_{0,a'}^N+\mathscr D_{12}Q_{N,n}^{aa'}.
\]
The two-point Wick identity, applied to the coupled section values
divided by their respective standard deviations, gives
\[
 \mathbb E\bigl[[Z_{s_N^a}]^{\leqslant n}\boxtimes
                 [Z_{s_N^{a'}}]^{\leqslant n}\bigr]=K_{N,n}^{aa'}.
\]
Independence for different section indices consequently yields the
exact identity
\begin{equation}\label{eq:remote-four-kernel-identity}
\begin{aligned}
 \mathbb E|Y_{N,n}^{\mathcal P}-Y_{N,n}^0|^2
 =\int_{S_N^2}\bigl[&(K_{N,n}^{\mathcal P\mathcal P})^{\wedge k}
 +(K_{N,n}^{00})^{\wedge k}\\
 &-(K_{N,n}^{\mathcal P0})^{\wedge k}
 -(K_{N,n}^{0\mathcal P})^{\wedge k}\bigr]
 \wedge\pi_1^*\varphi_{U,N}\wedge\pi_2^*\varphi_{U,N}.
\end{aligned}
\end{equation}
For every fixed \(L,b\), the finite sum defining \(Q_{N,n}^{aa'}\), the
chain rule, and~\eqref{eq:remote-four-correlation-jets}--
\eqref{eq:remote-mean-current-defect} give, uniformly for \(n\geqslant1\)
and all sufficiently large \(N\),
\begin{equation*}
 \max_{a,a'}\|K_{N,n}^{aa'}-K_{N,n}^{00}\|_{\mathscr C^0(S_{N,b}^2)}
 \leqslant C_{L,b}N^2(1+n)^4N^{-L},
 \qquad
 \max_{a,a'}\|K_{N,n}^{aa'}\|_{\mathscr C^0(S_{N,b}^2)}
 \leqslant C_{L,b}N^2(1+n)^3.
\end{equation*}
Indeed, four derivatives of \(|R_N^{aa'}|^{2\alpha}/\alpha^2\)
give coefficients bounded by \(C(1+\alpha)^2\), whose sum is
\(O((1+n)^3)\); the difference adds one factor \(1+\alpha\).
The constants are independent of \(N,n\), and these powers of \(n\)
are independent of \(L,b\).
Expanding each wedge power in~\eqref{eq:remote-four-kernel-identity}
around \(K_{N,n}^{00}\), and using the coefficient bounds for
\(\pi_1^*\varphi_{U,N}\) and \(\pi_2^*\varphi_{U,N}\), therefore gives
\begin{equation*}
 \|Y_{N,n}^{\mathcal P}-Y_{N,n}^0\|_2^2
 \leqslant C_{L,b}N^{2k+2\varkappa}
 (1+n)^{3k+1}N^{-L}.
\end{equation*}

With \(Q,\varkappa\) fixed, choose the remaining parameters
in the order \(H\), \(L\), \(b\), \(N_0\).
First choose \(H>0\) so that the product of
\(N^{2\varkappa}a_N^2r_N\) and \(N^{-H/4}\) is
\(O(N^{-2Q})\) in~\eqref{eq:remote-two-tail-estimate}, and set
\(n=\lceil N^H\rceil\).  To absorb the resulting polynomial factor
in the finite-cutoff estimate, choose
\[
 L>\max\{m-\tfrac12,\,
 2Q+2k+2\varkappa+H(3k+1)\},
\]
and finally choose \(b=b(L)\) in
\eqref{eq:remote-finite-rank-two-jet}, followed by \(N_0\) large enough
for all preceding estimates.  The triangle inequality proves
\eqref{eq:finite-conditioning-remote-bound}, also after centring.
These choices are independent of the Abel parameter.
\end{proof}

The lemma controls the part of a domain test supported away from the
prescribed points.  We now prove the full coupling bounds for both test
classes.

\begin{proof}[Proof of Proposition~\ref{prop:finite-conditioning-stability}]
The case \(\mathcal P=\varnothing\) is immediate.
For the fixed orthogonal coupling with \(\mathcal P\neq\varnothing\), write
\[
 U_{j,N}=\log|s_{j,N}^{\mathcal P}|_{h_N}
          -\log|s_j^N|_{h_N},
\]
and use the mixed intersection currents
\[
 \mathfrak I_{j,N}
 :=\left(\bigwedge_{\ell=1}^{j-1}[Z_{s_{\ell,N}^{\mathcal P}}]\right)
   \wedge\left(\bigwedge_{\ell=j+1}^{k}[Z_{s_\ell^N}]\right).
\]
Here \(U_{j,N}\) is independent of \(\mathfrak I_{j,N}\).
Telescoping the wedge products and applying Poincar\'e--Lelong gives
\begin{equation}\label{eq:conditional-current-telescope}
 [Z_{s_{1,N}^{\mathcal P},\ldots,s_{k,N}^{\mathcal P}}]
 -[Z_{s_1^N,\ldots,s_k^N}]
 =\frac{\sqrt{-1}}\pi\sum_{j=1}^k
   \partial\bar\partial U_{j,N}\wedge\mathfrak I_{j,N}.
\end{equation}

For every real \(\mathscr C^3\) form \(\psi\) of bidegree
\((m-k,m-k)\), Stokes on the compact manifold gives
\begin{equation}\label{eq:conditional-smooth-telescope}
 X_{N,\mathcal P}^{\psi}-X_N^\psi
 =\frac{\sqrt{-1}}\pi\sum_{j=1}^k
   \int_M U_{j,N}\,\mathfrak I_{j,N}
                    \wedge\partial\bar\partial\psi.
\end{equation}
The identity follows first for regularized logarithms and then in
\(L^1\).  Properness makes the divisor logarithms locally integrable
on \(\mathfrak I_{j,N}\); since this cycle has complex dimension
\(m-k+1\geqslant1\), \eqref{eq:weighted-mixed-local-mass-L2} also
controls the prescribed-point logarithmic weight.

Suppose first that \(\varphi\) is of type~\textup{(S)}.  Positivity yields
\[
 |\mathfrak I_{j,N}\wedge\partial\bar\partial\varphi|
 \leqslant C_\varphi\,
 \mathfrak I_{j,N}\wedge\omega^{m-k+1}.
\]
Set \(\nu_{j,N}:=\mathfrak I_{j,N}\wedge\omega^{m-k+1}\), and let
\(\Omega_{\ne j}\) denote the randomness in the section pairs whose
indices differ from \(j\).  Conditional on those pairs, \(\nu_{j,N}\)
is fixed and \(U_{j,N}\) is independent of it.  Minkowski's inequality
and Lemma~\ref{lem:finite-defect-log-profile} give
\[
 \left\|\int_M|U_{j,N}|\,\mathrm d\nu_{j,N}\right\|_{L^2(\Omega)}
 \leqslant
 \left\|\int_M\Lambda_{N,\mathcal P}\,\mathrm d\nu_{j,N}
       \right\|_{L^2(\Omega_{\ne j})}.
\]
By positivity, this also bounds the \(L^2\)-norm of the \(j\)-th integral
in~\eqref{eq:conditional-smooth-telescope}, up to \(C_\varphi\).

We integrate the two branches of the logarithmic profile separately.
Let \(r_0\) be as in Lemma~\ref{lem:mixed-local-intersection-mass} and
first isolate
\[
 \mathcal N_N
 :=\bigcup_{p\in\mathcal P}B\left(p,\frac{r_0}{2\sqrt N}\right).
\]
Shrink \(r_0\) so that, for large \(N\), the enlarged balls
\(B(p,r_0N^{-1/2})\) are disjoint and
\(\operatorname{dist}_{\mathcal P}=\operatorname{dist}(\,\cdot\,,p)\) on each of them.  Their contribution
is controlled by~\eqref{eq:weighted-mixed-local-mass-L2}.  Cover
\(M\setminus\mathcal N_N\) by balls of radius \(r_0/(8\sqrt N)\), with
centres in this complement and bounded overlap; a maximal
\(r_0/(16\sqrt N)\)-separated net gives such a cover.  For every centre
\(x\) of this net,
\[
 \inf_{z\in B(x,r_0/(8\sqrt N))}
 \sqrt N\,\operatorname{dist}_{\mathcal P}(z)\geqslant\frac{3r_0}{8}.
\]
The function \(\Lambda_{N,\mathcal P}\) is uniformly bounded on the
balls whose centres satisfy \(\sqrt N\,\operatorname{dist}_{\mathcal P}<2\).
For \(\ell\geqslant1\), group the remaining centres into dyadic annuli
\(2^\ell\leqslant\sqrt N\,\operatorname{dist}_{\mathcal P}<2^{\ell+1}\).
The disjoint \(r_0/(32\sqrt N)\)-balls around these centres lie in
\(\bigcup_{p\in\mathcal P}B(p,C2^\ell/\sqrt N)\).
The uniform ball-volume bounds therefore give at most \(C2^{2m\ell}\)
centres in this group.  On each covering ball \(\Lambda_{N,\mathcal P}\)
is bounded
by \(C2^{-A\ell}\sqrt{1+\ell}\), by~\eqref{eq:log-defect-profile}.
Minkowski's inequality and~\eqref{eq:mixed-local-mass-L2}, with
\(q=k-1\), bound the total contribution by
\[
 C N^{k-m-1}
 \left[1+\sum_{\ell\geqslant1}2^{(2m-A)\ell}\sqrt{1+\ell}\right]
 \leqslant C N^{k-m-1},
\]
because \(A>2m\).  Equation~\eqref{eq:conditional-smooth-telescope}
therefore proves the first line of
\eqref{eq:finite-conditioning-stability} for the full difference.

For a domain test \(\varphi_U\) of type~\textup{(D)}, fix
\(Q>m\) and \(\varkappa>3/2\), and take \(b\) from
Lemma~\ref{lem:finite-conditioning-remote} for these exponents.  Choose
\(\vartheta\in\mathscr C_c^\infty([0,\infty))\), with
\(0\leqslant\vartheta\leqslant1\), equal to one on \([0,1]\) and zero
on \([2,\infty)\), and set
\[
 L_N=b\sqrt{\log N},\qquad
 \chi_N(z)=\sum_{p\in\mathcal P}
 \vartheta\!\left(\frac{\sqrt N\,\operatorname{dist}(z,p)}{L_N}\right).
\]
For large \(N\), the supports are disjoint and within the injectivity
radius.  Each summand is constant near its centre, so \(\chi_N\) is
smooth, \(0\leqslant\chi_N\leqslant1\), and its first three derivatives
are \(O_b(N^{3/2})\).  With \(\varkappa\) and \(b\) fixed, increasing
\(N_0\) gives
\(\|(1-\chi_N)f\|_{\mathscr C^3}\leqslant N^\varkappa\).
Pair the left side of~\eqref{eq:conditional-current-telescope} first with
\(\chi_N\varphi_U\).  Positivity and
\eqref{eq:mixed-growing-ball-mass-L2}, with \(q=k\) and \(R=2L_N\), give
\begin{equation}\label{eq:conditional-numerical-near}
 \bigl\|\langle
 [Z_{s_{1,N}^{\mathcal P},\ldots,s_{k,N}^{\mathcal P}}]
 -[Z_{s_1^N,\ldots,s_k^N}],\chi_N\varphi_U\rangle\bigr\|_2
 \leqslant C_{\varphi_U,\mathcal P}N^{k-m}(\log N)^m.
\end{equation}
Lemma~\ref{lem:finite-conditioning-remote} bounds the pairing
with \(\varphi_{U,N}=(1-\chi_N)\varphi_U\) by \(O(N^{-Q})\), smaller than
\eqref{eq:conditional-numerical-near} because \(Q>m\).
The triangle inequality proves the second line of
\eqref{eq:finite-conditioning-stability} for the full difference.
Both bounds also hold after centring.

For \(k=m\), Lemma~\ref{lem:conditional-intersection-facts} makes the
prescribed zeros simple.  In the pairing with \(\chi_N\varphi_U\),
the prescribed zeros in \(U\) contribute with weights \(f(p)\)
because \(\chi_N(p)=1\) for \(p\in\mathcal P\).  Those on \(\partial U\)
have weight zero by the convention \(\chi_U=0\) there.  The smooth argument uses only \(m-1\) divisors,
so \(\mathfrak I_{j,N}\) has bidimension \((1,1)\) and the logarithm
is integrable.
\end{proof}

Dividing the coupling bounds by the unconditioned standard deviation
gives the relative errors used to transfer the variance asymptotics
and the central limit theorem.

\begin{proof}[Proof of Main Theorems~\ref{thm:finite-point-conditioned-variance}
and~\ref{thm:finite-point-conditioned-clt}, and
Proposition~\ref{prop:quantitative-conditional-comparison}]
Theorem~\ref{thm:two-estimate-BF-criterion}, with the kernel estimates
verified in Section~\ref{sec:gaussian-sections-verification}, gives the
unconditioned variance bounds and CLT; for smooth \(k=1\), these follow
from~\cite[Theorems~1.1--1.2]{MR2742043}.
In particular, the standard deviation satisfies
\[
 \sigma_N:=\sqrt{\operatorname{Var}X_N^\varphi}\asymp
 \begin{cases}
  N^{k-1-m/2},&\textup{(S)},\\
  N^{k-m/2-1/4},&\textup{(D)}.
 \end{cases}
\]
Dividing~\eqref{eq:finite-conditioning-stability} by these standard deviations
gives \(O(N^{-m/2})\) and
\(O(N^{-m/2+1/4}(\log N)^m)\), respectively, both tending to zero.
The mean discrepancy satisfies
\[
 \left|\mathbb E_{\mathcal P}X_{N,\mathcal P}^{\varphi}
       -\mathbb EX_N^\varphi\right|
 \leqslant\|X_{N,\mathcal P}^{\varphi}-X_N^\varphi\|_2
 =o(\sigma_N).
\]
The same relative bounds hold after centring.  If
\(\sigma_{N,\mathcal P}^2=
\operatorname{Var}_{\mathcal P}X_{N,\mathcal P}^\varphi\), the reverse
triangle inequality in \(L^2\) bounds the relative error
\(\left|\sigma_{N,\mathcal P}/\sigma_N-1\right|\) by these rates.
Squaring \(\sigma_{N,\mathcal P}/\sigma_N\) proves the rates in
\eqref{eq:conditional-unconditional-variance-ratio}.
For a smooth test with \(k\geqslant2\),
Proposition~\ref{prop:variance-lower-bound} with all cutoffs infinite gives
\(\operatorname{Var}X_N^\varphi\geqslant\mathcal W_{l_*}^N(\varphi)\).
Dividing by \(N^{2k-m-2}\) and passing to the limit using
Theorem~\ref{thm:SZ-variance-asymptotics}\textup{(i)} with
\(0<\varepsilon<1/2\) and~\eqref{eq:intrinsic-smooth-component-limit}
gives the lower bound~\eqref{eq:smooth-variance-coefficient-lower-bound}.
For smooth tests with \(k=1\), the known variance lower bound at the same
scale gives strict positivity by
Theorem~\ref{thm:SZ-variance-asymptotics}\textup{(i)}.
For domain tests,
Proposition~\ref{prop:weighted-domain-variance-asymptotics} gives the
unconditioned leading coefficient
\(\nu_{m,k}\int_{\partial U}f^2\,\mathrm d\sigma_{\partial U}>0\).
Multiplying these unconditioned asymptotics by the variance ratio proves
both conditional variance formulas in
Main Theorem~\ref{thm:finite-point-conditioned-variance}.
The centred coupling bound divided by \(\sigma_N\) and
\(\sigma_{N,\mathcal P}/\sigma_N\to1\) imply that the two standardized
statistics differ by \(o(1)\) in \(L^2\) under the coupling.
The unconditioned central limit theorem and Slutsky's theorem
give Main Theorem~\ref{thm:finite-point-conditioned-clt}.  The mean
estimate also shows
that the same Gaussian limit holds with the unconditioned mean
\(\mathbb EX_N^\varphi\) and either denominator \(\sigma_N\) or
\(\sigma_{N,\mathcal P}\), including the immediate case
\(\mathcal P=\varnothing\).
\end{proof}

\subsection{Leading correction to the conditional mean}
\label{sec:conditional-mean-asymptotics}

We identify the term of order \(N^{k-m-1}\) in the conditional mean
shift for smooth tests.  For weighted-domain tests with
\(\mathcal P\cap\partial U=\varnothing\), only the prescribed points in
\(U\) enter this term.

For a single divisor, Shiffman, Zelditch and Zhong
\cite[Theorem~1.1]{MR2806467} computed the leading correction to the
smooth mean.  The telescoping identity
\eqref{eq:conditional-smooth-telescope} gives the following expansion
for common-zero statistics.

\begin{pro}
\label{prop:conditional-mean-asymptotics}
For \(1\leqslant k\leqslant m\), let
\[
 C_{m,k}:=\frac{\pi^{m-k}}2
 \sum_{n=1}^{\infty}\frac{\binom{n+k-1}{k-1}}{n^{m+1}}>0.
\]
For every fixed finite set \(\mathcal P\subset M\) of distinct points,
possibly empty and independent of \(N\), and every fixed
\(\varepsilon>0\), the following hold as \(N\to\infty\).
\begin{enumerate}[label=\textup{(\roman*)}]
\item For every fixed real \(\mathscr C^3\) form \(\varphi\) of bidegree
\((m-k,m-k)\),
\begin{equation}\label{eq:smooth-conditional-mean-asymptotics}
\begin{aligned}
 \mathbb E_{\mathcal P}X_{N,\mathcal P}^{\varphi}-\mathbb EX_N^\varphi
 &=-C_{m,k}N^{k-m-1}\sum_{p\in\mathcal P}
 \left(\frac{\omega^{k-1}\wedge
          \sqrt{-1}\,\partial\bar\partial\varphi}{\omega^m/m!}\right)(p)\\
 &\qquad+O_{\varphi,\mathcal P,\varepsilon}(N^{k-m-\frac32+\varepsilon}).
\end{aligned}
\end{equation}
\item For a fixed domain test \(\varphi_U\) of type~\textup{(D)} with
\(\mathcal P\cap\partial U=\varnothing\),
\begin{equation}\label{eq:domain-conditional-mean-away}
\begin{aligned}
 \mathbb E_{\mathcal P}X_{N,\mathcal P}^{\varphi_U}
       -\mathbb EX_N^{\varphi_U}
 &=-\frac{C_{m,k}}{(m-k)!}N^{k-m-1}
 \sum_{p\in\mathcal P\cap U}
 \left(\frac{\omega^{m-1}\wedge
          \sqrt{-1}\,\partial\bar\partial f}{\omega^m/m!}\right)(p)\\
 &\qquad+O_{f,U,\mathcal P,\varepsilon}(N^{k-m-\frac32+\varepsilon}).
\end{aligned}
\end{equation}
\end{enumerate}
The fractions are pointwise ratios of top-degree forms.  In particular,
\(C_{m,1}=\pi^{m-1}\zeta(m+1)/2\), where \(\zeta\) is the Riemann
zeta function.  For smooth forms, this is the coefficient in the cited
theorem.  For \(m=k=2\), the series gives
\(C_{2,2}=(\zeta(2)+\zeta(3))/2\).
\end{pro}

\begin{proof}[Proof of Proposition~\ref{prop:conditional-mean-asymptotics}]
\textup{(i)} The case \(\mathcal P=\varnothing\) is immediate.  On
\(M\setminus\mathcal P\), the mixed intersection currents satisfy
\[
 \mathbb E\mathfrak I_{j,N}
 = (\mathcal C_{0,\mathcal P}^N)^{\wedge(j-1)}
   \wedge(\mathcal C_0^N)^{\wedge(k-j)},\qquad 1\leqslant j\leqslant k.
\]
For compactly supported tests in this open set, the local
Poincar\'e--Lelong and successive-conditioning argument of
Lemma~\ref{lem:conditional-intersection-facts} applies to each independent
conditional or unconditioned section.  Each step contributes its mean
divisor form, proving the displayed identity.
Also, \(U_{j,N}\) is independent of \(\mathfrak I_{j,N}\), and
\(\mathbb E U_{j,N}=\tfrac12\log\theta_{N,\mathcal P}\).
Lemma~\ref{lem:mixed-local-intersection-mass} justifies Fubini.
Taking expectations in~\eqref{eq:conditional-smooth-telescope}, with
\(\psi=\varphi\), therefore gives
\begin{equation}\label{eq:conditional-mean-telescope}
\begin{aligned}
 &\mathbb E_{\mathcal P}X_{N,\mathcal P}^{\varphi}-\mathbb EX_N^\varphi\\
 &\quad=\frac1{2\pi}\sum_{j=0}^{k-1}
 \int_{M\setminus\mathcal P}\log\theta_{N,\mathcal P}
 (\mathcal C_{0,\mathcal P}^N)^{\wedge j}
 \wedge(\mathcal C_0^N)^{\wedge(k-1-j)}
 \wedge\sqrt{-1}\,\partial\bar\partial\varphi.
\end{aligned}
\end{equation}
The mixed currents on the right have codimension \(k-1<m\), so their
trace measures have no point mass at the prescribed points.  Their
logarithmic pairings are integrable by
\eqref{eq:weighted-mixed-local-mass-L2}.  This identity includes the
prescribed atoms on the left when \(k=m\).

Fix \(p\in\mathcal P\) and use the scaling map \(\delta_{N,p}\).
In the scaled coordinates \(u\), set
\[
 s=|u|^2,\qquad h(s)=\log(1-e^{-s}),\qquad
 a(s)=(1-e^{-s})^{-1},\qquad
 \beta=\frac{\sqrt{-1}}2\sum_{\ell=1}^m
              \mathrm du_\ell\wedge\mathrm d\bar u_\ell.
\]
The limiting unconditioned and conditional mean divisor forms are
\[
 A_0=\frac\beta\pi,\qquad
 A_1=\frac{\sqrt{-1}}{2\pi}\partial\bar\partial(s+h(s)).
\]
The kernel expansion of
Theorem~\ref{thm:expansion-szego-kernel}\textup{(i)} and the
Schur complement~\eqref{eq:conditional-Schur-complement} give
quantitative convergence to these forms.  More precisely, choose
\(0<\varepsilon'<\min\{\varepsilon,1/2\}\).  On \(0<|u|\leqslant1\),
\[
\begin{aligned}
 \log\theta_{N,\mathcal P}(\delta_{N,p}(u))-h(|u|^2)
     &=O(N^{-1/2+\varepsilon'}),\\
 \delta_{N,p}^*\mathcal C_{0,\mathcal P}^N-A_1
     &=O(N^{-1/2+\varepsilon'}|u|^{-2}),\qquad
 \delta_{N,p}^*\mathcal C_0^N-A_0=O(N^{-1/2}).
\end{aligned}
\]
Here and below the bounds for forms are coefficient bounds.  Indeed,
\(\theta_{N,\mathcal P}\circ\delta_{N,p}-(1-e^{-|u|^2})\) and its
first derivatives vanish at zero, while its second derivatives are
\(O(N^{-1/2+\varepsilon'})\).  Taylor's formula and
\eqref{eq:conditional-diagonal-quadratic} then give these logarithmic
and second-derivative estimates, as in
\cite[Section~5, equations~(5.4)--(5.6)]{MR2806467}.

Freezing the coefficients of
\(N^{m-k+1}\delta_{N,p}^*(\sqrt{-1}\,\partial\bar\partial\varphi)\)
at zero costs \(O(N^{-1/2}|u|)\), since \(\varphi\) is \(\mathscr C^3\).
The radial majorant for the error in a summand of
\eqref{eq:conditional-mean-telescope} on \(|u|\leqslant1\), after
factoring out \(N^{k-m-1}\), is
\[
 C N^{-1/2+\varepsilon'}(1+|\log r|)
 r^{2m-1-2(k-1)}\,\mathrm dr,\qquad r=|u|.
\]
It is integrable for every \(k\leqslant m\).
For the remaining scaled region, let \(\eta_p\) be the constant-coefficient
form obtained by freezing \(\sqrt{-1}\,\partial\bar\partial\varphi\)
at \(p\) in these coordinates, and write
\[
\begin{aligned}
 \mathcal I_{N,j,p}
 :={}&N^{m-k+1}\delta_{N,p}^*\!\bigl[
 \log\theta_{N,\mathcal P}
 (\mathcal C_{0,\mathcal P}^N)^{\wedge j}\\
 &\qquad\wedge(\mathcal C_0^N)^{\wedge(k-1-j)}
 \wedge\sqrt{-1}\,\partial\bar\partial\varphi\bigr],\\
 \mathcal I_{j,p}^{\mathrm{BF}}
 :={}&h(|u|^2)A_1^{\wedge j}\wedge A_0^{\wedge(k-1-j)}\wedge\eta_p.
\end{aligned}
\]
For fixed \(b,Q>0\), on \(1\leqslant|u|\leqslant b\sqrt{\log N}\),
the diagonal ratio is
bounded away from zero.  The differentiated kernel expansion and the
Schur complement therefore give the coefficient bound
\[
 |\mathcal I_{N,j,p}(u)-\mathcal I_{j,p}^{\mathrm{BF}}(u)|
 \leqslant C N^{-1/2+\varepsilon'}(1+|u|)^D e^{-c|u|^2}
       +C_QN^{-Q},\qquad 0\leqslant j\leqslant k-1,
\]
for fixed \(c>0\) and \(D\geqslant0\), uniformly in \(p\) and \(j\).
The Gaussian factor comes from the covariance with the prescribed
point.  Freezing the coefficients contributes only another polynomial factor.
The other prescribed points contribute \(O(N^{-Q})\), for any fixed
\(Q>0\), by off-diagonal decay.  The scaled region has volume
\(O((\log N)^m)\), so choosing \(Q\) large absorbs its contribution
into the same error after restoring the factor \(N^{k-m-1}\).
Outside these balls, off-diagonal decay makes the original contribution
\(O(N^{-Q})\) for any prescribed \(Q\), on choosing \(b\) sufficiently
large.  For \(|u|\geqslant1\), the coefficients of the frozen integrands
are bounded by \(Ce^{-|u|^2}\) because \(h(s)=O(e^{-s})\) as
\(s\to\infty\) and \(A_1\) is bounded there.  Their integrals over
\(|u|>b\sqrt{\log N}\) in the scaled variables are \(O(N^{-Q})\) after
increasing \(b\).  Thus the contributions of the distinct points add,
with each frozen integral taken over \(\mathbb C^m\).
All errors are \(O(N^{k-m-3/2+\varepsilon})\).

The form \(A_1\) has tangential eigenvalue \(a(s)/\pi\) and radial
eigenvalue \((a(s)+sa'(s))/\pi\), relative to \(\beta\).
On \(|u|^2=s>0\), the rank-one projector onto the complex radial line
has spherical average \(I/m\).  Its associated \((1,1)\)-form has
zero exterior square, so at most one radial correction contributes
to \(A_1^{\wedge j}\).  Consequently, the unitary average of
\(A_1^{\wedge j}\wedge A_0^{\wedge(k-1-j)}\) is
\[
 \frac{\beta^{k-1}}{\pi^{k-1}}
 \left[a(s)^j+\frac{s}{m}(a(s)^j)'\right].
\]
Since \(h'=a-1\), radial integration by parts gives
\[
 \int_0^\infty h(s)s^{m-1}
 \left[a(s)^j+\frac{s}{m}(a(s)^j)'\right]\,\mathrm ds
 =-\frac1m\int_0^\infty s^m(a(s)-1)a(s)^j\,\mathrm ds.
\]
The boundary terms vanish because \(j\leqslant k-1<m\).
Summing in \(j\), using
\(\sum_{j=0}^{k-1}(a-1)a^j=a^k-1\), and including the radial volume
factor \(\pi^m/(m-1)!\) in \eqref{eq:conditional-mean-telescope}, yields
\[
 C_{m,k}=\frac{\pi^{m-k}}{2m!}
 \int_0^\infty s^m\bigl[(1-e^{-s})^{-k}-1\bigr]\,\mathrm ds.
\]
Finally, expand
\((1-e^{-s})^{-k}-1=\sum_{n\geqslant1}\binom{n+k-1}{k-1}e^{-ns}\).
Termwise integration gives the stated series, which converges because
\(k\leqslant m\).  The frozen top-degree coefficient is precisely
\((\omega^{k-1}\wedge\sqrt{-1}\,\partial\bar\partial\varphi)/(\omega^m/m!)\),
proving~\eqref{eq:smooth-conditional-mean-asymptotics}.

\textup{(ii)} Choose \(\chi\in\mathscr C_c^\infty(U)\) equal to one near
\(\mathcal P\cap U\), and apply part~\textup{(i)} to the smooth form
\(\chi f\omega^{m-k}/(m-k)!\).
The difference from \(\varphi_U\) is supported a positive distance from
\(\mathcal P\).  On this support,
\eqref{eq:remote-mean-current-defect} and the mean-current product
identity give a conditional mean correction \(O(N^{-Q})\) for every
\(Q>0\).  At each point of \(\mathcal P\cap U\), the smooth test equals
\(f\omega^{m-k}/(m-k)!\) on a neighborhood, while it vanishes near the
other prescribed points.  Substitution into
\eqref{eq:smooth-conditional-mean-asymptotics} gives
\eqref{eq:domain-conditional-mean-away}.
\end{proof}

\appendix
\section{Complementary-index identities for the Bargmann--Fock model}
\label{app:BF-complementary-indices}

We verify the trace and partial trace identities used in
Lemma~\ref{lem:co-primitive-BF-contraction}.  Fix a normal-coordinate
centre \(z_0\) and \(1\leqslant l<k\), and set
\(t=k-l\), \(d_l=m-t\).  The real constant-coefficient
\((d_l,d_l)\)-forms
\(\alpha_{l,z_0}^{(z)}\) and \(\alpha_{l,z_0}^{(v)}\) are given in
\eqref{eq:frozen-smooth-coupling-forms}.  For a \(t\)-element subset
\(D\subset\{1,\ldots,m\}\),
write \(D^c=\{1,\ldots,m\}\setminus D\) and let \(\sigma(D)\) satisfy
\(\mathrm dz_{D^c}\wedge\mathrm dz_D
=\sigma(D)\mathrm dz_1\wedge\cdots\wedge\mathrm dz_m\).
Recall that
\[
 B_{D,C}=(\sqrt{-1})^{d_l}\sigma(D)\sigma(C)
 \psi^l_{D^c,C^c}(z_0),\qquad B=B_l(z_0),
\]
on the standard basis of \(\bigwedge^t\mathbb C^m\), and that
\(M_t(v)=\bigwedge^t(I-vv^*)\).  Exterior multiplication and contraction
are as in Lemma~\ref{lem:exterior-Gaussian-contraction}.  We use the
following differential forms:
\[
\begin{aligned}
\mathsf E&=\sum_a\mathrm d\bar z_a\wedge\mathrm dv_a,
&\mathsf F&=\sum_a\mathrm dz_a\wedge\mathrm d\bar v_a,\\
\mathsf Z&=\sum_a\bar v_a\,\mathrm dz_a,
&\overline{\mathsf Z}&=\sum_av_a\,\mathrm d\bar z_a,\\
\mathsf V&=\sum_a\bar v_a\,\mathrm dv_a,
&\overline{\mathsf V}&=\sum_av_a\,\mathrm d\bar v_a.
\end{aligned}
\]
The volume and K\"ahler conventions are
\[
\begin{aligned}
 \Omega_0&=\prod_{a=1}^m\sqrt{-1}\,\mathrm dz_a\wedge\mathrm d\bar z_a,
 &\omega_0&=\frac{\sqrt{-1}}2
 \sum_a\mathrm dz_a\wedge\mathrm d\bar z_a,\\
 \mathrm d\mathsf{Vol}_{\mathbb C^m}(v)
 &=\Bigl(\frac{\sqrt{-1}}2\Bigr)^m
 \bigwedge_a\mathrm dv_a\wedge\mathrm d\bar v_a.
\end{aligned}
\]

\smallskip\noindent\emph{Matrix coefficients and complementary index sets.}
Since \(M_t(v)=\bigwedge^t(I-vv^*)\), its matrix in the standard
exterior basis consists of the \(t\times t\) minors of \(I-vv^*\):
\[
 (M_t)_{D,C}
 =\det[(\delta_{dc}-v_d\bar v_c)_{d\in D,c\in C}].
\]
Expanding the two exterior powers and moving the
\(\mathrm dv,\mathrm d\bar v\) factors past the
\(\mathrm dz,\mathrm d\bar z\) factors gives
\begin{align*}
(\mathsf E-\overline{\mathsf Z}\wedge\mathsf V)^{\wedge t}
&=(-1)^{t(t-1)/2}t!
\sum_{|C|=|E|=t}(M_t)_{C,E}\,
\mathrm d\bar z_C\wedge\mathrm dv_E,
\\
(\mathsf F-\mathsf Z\wedge\overline{\mathsf V})^{\wedge t}
&=(-1)^{t(t-1)/2}t!
\sum_{|D|=|F|=t}(M_t)_{F,D}\,
\mathrm dz_D\wedge\mathrm d\bar v_F.
\end{align*}
The coefficients in the second line are the minors of the transpose
\((I-vv^*)^{\mathsf T}\).  The constant-coefficient forms in~\eqref{eq:frozen-smooth-coupling-forms} are
\[
\begin{aligned}
\alpha_{l,z_0}^{(z)}
&=(\sqrt{-1})^{-d_l}\sum_{|D|=|C|=t}
\sigma(D)\sigma(C)B_{D,C}\,
\mathrm dz_{D^c}\wedge\mathrm d\bar z_{C^c},\\
\alpha_{l,z_0}^{(v)}
&=(\sqrt{-1})^{-d_l}\sum_{|E|=|F|=t}
\sigma(E)\sigma(F)B_{E,F}\,
\mathrm dv_{E^c}\wedge\mathrm d\bar v_{F^c}.
\end{aligned}
\]
In the exterior product, a term is nonzero only when its index sets in
each coordinate block are complementary.  For such terms,
\begin{align*}
&(-1)^t(\sqrt{-1})^{-2d_l}
\sigma(D)\sigma(C)\sigma(E)\sigma(F)
\\
&\quad{}\times
\mathrm dz_{D^c}\wedge\mathrm d\bar z_{C^c}
\wedge\mathrm d\bar z_C\wedge\mathrm dv_E
\wedge\mathrm dz_D\wedge\mathrm d\bar v_F
\wedge\mathrm dv_{E^c}\wedge\mathrm d\bar v_{F^c}
\\
&\qquad=2^m\Omega_0\wedge\mathrm d\mathsf{Vol}_{\mathbb C^m}(v).
\end{align*}
The four \(\sigma\) factors account for the complementary index
orderings.  Moving \(\mathrm dz_D\) past
\(\mathrm d\bar z_{C^c},\mathrm d\bar z_C,\mathrm dv_E\) contributes
\(t(m+t)\) transpositions.  Moving \(\mathrm dv_{E^c}\) past
\(\mathrm dv_E,\mathrm d\bar v_F\) and \(\mathrm d\bar v_{F^c}\) past
\(\mathrm d\bar v_F\) contributes \(3td_l\).  Their sum is the even
number \(2t(m+d_l)\).  Interleaving the ordered holomorphic and
antiholomorphic differentials in each variable adds an even total of
\(2\binom m2\) transpositions.  The factor \((-1)^t\) comes from the
\(t\)-th exterior power of the model kernel in~\eqref{eq:whole-form}.
Together with
\((\sqrt{-1})^{-2d_l}\), it gives the phase \((-1)^m\) of the product
volume form, and \(2^m\) converts the volume normalizations.
Summing over \(D,C,E,F\) gives
\begin{align}
&(-1)^t\alpha_{l,z_0}^{(z)}\wedge
(\mathsf E-\overline{\mathsf Z}\wedge\mathsf V)^{\wedge t}
\wedge(\mathsf F-\mathsf Z\wedge\overline{\mathsf V})^{\wedge t}
\wedge\alpha_{l,z_0}^{(v)}
\nonumber\\
&\qquad=2^m(t!)^2
\sum_{D,C,E,F}B_{D,C}(M_t)_{C,E}B_{E,F}(M_t)_{F,D}\,
\Omega_0\wedge\mathrm d\mathsf{Vol}_{\mathbb C^m}(v)
\nonumber\\
&\qquad=2^m(t!)^2\operatorname{Tr}(BM_t(v)BM_t(v))\,
\Omega_0\wedge\mathrm d\mathsf{Vol}_{\mathbb C^m}(v).
\label{eq:compound-matrix-contraction}
\end{align}
Each determinant expansion contributes one \(t!\).

\smallskip\noindent\emph{The equation \(\omega_0\wedge\alpha_{l,z_0}^{(z)}=0\).}
For an \(r\)-element subset \(E\subset\{1,\ldots,m\}\), let
\(\sigma_r(E)\) be the sign satisfying
\(\mathrm dz_{E^c}\wedge\mathrm dz_E
=\sigma_r(E)\mathrm dz_1\wedge\cdots\wedge\mathrm dz_m\);
thus \(\sigma=\sigma_t\).  Counting inversions gives
\[
\sigma_r(E)
=(-1)^{\sum_{e\in E}(m-e)-r(r-1)/2}.
\]
For \((t-1)\)-subsets \(D,C\), set
\(a_D(b)=|\{d\in D:d<b\}|\).  The matrix
\(T=\sum_b\iota_bB\varepsilon_b\) has entries
\begin{equation*}
T_{D,C}=\sum_{b\notin D\cup C}
(-1)^{a_D(b)+a_C(b)}
B_{D\cup\{b\},C\cup\{b\}}.
\end{equation*}
The coefficient \(\chi_{D^c,C^c}\) of
\(\mathrm dz_{D^c}\wedge\mathrm d\bar z_{C^c}\) in
\(\omega_0\wedge\alpha_{l,z_0}^{(z)}\) is
\begin{equation*}
\chi_{D^c,C^c}
=\frac{\sqrt{-1}}2(-1)^{d_l}
\sum_{b\notin D\cup C}(-1)^{a_D(b)+a_C(b)}
\psi^l_{D^c\setminus\{b\},C^c\setminus\{b\}}(z_0).
\end{equation*}
Each \(b\notin D\cup C\) occurs once.  The formula for
\(\sigma_r\) gives
\begin{equation*}
\sigma_t(D\cup\{b\})
=(-1)^{m-b-t+1}\sigma_{t-1}(D).
\end{equation*}
Applying this sign identity to both \(D\) and \(C\) cancels the two
factors \((-1)^{m-b-t+1}\).  Let \(B'\)
have entries
\[
B'_{D,C}=(\sqrt{-1})^{d_l+1}\sigma_{t-1}(D)\sigma_{t-1}(C)
\chi_{D^c,C^c}.
\]
The two matrices are related by
\begin{equation*}
T=2(-1)^{d_l+1}B'.
\end{equation*}
The factor \(2\) comes from \(\omega_0\); the sign includes both
holomorphic and antiholomorphic reorderings.  Hence
\(\omega_0\wedge\alpha_{l,z_0}^{(z)}=0\) is equivalent to
\(\operatorname{ptr}(B)=T=0\).

\section{Extension to piecewise \texorpdfstring{\(\mathscr C^2\)}{C2} boundaries}
\label{app:piecewise-boundary-extension}

We verify the geometric assumptions and prove the extension stated in
Remark~\ref{rmk:piecewise-boundary-extension}.  Condition~\textup{(i)}
gives the Stokes identities.  In~\textup{(ii)}, the area bound in balls
controls the kernel integrals and the bound on \(\sigma_{\partial U}(S_\delta)\)
controls the contribution from the strip near the singular set.
Condition~\textup{(iii)} gives the uniform scaled boundary calculation
away from that strip.

For the polyhedral local models of
\cite[Section~1 and the proof of Theorem~1.4]{MR2465693}, use the
correspondence of both \(U\) and \(\partial U\) with the cone interior
and boundary specified in Remark~\ref{rmk:piecewise-boundary-extension}.
Take \(\Sigma\) to be the points
at which the boundary is not locally a \(\mathscr C^2\) hypersurface,
coalescing coplanar adjacent facet pieces in each polyhedral model.
The regular boundary may be taken as \(F_1\), since the faces need
not be connected; \(\Sigma\) lies locally in the images of faces of
dimension at most \(2m-2\).  Choose relatively compact subcharts in a
finite model cover, so their maps and inverses have common
\(\mathscr C^2\) and bi-Lipschitz bounds.

For \textup{(i)}, take a convex polyhedral subdivision.
Stokes on the full-dimensional cells
cancels internal facets, and a partition of unity transports the
formula to \(U\).

For \textup{(ii)}, in bounded model charts each facet has
\(O(\rho^{2m-1})\) area in a \(\rho\)-ball and \(O(\delta)\) area
within distance \(\delta\) of the lower-dimensional faces, by
integration in linear facet coordinates.  The finite cover and metric
comparisons give the required bounds.

For \textup{(iii)}, in models with singular boundary faces, the facet-separation estimate
in the cited proof bounds the distance from a regular point to the boundary
outside its local smooth face below by a fixed positive multiple of
its distance to \(\Sigma\); in smooth models, use a fixed chart radius.
After transport, the distance condition
in \textup{(iii)} isolates a single face; uniform
\(\mathscr C^2\) graph bounds and Taylor's theorem give the stated
estimates.

The boundary has zero ambient volume, so the domain-test
projection argument still applies.  Conditions~\textup{(i)}--\textup{(ii)}
give the Stokes identities and kernel integral estimates used for
the diagram and chaos-tail bounds.  For fixed \(L,\delta>0\),
condition~\textup{(iii)} gives the same scaled boundary calculation
uniformly on \(\partial U\setminus S_\delta\) for large \(N\).
The contribution from \(S_\delta\times\partial U\) to
\eqref{eq:weighted-sector-boundary-pair} is
\(O(N^{2k-m-1/2}\delta)\): the integrand is
\(O(N^{2k-1}\mathfrak h_N)\), integration in the second boundary
variable contributes \(N^{-m+1/2}\), and
\(\sigma_{\partial U}(S_\delta)=O(\delta)\).
The omitted model integral is also \(O(\delta)\).
Letting \(N\to\infty\), then \(L\to\infty\), and finally
\(\delta\downarrow0\) proves~\eqref{eq:weighted-sector-positive-limit}
and hence the domain criterion.

The complete weighted-domain asymptotic also extends.  For every fixed
cutoff \(q\), \eqref{eq:full-cutoff-boundary-bound} and the surface row
bound make the contribution from \(S_\delta\times\partial U\) to its
boundary integral \(O_q(N^{2k-m-1/2}\delta)\).
On \(\partial U\setminus S_\delta\), condition~\textup{(iii)} gives the
same local limit as in
Proposition~\ref{prop:weighted-domain-variance-asymptotics}.
Taking \(N\to\infty\), then the scaled radius to infinity, and then
\(\delta\downarrow0\) proves~\eqref{eq:weighted-cutoff-variance-limit}
on these faces.  The geometric assumptions also give the uniform
tail estimate~\eqref{eq:weighted-variance-cutoff-error}.
The constants \(c_{m,k,q}\) are those of
Proposition~\ref{prop:weighted-domain-variance-asymptotics} and
therefore converge to \(\nu_{m,k}\).  Removing the cutoff proves
the complete weighted-domain variance asymptotic.

Away from \(\mathcal P\), the expected intersection trace measures
have smooth densities.  Since \(\partial U\) has zero ambient volume,
the random intersection trace measures almost surely assign zero mass to
the boundary on the supports used in the remote comparison.
Condition~\textup{(i)} supplies the Stokes identities there, and
condition~\textup{(ii)} supplies the area bounds for the kernel
integral estimates in Lemma~\ref{lem:finite-conditioning-remote}.
The local intersection-mass estimates in
Section~\ref{sec:finite-conditioning} use full balls and also cover
prescribed points of \(\Sigma\).  The \(L^2\) stability estimates in
Proposition~\ref{prop:finite-conditioning-stability} thus persist, with
\(\chi_U=0\) on \(\partial U\).  Dividing the domain coupling bound by the
unconditioned standard deviation, of order \(N^{k-m/2-1/4}\), gives
the same relative error as in Section~\ref{sec:comparison-coupled}.
The proof there yields the domain conclusions of
Proposition~\ref{prop:quantitative-conditional-comparison} and
Main Theorems~\ref{thm:finite-point-conditioned-variance}
and~\ref{thm:finite-point-conditioned-clt}.

\paragraph{Acknowledgments.}
I thank my doctoral advisor, Professor Song-Yan Xie, for his support, encouragement, and many suggestions that improved the manuscript. I also thank Professor Sébastien Boucksom for hosting my visit to IMJ-PRG at Sorbonne Université, for reading the manuscript carefully, and for many helpful comments and discussions. I am grateful to Professors Xiaonan Ma and Guokuan Shao for pointing out missing references, and to Professor Hao Wu (Nanjing University) for drawing my attention to the higher-codimension fluctuation problem.
I also thank the anonymous referees for their time and consideration.

\paragraph{Funding.}
This work was supported by the National Natural Science Foundation of
China under Grant No. 12471081 and the CAS Project for Young Scientists
in Basic Research under Grant No. YSBR-033.

\paragraph{Use of generative AI.}
An earlier version of this work,
\href{https://arxiv.org/abs/2603.04588v3}{arXiv:2603.04588v3}, was
prepared without generative AI.  It already contained the main ideas
and core proof strategy for the central limit theorems for unconditioned
common-zero statistics, including the Wiener-chaos expansion of zero
currents, diagrammatic moment estimates, and the use of selected chaos
components for variance lower bounds.  For the present revision, the
author used OpenAI's generative AI tools to help reorganize and shorten
the exposition.  In particular, AI assistance helped complete and
correct the Gaussian integral calculation establishing strict
positivity of the leading asymptotic coefficient of the variance of a
selected joint Wiener-chaos component, which supplies the variance
lower bound for smooth statistics.  The tools also assisted in
formulating the Bergman-kernel criterion from the estimates and
computations underlying the original proof, and with the proofs of the
weighted-domain variance asymptotics and finite-point conditioning
results added in this version.  The author also used an AI agent through
DeepSeek Harness to help revise the exposition and provide preliminary
feedback on selected arguments.  The author takes full responsibility for the
mathematical content and final text.

\paragraph{Data availability.}
No datasets were generated or analysed during the current study.

\end{document}